\documentclass{article}

\usepackage{microtype}
\usepackage{graphicx}
\usepackage{subfigure}
\usepackage{booktabs} 
\usepackage[accepted]{icml2018}
\usepackage{hyperref}
\usepackage{amsfonts}
\usepackage{longtable}
\usepackage{amsfonts,amsmath,amssymb,graphicx,color}

\usepackage{psfrag}
\usepackage{epstopdf}

\usepackage{amsthm}
\usepackage{psfrag}
\usepackage{pdflscape}
\usepackage{multirow}
\usepackage{blindtext} 
\usepackage{enumitem}
\hypersetup{colorlinks=true}
\usepackage{framed}
\usepackage{footnote}
\usepackage[dvipsnames]{xcolor}
\usepackage{soul}
\usepackage{multirow}
\usepackage{grffile}
\usepackage{nicefrac}

\usepackage{chngcntr}

\newtheorem{thm}{Theorem}[section]

\newtheorem{lem}[thm]{Lemma}

\newtheorem{assum}[thm]{Assumptions}

\counterwithout{equation}{section}

\DeclareMathOperator{\Var}{Var}

\newcommand{\defeq}{\stackrel{\triangle}{=}}

\newcommand{\R}{{\mathbb R}}        

\def\noprint#1{}

\def\E{\mathbb{E}}

\begin{document}
	\twocolumn[
	\icmltitle{A Progressive Batching L-BFGS Method for Machine Learning}
	\icmlsetsymbol{equal}{*}
	
	\begin{icmlauthorlist}
		\icmlauthor{Raghu Bollapragada}{nu}
		\icmlauthor{Dheevatsa Mudigere}{intel-india}
		\icmlauthor{Jorge Nocedal}{nu}
		\icmlauthor{Hao-Jun Michael Shi}{nu}
		\icmlauthor{Ping Tak Peter Tang}{intel-usa}
	\end{icmlauthorlist}
	
	\icmlaffiliation{nu}{Department of Industrial Engineering and Management Sciences, Northwestern University, Evanston, IL, USA}
	\icmlaffiliation{intel-india}{Intel Corporation, Bangalore, India}
	\icmlaffiliation{intel-usa}{Intel Corporation, Santa Clara, CA, USA}
	
	\icmlcorrespondingauthor{Raghu Bollapragada}{raghu.bollapragada@u.northwestern.edu}
	\icmlcorrespondingauthor{Dheevatsa Mudigere}{dheevatsa.mudigere@intel.com}
	\icmlcorrespondingauthor{Jorge Nocedal}{j-nocedal@northwestern.edu}
	\icmlcorrespondingauthor{Hao-Jun Michael Shi}{hjmshi@u.northwestern.edu}
	\icmlcorrespondingauthor{Ping Tak Peter Tang}{peter.tang@intel.com}
	
	\icmlkeywords{Deep Learning, Sample Selection, Nonconvex Optimization, L-BFGS, Distributed Training}
	
	\vskip 0.3in
	]
	
	\printAffiliationsAndNotice{}

\begin{abstract} 
	
The standard L-BFGS method relies on gradient approximations that are not dominated by noise, so that search directions are descent directions, the line search is reliable, and quasi-Newton updating yields useful quadratic models of the objective function. All of this appears to call for a full batch approach, but since small batch sizes give rise to faster algorithms with better generalization properties, L-BFGS is currently not considered an algorithm of choice for large-scale machine learning applications. One need not, however, choose between the two extremes represented by the full batch or highly stochastic regimes, and may instead follow a progressive batching approach in which the sample size increases during the course of the optimization. In this paper, we present a new version of the L-BFGS algorithm that combines three basic components --- progressive batching, a stochastic line search, and stable quasi-Newton updating ---  and that performs well on training logistic regression and deep neural networks. We provide supporting convergence theory for the method. 

\end{abstract}

\section{Introduction}

The L-BFGS method \cite{liu1989limited} has traditionally been regarded as a batch method in the 
machine learning community. This is because quasi-Newton algorithms need 
gradients of high quality in order to construct useful quadratic models and perform reliable line searches. 
These algorithmic ingredients can be implemented, it seems, only by using very large 
batch sizes, resulting in a costly iteration that makes the overall algorithm 
slow compared with stochastic gradient methods \cite{robbins1951stochastic}.


Even before the resurgence of  neural networks, many researchers observed that a well-tuned implementation of the 
stochastic gradient (SG) method was far more effective on large-scale logistic regression
applications than the batch L-BFGS method, even when taking into account the 
advantages of  parallelism offered by the use of large batches. The 
preeminence of the SG method (and its variants) became more pronounced with the advent of 
deep neural networks, and some researchers have speculated that SG is endowed with 
certain regularization properties that are essential in the minimization of 
such complex nonconvex functions \cite{hardt2015train,keskar2016large}.

In this paper, we postulate that the most efficient algorithms for machine 
learning may not reside entirely in the highly stochastic or full batch 
regimes, but should employ a progressive batching approach in which the sample size is
initially small, and is increased as the iteration progresses. This view is consistent with recent numerical experiments 
on training various deep neural networks \cite{smith2017don,goyal2017accurate}, where the SG method, with increasing sample sizes, yields similar test loss and 
accuracy as the standard (fixed mini-batch) SG method, while 
offering significantly greater opportunities for parallelism. 
 
Progressive batching algorithms have  received much attention recently from 
a theoretical perspective. It has been shown that they enjoy complexity 
bounds that rival those of the SG method \cite{byrd2012sample}, and that they can achieve a fast rate of convergence \cite{friedlander2012hybrid}.  The main appeal of these 
methods is that they inherit the efficient initial behavior of the SG method, offer greater opportunities to exploit parallelism,
and allow for the incorporation of second-order information. The latter can be done
efficiently via quasi-Newton updating. 

An integral part of quasi-Newton methods is the line search, which ensures that a convex quadratic model can be constructed at every iteration. One challenge that immediately arises is how to perform this line search  when  the objective function is stochastic. This is an issue that has not 
received sufficient attention in the literature, where stochastic line searches have been largely dismissed as inappropriate. 
In this paper, we take a step towards the development of stochastic line searches for machine learning by studying a key component, namely the initial estimate in the one-dimensional search. Our approach, which is based on statistical considerations, is designed for an Armijo-style backtracking line search.


\subsection{Literature Review}
Progressive batching (sometimes referred to as dynamic sampling)  has been well studied in the 
optimization literature, both for stochastic gradient and subsampled 
Newton-type methods
\cite{byrd2012sample,friedlander2012hybrid,cartis2015global,2014pasglyetal,roosta2016sub,roosta2016sub1,bollapragada2016exact,bollapragada2017adaptive,de2017automated}. 
Friedlander and Schmidt \yrcite{friedlander2012hybrid} introduced 
theoretical conditions under which a progressive batching SG method converges 
linearly for finite sum problems, and  experimented with a 
quasi-Newton adaptation of their algorithm. Byrd et al. \yrcite{byrd2012sample} proposed a
progressive batching strategy, based on a \emph{norm test}, that determines when to 
increase the sample size; they established linear convergence and computational complexity
bounds in the case when the batch size grows geometrically. More recently, Bollapragada et al. 
\yrcite{bollapragada2017adaptive} introduced a batch control mechanism based on an \emph{inner product} test
that improves upon the norm test mentioned above.

There has been a renewed interest in understanding the
generalization properties of small-batch and large-batch methods for training neural networks; 
see \cite{keskar2016large,dinh2017sharp,goyal2017accurate,hoffer2017train}. Keskar et al. \yrcite{keskar2016large} empirically observed that large-batch 
methods converge to solutions with inferior generalization properties; however, Goyal et al. \yrcite{goyal2017accurate} showed 
that large-batch methods can match the performance of small-batch methods when 
a warm-up strategy is used in conjunction with scaling the step length by the 
same factor as the batch size. Hoffer et al. \yrcite{hoffer2017train} and You et 
al. \yrcite{you2017scaling} also explored larger batch sizes and steplengths to 
reduce the number of updates necessary to train the network. All of these studies naturally led 
to an interest in progressive batching techniques. Smith et al. 
\yrcite{smith2017don} showed empirically that increasing the sample size and 
decaying the steplength are quantitatively equivalent for the SG method; 
hence, steplength schedules could be directly converted to batch size 
schedules. This approach was parallelized by Devarakonda et al. 
\yrcite{devarakonda2017adabatch}. De et al. \yrcite{de2017automated} presented numerical results with a progressive batching method that employs the norm test. Balles et al. \yrcite{balles2016coupling} proposed an adaptive dynamic sample size scheme and couples the sample size with the steplength.

Stochastic second-order methods have been explored within the context of 
convex and non-convex optimization; see
\cite{schraudolph2007stochastic,sohl2014fast,mokhtari2015global,berahas2016multi,byrd2016stochastic,keskar2016adaqn,curtis2016self,berahas2017robust,zhou2017stochastic}.
Schraudolph et al. 
\yrcite{schraudolph2007stochastic} ensured stability of quasi-Newton updating by computing
gradients using the same batch at the beginning and end of the iteration. Since this can
potentially double the cost of the iteration,  Berahas et al. \yrcite{berahas2016multi} proposed to achieve
gradient consistency by computing gradients based on the overlap between consecutive batches; this
 approach was further tested by Berahas and Takac \yrcite{berahas2017robust}. An interesting approach introduced by Martens and Grosse 
\yrcite{martens2015optimizing,grosse2016kronecker} approximates the Fisher information matrix to scale the
gradient; a distributed implementation of their K-FAC approach is described in \cite{ba2016distributed}. Another approach approximately computes the inverse Hessian by using the Neumann power series representation of matrices \cite{krishnan2017neumann}.

\subsection{Contributions}  

This paper builds upon three algorithmic components that have recently received  attention in the literature --- progressive batching, stable quasi-Newton updating, and  adaptive steplength selection. It advances their design and puts them together in a
novel algorithm with attractive theoretical and computational properties. 

The cornerstone of our progressive batching strategy is the mechanism proposed by Bollapragada et al. \yrcite{bollapragada2017adaptive} in the context of first-order methods. We extend their \emph{inner product} control test  to second-order algorithms, something that is delicate and leads to a significant modification of the original procedure. Another main contribution of the paper is the design of an Armijo-style backtracking line search where the initial steplength is chosen based on statistical information gathered during the course of the iteration. We show that this steplength procedure is effective on a wide range of applications, as it leads to well scaled steps and  allows for the BFGS update to be performed most of the time, even for nonconvex problems.  We also test two techniques for ensuring the stability of quasi-Newton updating, and observe that the overlapping procedure described by Berahas et al. \yrcite{berahas2016multi} is more efficient than a straightforward adaptation of classical quasi-Newton methods \cite{schraudolph2007stochastic}.
 
We report numerical tests on large-scale logistic regression and deep neural network training tasks that indicate that our method is robust and efficient, and has good generalization properties.  An additional advantage is that the method requires almost no parameter tuning, which is possible due to the incorporation of second-order information.   All of this suggests that our approach has the potential to become one of the leading optimization methods for training deep neural networks. In order to achieve this, the algorithm must be optimized for parallel execution, something that was only briefly explored in this study.

\section{A Progressive Batching Quasi-Newton Method}
\label{sec:progressive}

The problem of interest is
\begin{equation}\label{prob}
\min_{x \in \R^{d}}   F(x) = \int f(x; z,y) dP( z,y),
\end{equation}
where $f$ is the composition of a prediction function (parametrized by  $x $) 
and a loss function, and $(z,y)$ are random input-output pairs with 
probability distribution $P(z,y)$.  The associated empirical risk problem 
consists of minimizing
$$R(x) = \frac{1}{N} \sum_{i =1}^N f(x;z^i,y^i) \defeq \frac{1}{N} \sum_{i =1}^N F_i(x), $$
where we define 
$
F_i(x) = f(x;z^i,y^i).
$
A stochastic quasi-Newton method is given by 
\begin{equation} \label{iter}
x_{k+1} = x_k - \alpha_k H_k g_k^{S_k},
\end{equation}
where the batch  (or subsampled) gradient is given by
\begin{equation}   \label{batch}
g_k^{S_k}=\nabla F_{S_k}(x_k) \defeq \frac{1}{|S_k|}\sum_{i \in S_k}\nabla F_i(x_k),
\end{equation}
the set $S_k \subset \{1,2,\cdots\}$ indexes data points $(y^i, z^i)$ sampled from the distribution $P(z,y)$, and $H_k$ is a positive definite quasi-Newton matrix. 
We now discuss each of the components of the new method.

\subsection{Sample Size Selection}

The proposed algorithm has the form \eqref{iter}-\eqref{batch}. Initially, it utilizes a small batch size $|S_k|$, and increases it gradually in order to attain a fast local rate of convergence and permit the use of second-order information. A challenging question is to determine when, and by how much, to increase the batch size $|S_k|$ over the course of the optimization procedure based on observed gradients --- as opposed to using prescribed rules that depend on the iteration number $k$. 

We propose to build upon the strategy introduced by Bollapragada et al. \yrcite{bollapragada2017adaptive} in the context of first-order methods. Their \textit{inner product test} determines a sample size such that the search direction is a descent direction with high probability. A straightforward extension of this strategy to the quasi-Newton setting is not appropriate since requiring only that a stochastic quasi-Newton search direction be a descent direction with high probability would underutilize the curvature information contained in the search direction. 

We would like, instead, for the search direction $d_k = - H_k g_k^{S_k}$ to make an acute angle with the true quasi-Newton search direction $- H_k \nabla F(x_k)$, with high probability. Although this does not imply that $d_k$ is a descent direction for $F$, this will normally be the case for any reasonable quasi-Newton matrix. 

To derive the new inner product quasi-Newton (IPQN) test, we first observe that the stochastic quasi-Newton search direction makes an acute angle with the true quasi-Newton direction in expectation, i.e.,
\begin{equation} \label{IP-Expectation}
\E_k \left[(H_k \nabla F(x_k))^T (H_k g_k^{S_k}) \right] = \|H_k \nabla F(x_k)\|^2 ,
\end{equation}
where $\E_k$ denotes the conditional expectation at $x_k$.
We must, however, control the variance of this quantity to achieve our stated objective.
Specifically, we select the sample size $|S_k|$ such that the following condition is satisfied:
\begin{equation}\label{IPQN: exact}
\begin{aligned} 
&\E_k \left[\left((H_k \nabla F(x_k))^T (H_k g_k^{S_k}) - \|H_k \nabla F(x_k)\|^2\right)^2\right] \\
&~~~~~~~~~~~~~~\leq \theta^2   \|H_k \nabla F(x_k)\|^4,
\end{aligned}
\end{equation}
for some $\theta > 0$. The left hand side of \eqref{IPQN: exact} is difficult to compute but can be bounded by the true variance of individual search directions, i.e.,
\begin{equation}\label{IPgrad: exacttest}
\begin{aligned}
&\frac{\E_k \left[\left((H_k \nabla F(x_k))^T (H_k g_k^i) -  \|H_k \nabla F(x_k)\|^2\right)^2\right]}{|S_k|} \\
&~~~~~~~~~~~~~~\leq \theta^2  \|H_k \nabla F(x_k)\|^4,
\end{aligned} 
\end{equation}
where $g_k^i = \nabla F_i(x_k)$. This test involves the true expected gradient and variance, but we can approximate these quantities with sample gradient and variance estimates, respectively, yielding the practical inner product quasi-Newton test:
\begin{equation}\label{IPQN: test}
\frac{\Var_{i \in S_k^v}\left((g_k^i)^T H_k^2 g_k^{S_k}\right)}{|S_k|} \leq \theta^2  \left\|H_k  g_k^{S_k} \right\|^4 ,
\end{equation}
where $S_k^v \subseteq S_k$ is a subset of the current sample (batch), and the variance term  is defined as 
\begin{equation}\label{IPQN: variance}
\frac{\sum_{i \in S_k^v} \left((g_k^i)^T H_k^2 g_k^{S_k} - \left\|H_k g_k^{S_k} \right\|^2\right)^2}{|S_k^v| - 1}.
\end{equation}

The variance  \eqref{IPQN: variance} may be computed using just one additional Hessian vector product of $H_k$ with $H_k g_k^{S_k}$. Whenever condition \eqref{IPQN: test} is not satisfied, we increase the sample size $|S_k|$. In order to estimate the increase that would lead to a satisfaction of \eqref{IPQN: test}, we reason as follows.   If we assume that new sample $|\bar{S}_k|$ is such that
\begin{equation*}
\left\|H_k g_k^{S_k}\right\| \approxeq \left\|H_k g_k^{\bar{S}_k}\right\|,
\end{equation*}
and similarly for the variance estimate, then a simple computation shows that a lower bound on the new sample size is
\begin{equation}\label{IPQN: batch size}
|\bar{S}_k| \geq \frac{\Var_{i \in S_k^v}\left((g_k^i)^T H_k^2 g_k^{S_k}\right)}{\theta^2 \left\|H_k g_k^{S_k}\right\|^4} \defeq b_k.
\end{equation}
In our implementation of the algorithm, we set the new sample size as $|S_{k+1}| = \lceil b_k \rceil$.
When the sample approximation of $F(x_k)$ is not accurate, which can occur when $|S_k|$ is  small, the progressive batching mechanism just described may not be reliable. In this case we employ the moving window technique described in Section  4.2 of Bollapragada et al. \yrcite{bollapragada2017adaptive}, to produce a sample estimate of $\nabla F(x_k)$.

\subsection{The Line Search}
In deterministic optimization, line searches are employed to ensure that the step is not too short and to guarantee sufficient decrease in the objective function.
%
 Line searches are particularly important in quasi-Newton methods since they  ensure robustness and efficiency of the iteration with little additional  cost. 

In contrast, stochastic line searches are poorly understood and rarely employed in practice because they must make decisions based on sample function values 
\begin{equation}
F_{S_k}(x) = \frac{1}{|S_k|} \sum_{i \in S_k} F_i(x),
\end{equation} 
which are noisy approximations to the true objective $F$. One of the key questions in the design of a stochastic line search is how to ensure, with high probability, that there is a decrease in the true function when one can only observe stochastic approximations $F_{S_k}(x)$. We address this question by proposing a formula for the step size $\alpha_k$ that controls possible increases in the true function.  Specifically, the first trial steplength in the stochastic backtracking line search is computed  so that the predicted decrease in the expected function value is sufficiently large, as we now explain. 

Using  Lipschitz continuity of $\nabla F(x)$ and taking  conditional expectation, we can show the following inequality
\begin{equation}
\E_k \left[ F_{k+1} \right] \leq F_k - \alpha_k \nabla F(x_k)^T H_k^{1/2} W_k H_k^{1/2} \nabla F(x_k) \label{sequoia}
\end{equation}
where 
$$W_k = \left(I - \frac{L\alpha_k}{2}\left(1 +  \frac{{\rm Var}\{H_k g_k^{i}\}}{|S_k| \|H_k \nabla F(x_k)\|^2}\right) H_k \right),$$
$\Var\{H_k g_k^{i} \} = \E_k\left[\|H_k g_k^{i} - H_k \nabla F(x_k)\|^2\right]$, $F_k= F(x_k)$,  and $L$ is the Lipschitz constant. The proof of \eqref{sequoia} is given in the supplement.

The only difference in \eqref{sequoia} between the deterministic and stochastic quasi-Newton methods is the additional variance term in the matrix $W_k$. To obtain decrease in the function value in the deterministic case, the matrix $\left(I - \frac{L\alpha_k}{2}H_k\right)$ must be positive definite, whereas in the stochastic case the matrix $W_k$ must be positive definite to yield a decrease in $F$ \emph{in expectation}. 
In the deterministic case, for a reasonably good quasi-Newton matrix $H_k$, one expects that  $\alpha_k=1$  will result in a decrease in the function, and therefore the initial trial steplength parameter should be chosen to be 1.
In the stochastic case, the initial trial value
\begin{equation}\label{initialstep}
\hat \alpha_k = \left( 1 + \frac{\Var\{H_k g_k^{i}\}}{|S_k| \|H_k \nabla F(x_k)\|^2} \right)^{-1} 
\end{equation}
will result in decrease in the expected function value. 
However, since formula \eqref{initialstep} involves the expensive computation of the individual matrix-vector products $H_k g_k^i$, we approximate the variance-bias ratio as follows:
\begin{equation} \label{initialstep-practical}
\bar \alpha_k = \left( 1 + \frac{\Var\{g_k^{i}\}}{|S_k| \|\nabla F(x_k)\|^2} \right)^{-1},
\end{equation} 
where $\Var\{g_k^{i}\} = \E_k\left[\|g_k^{i} - \nabla F(x_k)\|^2\right]$. In our practical implementation, we estimate the population variance and gradient with the sample variance and gradient, respectively, yielding the initial steplength
\begin{equation}\label{eq: initial step length}
\alpha_k = \left( 1 + \frac{\Var_{i \in S_k^v}\{ 
g_k^{i}\}}{|S_k|\left\|g_k^{S_k}\right\|^2} \right)^{-1} ,
\end{equation}
where 
\begin{equation}
\Var_{i \in S_k^v}\{g_k^{i}\} = \frac{1}{|S_k^v| - 1}\sum_{i \in S_k^v}\left\|g_k^{i} - g_k^{S_k}\right\|^2
\end{equation}
and $S_k^v \subseteq S_k$.
With this initial value of $\alpha_k$ in hand, our algorithm performs a backtracking line search that aims to satisfy the Armijo condition
\begin{equation}\label{eq: stochastic line search}
\begin{aligned}
& F_{S_k} (x_k - \alpha_k H_k g_k^{S_k}) \\
& ~~~~~~~ \leq F_{S_k} (x_k) - c_1 \alpha_k (g_k^{S_k})^T H_k g_k^{S_k},
\end{aligned}
\end{equation}
where $c_1 >0$. 
\subsection{Stable Quasi-Newton Updates}

In the BFGS and L-BFGS methods, the inverse Hessian approximation is updated using the formula
\begin{equation}
\begin{aligned}
H_{k+1} & = V_k^TH_kV_k + \rho_ks_ks_k^T \\ 
\rho_k & = (y_k^Ts_k)^{-1}\\ 
V_k & = I - \rho_ky_ks_k^T
\end{aligned}
\end{equation}
where $s_k = x_{k+1} - x_k$ and $y_k$ is  the difference in the gradients at $x_{k+1}$ and $x_k$. When the batch changes from one iteration to the next ($S_{k+1} \neq S_k$), it is not obvious how $y_k$ should be defined. It has been observed that when  $y_k$ is computed using different samples, the updating process may be unstable, and hence it seems natural to use the same sample at the beginning and at the end of the iteration \cite{schraudolph2007stochastic}, and define  
\begin{equation}\label{full-overlap}
y_k = g_{k + 1}^{S_k} - g_k^{S_k}.
\end{equation}
However, this requires that the gradient be evaluated twice for every batch $S_k$ at $x_k$ and $x_{k+1}$. To avoid this additional cost,  Berahas et al. \yrcite{berahas2016multi} propose to use the overlap between consecutive samples in the gradient differencing. If we denote this overlap as $O_k = S_k \cap S_{k + 1}$, then one defines
\begin{equation}\label{multi-batch}
y_k = g_{k + 1}^{O_k} - g_k^{O_k}.
\end{equation}
This requires no extra computation since the two gradients in this expression are subsets of the gradients corresponding to the samples $S_k$ and $S_{k+1}$. The overlap should not be too small to avoid differencing noise, but this is easily achieved in practice.  We test both formulas for $y_k$ in our implementation of the method; see Section \ref{sec:parallel}.

\subsection{The Complete Algorithm}
\label{sec:algorithm}
The  pseudocode of the progressive batching L-BFGS method is given in Algorithm~\ref{alg: pb_L-BFGS}. Observe that the limited memory Hessian approximation $H_k$ in Line 8 is independent of the choice of the sample $S_k$. Specifically,  $H_k$ is defined  by a collection of curvature pairs $\{(s_j, y_j)\}$, where the most recent pair is based on the sample $S_{k-1}$; see Line 14.
For the batch size control test \eqref{IPQN: test}, we choose $\theta =0.9$ in the logistic regression experiments, and $\theta$ is a tunable parameter chosen in the interval $[0.9, 3]$ in the neural network experiments. The constant $c_1$ in \eqref{eq: stochastic line search} is set to $c_1 = 10^{-4}$.
%
%
%
For L-BFGS, we set the memory as $m = 10$. We skip the quasi-Newton update if the following curvature condition is not satisfied:
\begin{equation}
y_k^T s_k > \epsilon \|s_k\|^2, \quad\mbox{with} \ \epsilon=10^{-2}. 
\end{equation}
The initial Hessian matrix $H_0^k$ in the L-BFGS recursion at each iteration is chosen as $\gamma_k I$ where $\gamma_k = {y_k^Ts_k}/{y_k^Ty_k}$.  

\begin{algorithm}[h!]
	\caption{Progressive Batching L-BFGS Method }
	\label{alg: pb_L-BFGS}
	\textbf{Input:} Initial iterate $x_0$, initial sample size $|S_0|$;\\
	\textbf{Initialization:} Set $k \leftarrow 0$ 
	
	\textbf{Repeat} until convergence:
	\begin{algorithmic}[1]
		\STATE Sample $S_k \subseteq \{1,\cdots,N\}$ with sample size 
		$|S_k|$
		\IF{condition \eqref{IPQN: test} is not satisfied} 
		\STATE Compute $b_k$ using \eqref{IPQN: batch size}, and set $\hat b_k \leftarrow \lceil b_k \rceil - |S_k|$
		\STATE Sample $S^+ \subseteq \{1,\cdots,N\} \setminus S_k$  with $|S^+|= \hat b_k$
		\STATE Set $S_k \leftarrow S_k \cup S^+$
		\ENDIF
		\STATE Compute $g_k^{S_k}$
		\STATE Compute $p_k = -H_k g_k^{S_k}$ using L-BFGS 
		Two-Loop Recursion in \cite{mybook}
		\STATE Compute $\alpha_k$ using \eqref{eq: initial step length}
		\WHILE {the Armijo condition \eqref{eq: stochastic line search} not satisfied} 
		\STATE Set $\alpha_k = \alpha_k / 2$ 
		\ENDWHILE
		\STATE Compute $x_{k+1} = x_k + \alpha_k p_k$ 
		
		\STATE Compute $y_k$ using 
		(\ref{full-overlap}) or (\ref{multi-batch}) 
		\STATE Compute $s_k = x_{k+1} - x_k$
		
		\IF{$y_k^T s_k > \epsilon \|s_k\|^2$} 
		\IF{number of stored $(y_j, s_j)$ exceeds $m$}
		\STATE Discard oldest curvature pair $(y_j, s_j)$
		\ENDIF
		\STATE Store new curvature pair $(y_k, s_k)$
		\ENDIF	
		\STATE Set $k \leftarrow k + 1$ 
		\STATE Set $|S_k| = |S_{k-1}|$
	\end{algorithmic}
\end{algorithm}

\section{Convergence Analysis}
\label{sec:convergence}

We now present convergence results for the proposed algorithm, both for strongly convex  and nonconvex objective functions. Our emphasis is in analyzing the effect of progressive sampling, and therefore, we follow common practice and assume that the steplength in the algorithm is fixed ($\alpha_k= \alpha$), and that the inverse L-BFGS matrix $H_k$ has bounded eigenvalues, i.e.,
\begin{equation} \label{assum: eigs}
\Lambda_1 I \preceq H_k \preceq \Lambda_2 I.
\end{equation}
This assumption can be justified both in the convex and nonconvex cases under certain conditions; see \cite{berahas2016multi}. We assume that the sample size is controlled by the exact inner product quasi-Newton test \eqref{IPgrad: exacttest}. This test is designed for efficiency, and in rare situations could allow for the generation of arbitrarily long search directions.  To prevent this from happening, we introduce an additional control on the sample size $|S_k|$, by extending (to the quasi-Newton setting) the orthogonality test introduced in \cite{bollapragada2017adaptive}. This additional requirement states that the current sample size $|S_k|$ is acceptable only if 
\begin{align} \label{orth-i}
&\frac{\E_k \left[\left\|H_kg_k^{i} - \frac{\left(H_kg_k^{S_k}\right)^T (H_k\nabla F(x_k))}{\|H_k\nabla F(x_k)\|^2}H_k\nabla F(x_k)\right\|^2\right]}{|S_k|} \nonumber \\ &~~~~~~~~~~~~~~\leq \nu^2 \|H_k\nabla F(x_k)\|^2,
\end{align} 
for some given $\nu >0$.

We now establish linear convergence  when the objective is strongly convex. 
\begin{thm} \label{thmlin} 
Suppose that $F$ is twice continuously differentiable and that there exist constants $0 < \mu \leq L$ such that 
\begin{equation}
\mu I   \preceq \nabla^2 F(x) \preceq L I, \quad \forall x \in \R^d.
\end{equation} 
Let $\{x_k\}$ be generated by iteration \eqref{iter}, for any $x_0$, 
where $|S_{k}|$ is chosen by the (exact variance) inner product quasi-Newton test \eqref{IPgrad: exacttest}. Suppose  that the orthogonality condition \eqref{orth-i}  holds at every iteration, and that the matrices $H_k$ satisfy \eqref{assum: eigs}.  Then, if
\begin{equation}   \label{c-stepform}
\alpha_k = \alpha \leq \frac{1}{(1 + \theta^2+\nu^2)L\Lambda_2},
\end{equation}
we have that   
\begin{equation} \label{linear}
\E[F(x_k) - F(x^*)] \leq \rho^k (F(x_0) - F(x^*)),
\end{equation}
where  $x^*$ denotes the  minimizer of $F$, 
$
\rho= 1 -\mu \Lambda_1\alpha,
$
and $\E$ denotes the total expectation.

\end{thm}
The proof of this result is given in the supplement. 
We now consider the case when $F$ is nonconvex and bounded below.

\begin{thm} \label{thmsublin}
Suppose that $F$ is twice continuously 
differentiable and bounded below, and that there exists a constant $ L > 0$ such that 
\begin{equation}
\nabla^2 F(x) \preceq L I, \quad \forall x \in \R^d.
\end{equation} 
Let $\{x_k\}$ be generated by iteration \eqref{iter}, for any $x_0$, 
where $|S_{k}|$ is chosen so that 
\eqref{IPgrad: exacttest} and  \eqref{orth-i} are satisfied, and suppose that \eqref{assum: eigs} holds. Then, if $\alpha_k$ satisfies
 \eqref{c-stepform},
%
we have
\begin{equation}   \label{convergence}
\lim_{k \rightarrow \infty} \mathbb{E} [\|\nabla F(x_k)\|^2] \rightarrow 0.
\end{equation}  
Moreover, for any positive integer $T$ we have that 
\begin{align*}
\min_{0\leq k \leq T-1} \E [\|\nabla F(x_k)\|^2] &\leq \frac{2}{\alpha 
T\Lambda_1}  (F(x_0) - F_{min}),
\end{align*}
where $F_{min}$ is a lower bound on  $F$ in $\mathbb{R}^d$.
\end{thm}
The proof is given in the supplement. This result shows that the sequence of gradients $\{ \| \nabla F(x_k) \| \}$ converges to zero in expectation, and establishes a global sublinear rate of convergence of the smallest gradients generated after every $T$ steps. 

\section{Numerical Results}
\label{sec:parallel}

In this section, we present numerical results for the proposed algorithm, which we refer to as PBQN for the \textbf{P}rogressive \textbf{B}atching \textbf{Q}uasi-\textbf{N}ewton algorithm.

	\subsection{Experiments on Logistic Regression Problems}

We first test our algorithm on binary classification problems where the objective function is given by the logistic loss with $\ell_2$ regularization:
\begin{equation} \label{logistic-loss}
R(x)=\frac{1}{N} \sum_{i=1}^{N}\log(1 + \exp(-z^ix^Ty^i)) + \frac{\lambda}{2}\|x\|^2, 
\end{equation} 
with $\lambda = {1}/{N}$. We consider the $8$ datasets listed in the supplement. An approximation $R^*$ of the optimal function value is computed for each problem by running the full batch L-BFGS method until $\|\nabla R(x_k)\|_{\infty} \leq 10^{-8}$. Training error is defined as $R(x_k) - R^*$, where $R(x_k)$ is evaluated over the training set; test loss is evaluated over the test set without the $\ell_2$ regularization term.


 We tested two options for computing the curvature vector $y_k$ in the PBQN method: the multi-batch (MB) approach \eqref{multi-batch} with 25\% sample overlap, and  the full overlap (FO) approach \eqref{full-overlap}. We set $\theta = 0.9$ in \eqref{IPQN: test}, chose $|S_0| = 512$, and set all other parameters to the default values given in Section~\ref{sec:progressive}. Thus, none of the parameters in our PBQN method were tuned for each individual dataset.  We compared our algorithm against two other methods: (i) Stochastic gradient (SG) with a batch size of 1; (ii) SVRG \cite{johnson2013accelerating} with the inner loop length set to $N$. The steplength for SG and SVRG is constant and tuned for each problem ($\alpha_k \equiv \alpha = 2^j$, for $j \in \{-10, -9, ..., 9, 10\}$) so as to give best performance.

In Figures \ref{exp:spam} and \ref{exp:covtype} we present results for two datasets, {\tt spam} and {\tt covertype}; the rest of the results are given in the supplement. The horizontal axis measures the number of full gradient evaluations, or equivalently, the number of times that $N$ component gradients $\nabla F_i$ were evaluated. 
The left-most figure reports the long term trend over 100 gradient evaluations, while the rest of the figures zoom into the first 10 gradient evaluations to show the initial behavior of the methods. The vertical axis measures training error, test loss, and test accuracy, respectively, from left to right.
\begin{figure*}[!htp]
	\begin{centering}
		\includegraphics[width=0.24\linewidth]{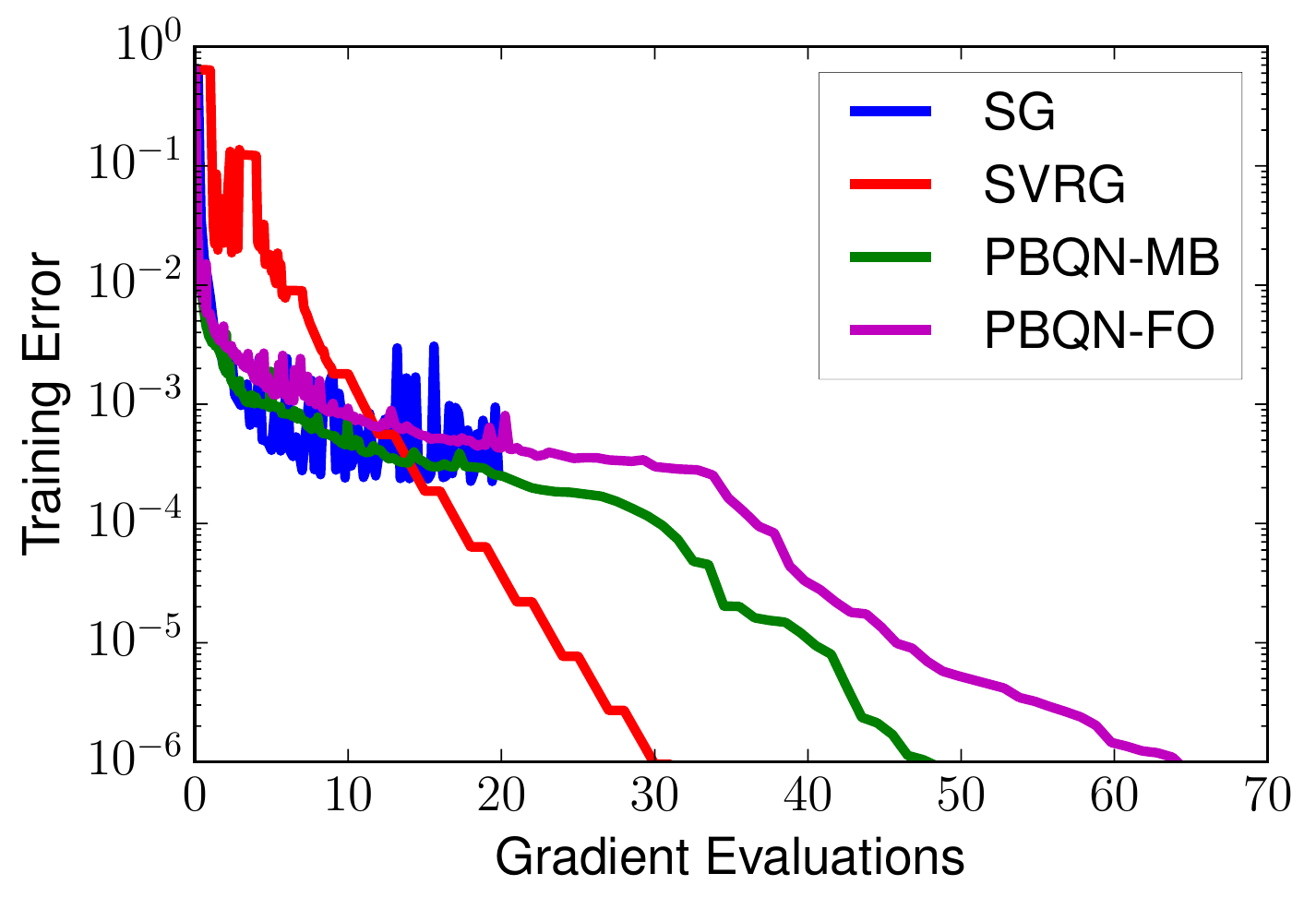}
		\includegraphics[width=0.24\linewidth]{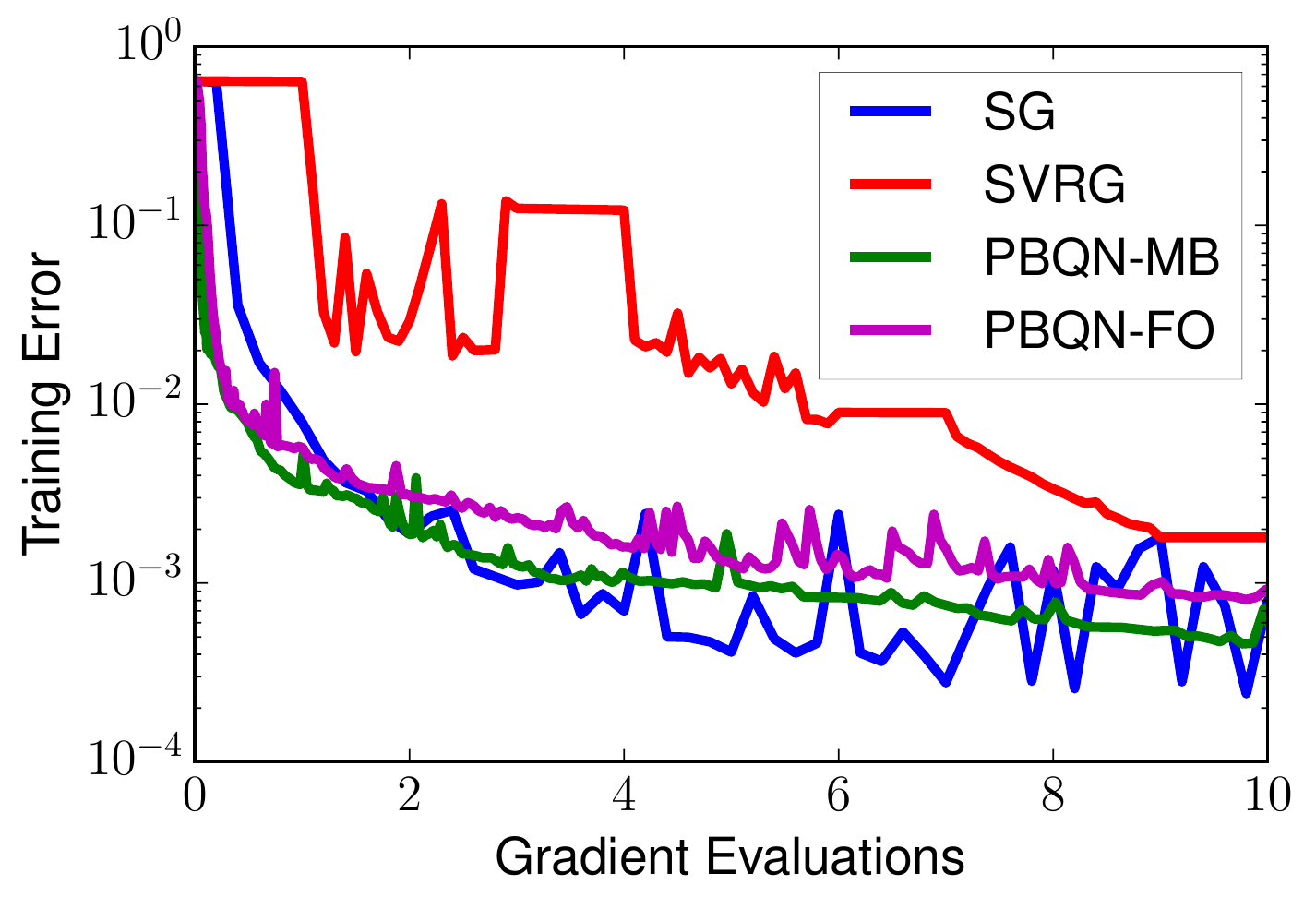}
		\includegraphics[width=0.24\linewidth]{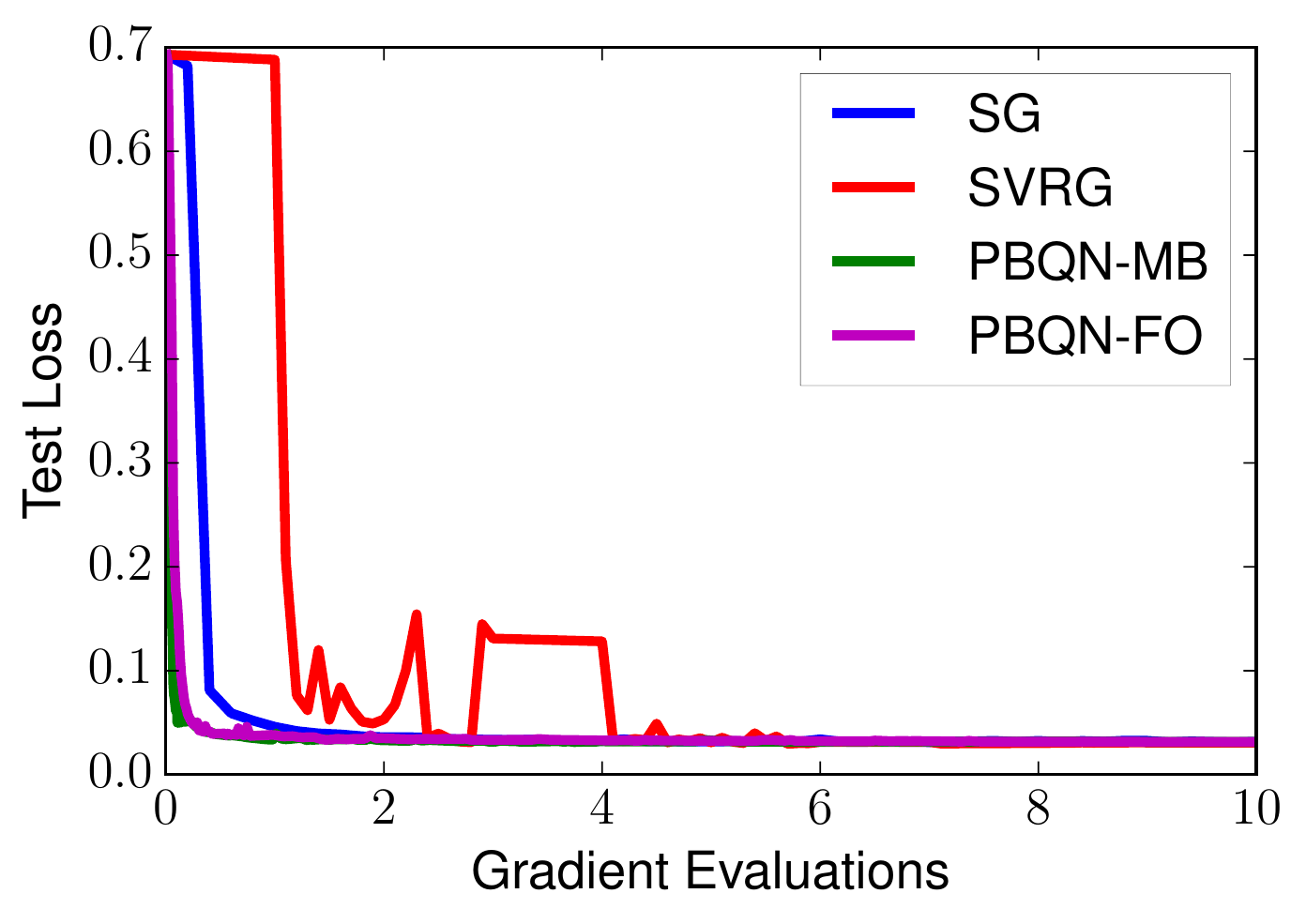}
		\includegraphics[width=0.24\linewidth]{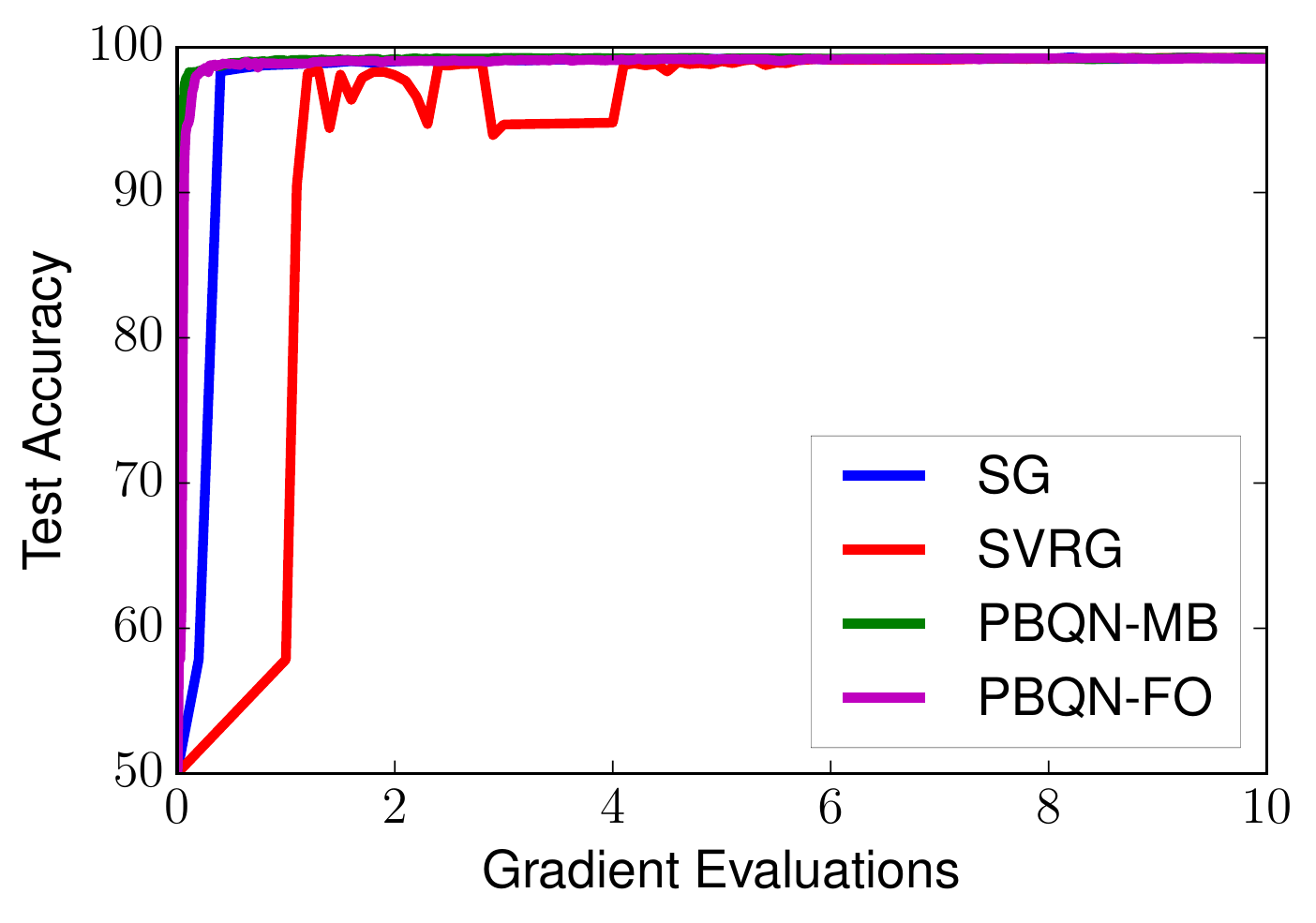}
		
		\par\end{centering}
	\caption{ \textbf{spam dataset:} Performance of the progressive batching L-BFGS method (PBQN), with multi-batch (25\% overlap) and full-overlap approaches, and the SG and SVRG methods.}
	\label{exp:spam} 
\end{figure*}

\begin{figure*}[!htp]
	\begin{centering}
		\includegraphics[width=0.24\linewidth]{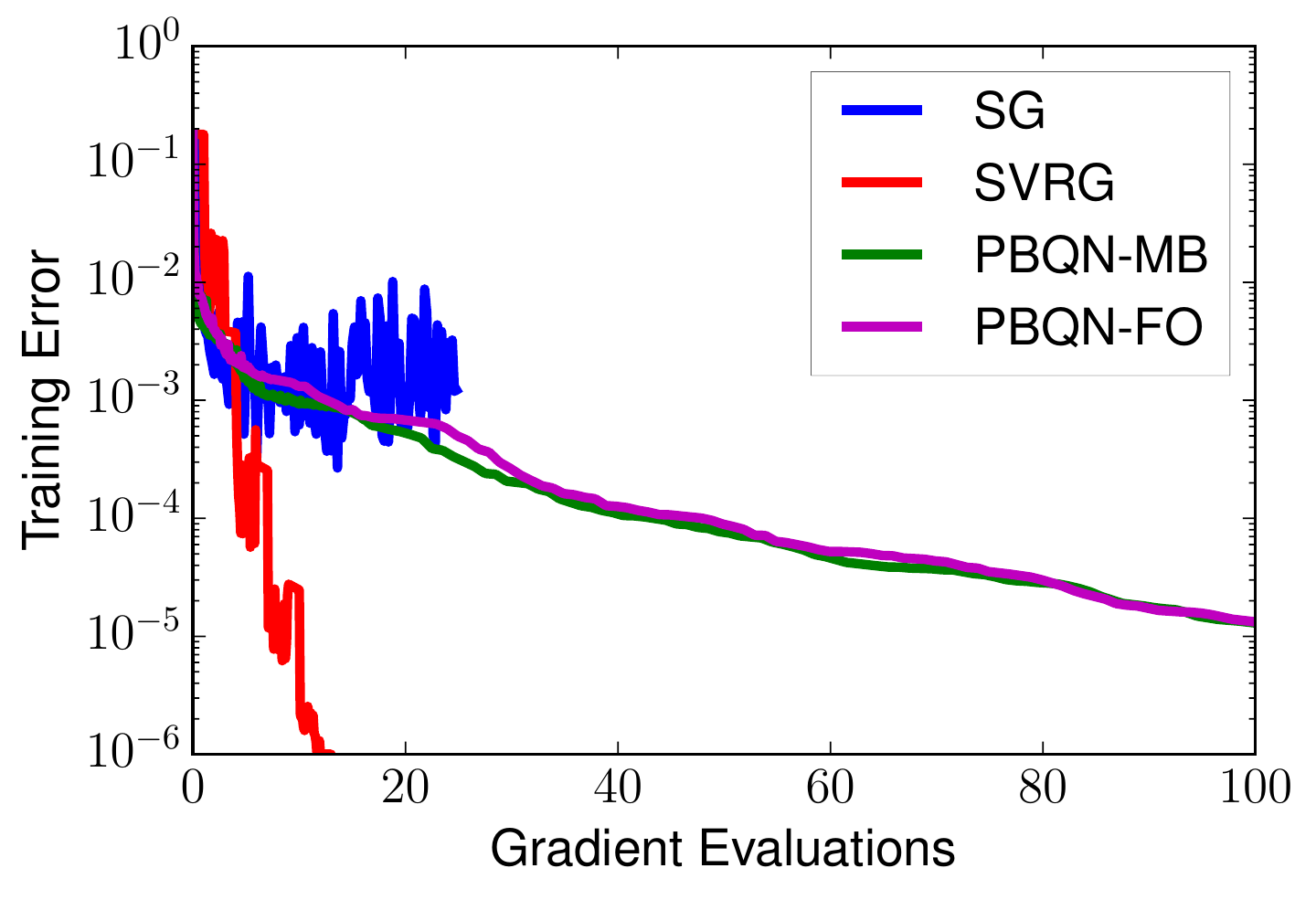}
		\includegraphics[width=0.24\linewidth]{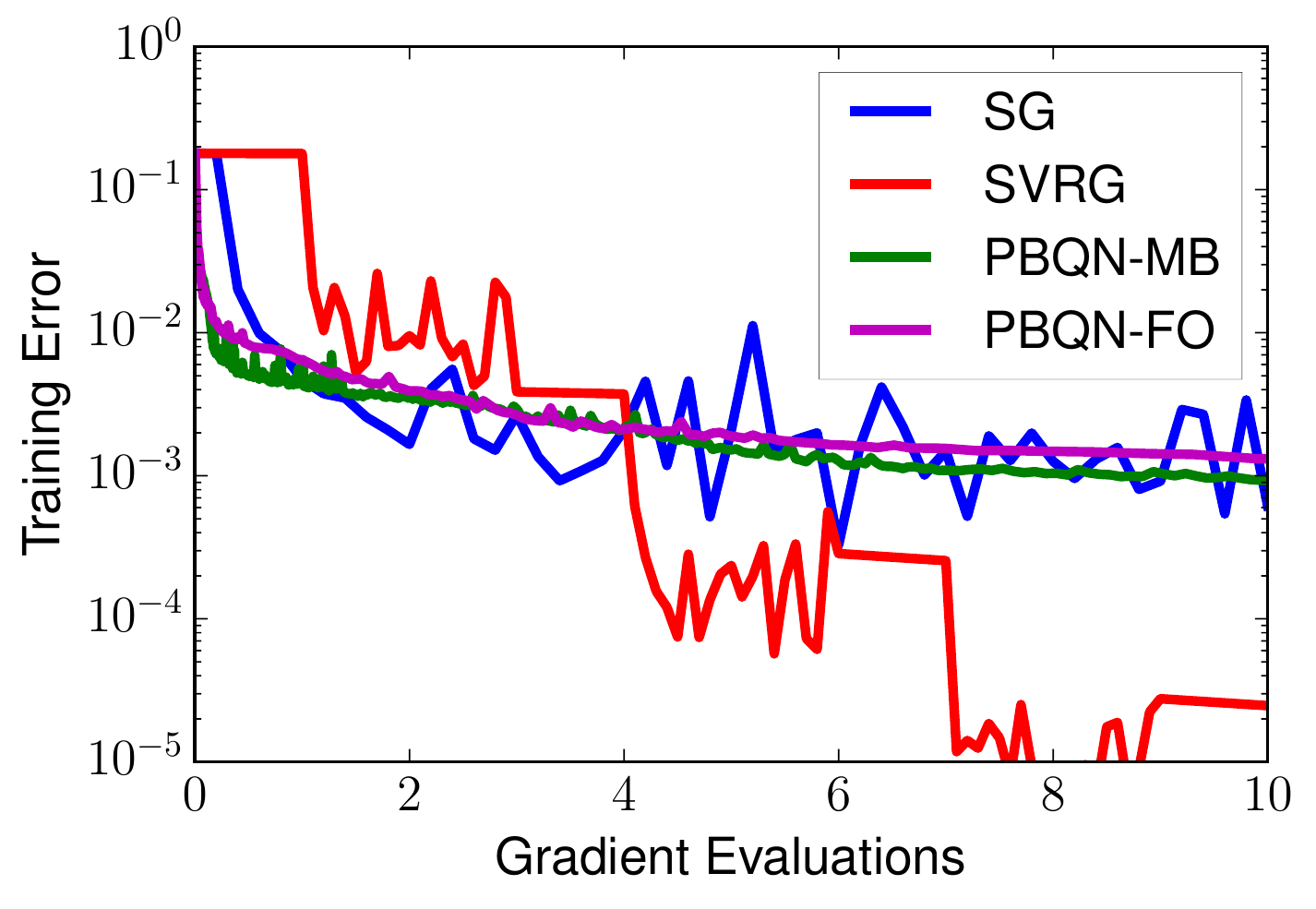}
		\includegraphics[width=0.24\linewidth]{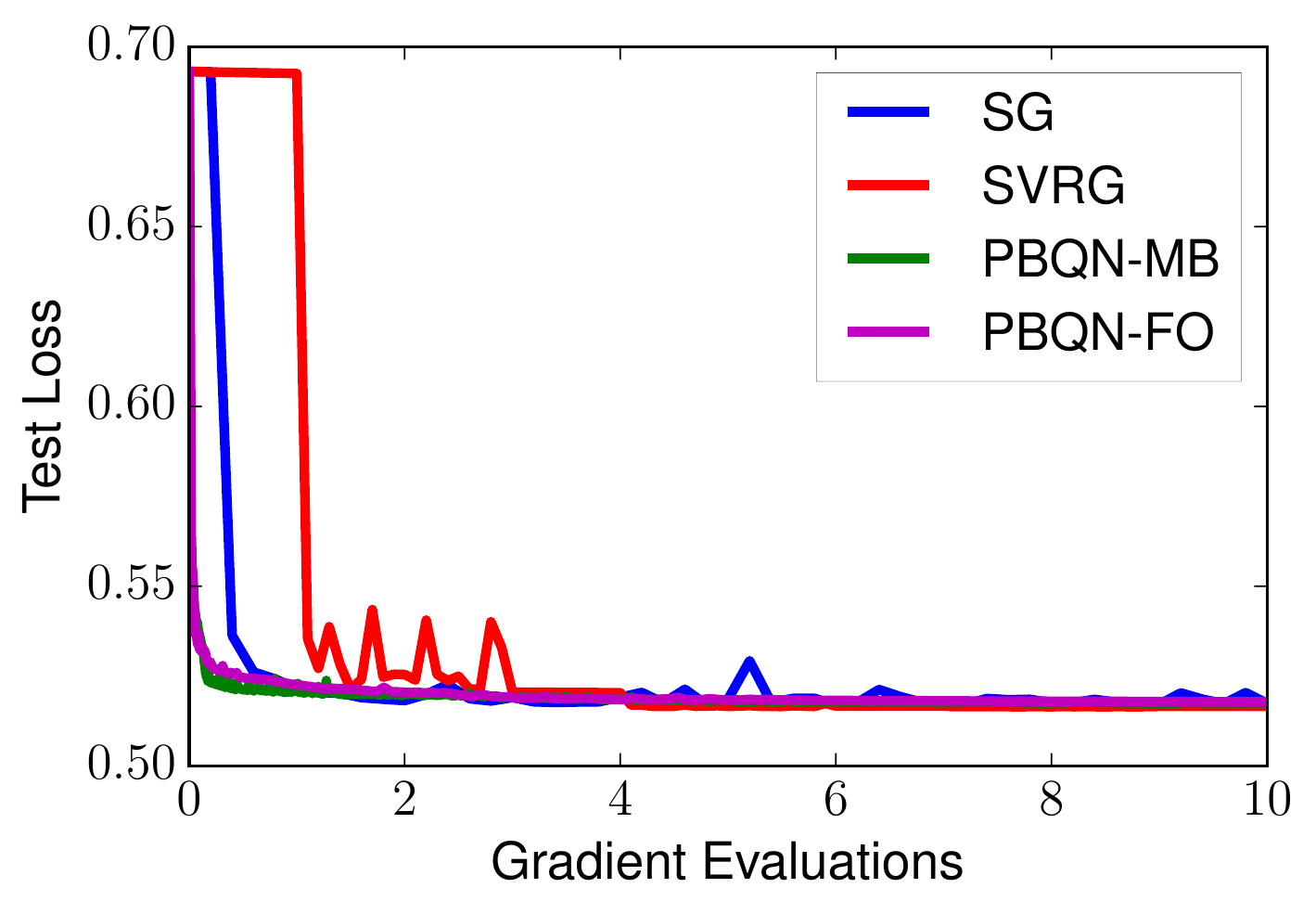}
		\includegraphics[width=0.24\linewidth]{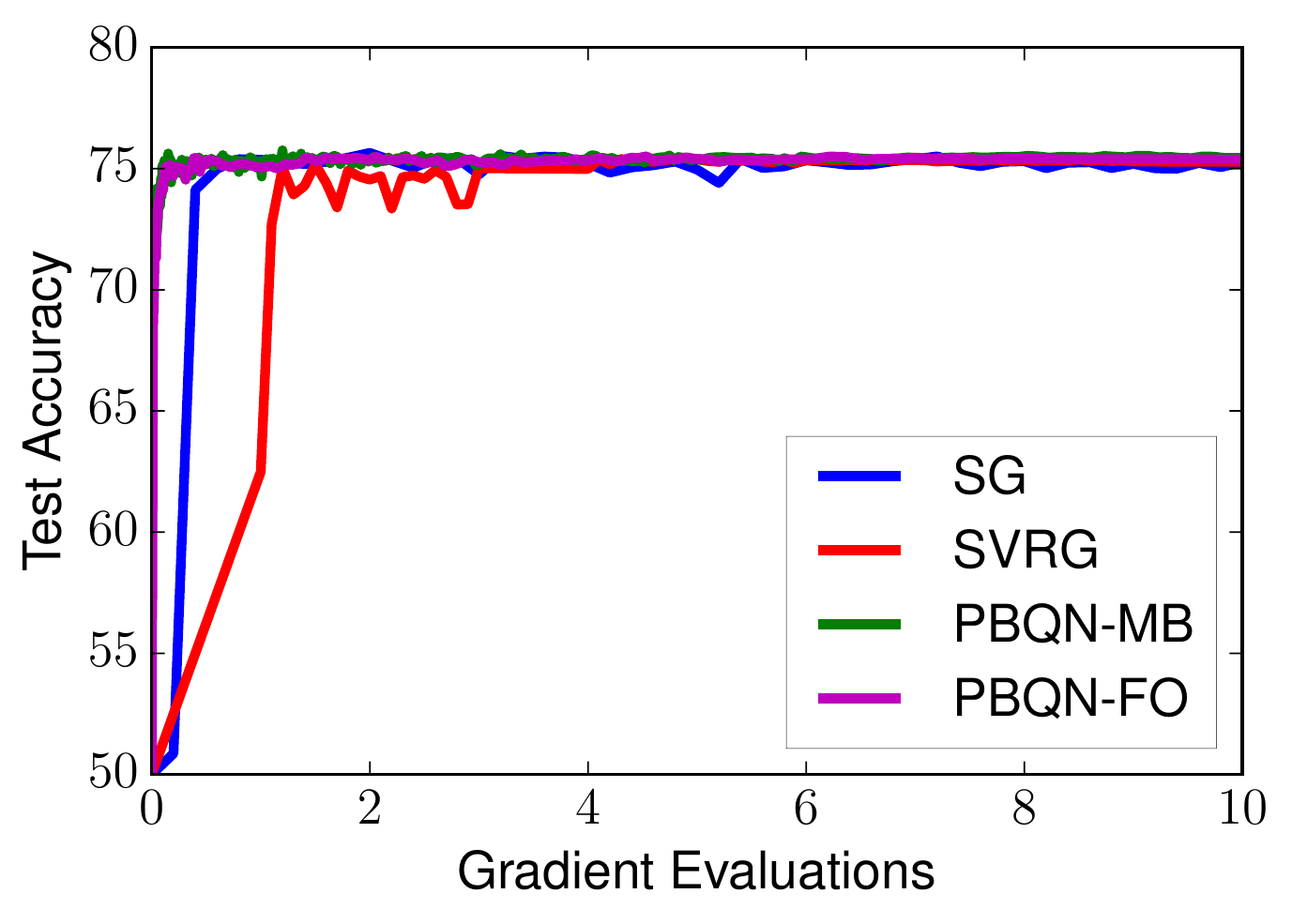}
		
		\par\end{centering}
	\caption{ \textbf{covertype dataset:} Performance of the progressive batching L-BFGS methods, with multi-batch (25\% overlap) and full-overlap approaches, and the SG and SVRG methods.}
	\label{exp:covtype} 
\end{figure*}

The proposed algorithm competes well for these two datasets
in terms of training error, test loss and test accuracy, and decreases these measures more evenly than the SG and SVRG. Our numerical experience indicates that formula \eqref{eq: initial step length} is quite effective at estimating the steplength parameter, as it is accepted by the backtracking line search for most iterations. As a result, the line search computes very few additional function values.

It is interesting to note that SVRG  is not as efficient in the initial epochs compared to PBQN or SG, when measured either in terms of test loss and test accuracy.  The training error for SVRG decreases rapidly in  later epochs but this rapid improvement is not observed in the test loss and accuracy. Neither the PBQN nor SVRG  significantly outperforms the other across all datasets tested in terms of training error, as observed in the supplement. 

Our results  indicate that defining the curvature vector using the MB approach is preferable to using the FB approach.
The number of iterations required by the PBQN method is significantly smaller compared to the SG method, suggesting the potential efficiency gains of a parallel implementation of our algorithm.

	\subsection{Results on Neural Networks}

We have performed a preliminary investigation into the performance of the PBQN algorithm for training neural networks.  As is well-known, the resulting optimization problems are quite difficult due to the existence of local minimizers, some of which generalize poorly. Thus our first requirement when applying the PBQN method was to obtain as good generalization as SG, something we have achieved.

Our investigation into how to obtain fast performance is, however, still underway for reasons discussed below.  Nevertheless, our  results are worth reporting because they show that our line search procedure is performing as expected, and that the overall number of iterations required by the PBQN method is small enough so that a parallel implementation could yield state-of-the-art results, based  on the theoretical performance model detailed in the supplement.

We compared our algorithm, as described in Section \ref{sec:progressive}, against SG and Adam \cite{kingma2014adam}. 
It has taken many years to design regularizations techniques and heuristics that greatly improve the performance of the SG method for deep learning \cite{srivastava2014dropout,ioffe2015batch}. These include batch normalization and dropout, which (in their current form) are not conducive to the PBQN approach due to the need for gradient consistency when evaluating the curvature pairs in L-BFGS. Therefore, we do not implement batch normalization and dropout in any of the methods tested, and leave the study of their extension to the PBQN setting for future work.

We consider three network architectures: (i) a small convolutional neural network on CIFAR-10 ($\mathcal{C}$) \cite{krizhevsky2009learning}, (ii) an AlexNet-like convolutional network on MNIST and CIFAR-10 ($\mathcal{A}_1, \mathcal{A}_2$, respectively) \cite{lecun1998gradient,krizhevsky2012imagenet}, and (iii) a residual network (ResNet18) on CIFAR-10 ($\mathcal{R}$) \cite{he2016deep}. The network architecture details and additional plots are given in the supplement. All of these networks were implemented in PyTorch \cite{paszke2017automatic}.
The results for the CIFAR-10 AlexNet and CIFAR-10 ResNet18 are given in Figures \ref{exp:cifar10_alexnet} and \ref{exp:cifar10_resnet18}, respectively. We report results  both against the total number of iterations and the total number of gradient evaluations. Table \ref{table: test accuracy} shows the best test accuracies attained by each of the four methods over the various networks.

In all our experiments, we  initialize the batch size as $|S_0| = 512$ in the PBQN method, and fix the batch size to $|S_k| = 128$ for SG and Adam. The parameter $\theta$ given in \eqref{IPQN: test}, which controls the batch size increase in the PBQN method, was tuned lightly by chosing among the 3 values: 0.9, 2, 3.    SG and Adam are tuned using a development-based decay (dev-decay) scheme, which track the best validation loss at each epoch and reduces the steplength by a constant factor $\delta$ if the validation loss does not improve after $e$ epochs.   

\begin{table}[!htp]
\caption{Best test accuracy performance of SG, Adam, multi-batch L-BFGS, and full overlap L-BFGS on various networks over 5 different runs and initializations.}
\label{table: test accuracy}
\vskip 0.15in
\centering
\begin{tabular}{c p{1cm} p{1cm} p{1cm} p{1cm}}
\toprule
Network           & SG & Adam & MB & FO \\
\midrule
$\mathcal{C}$    & 66.24 & 67.03 & 67.37 & 62.46 \\
$\mathcal{A}_1$  & 99.25 & 99.34 & 99.16 & 99.05 \\
$\mathcal{A}_2$  & 73.46 & 73.59 & 73.02 & 72.74 \\
$\mathcal{R}$    & 69.5  & 70.16 & 70.28 & 69.44 \\
\bottomrule
\end{tabular}
\end{table}

\begin{figure*}[!htp]
\begin{centering}
\includegraphics[width=0.24\linewidth]{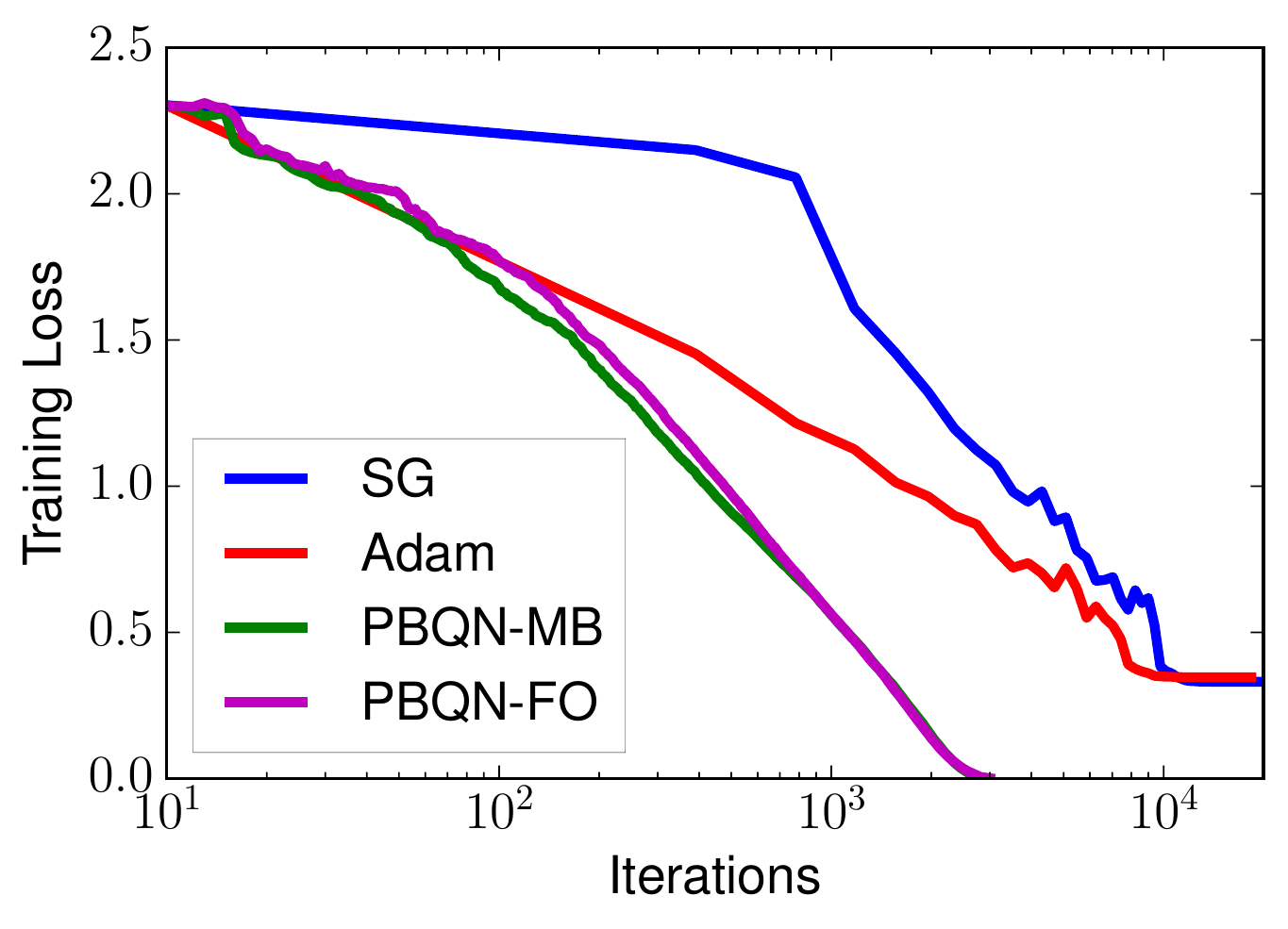}
\includegraphics[width=0.24\linewidth]{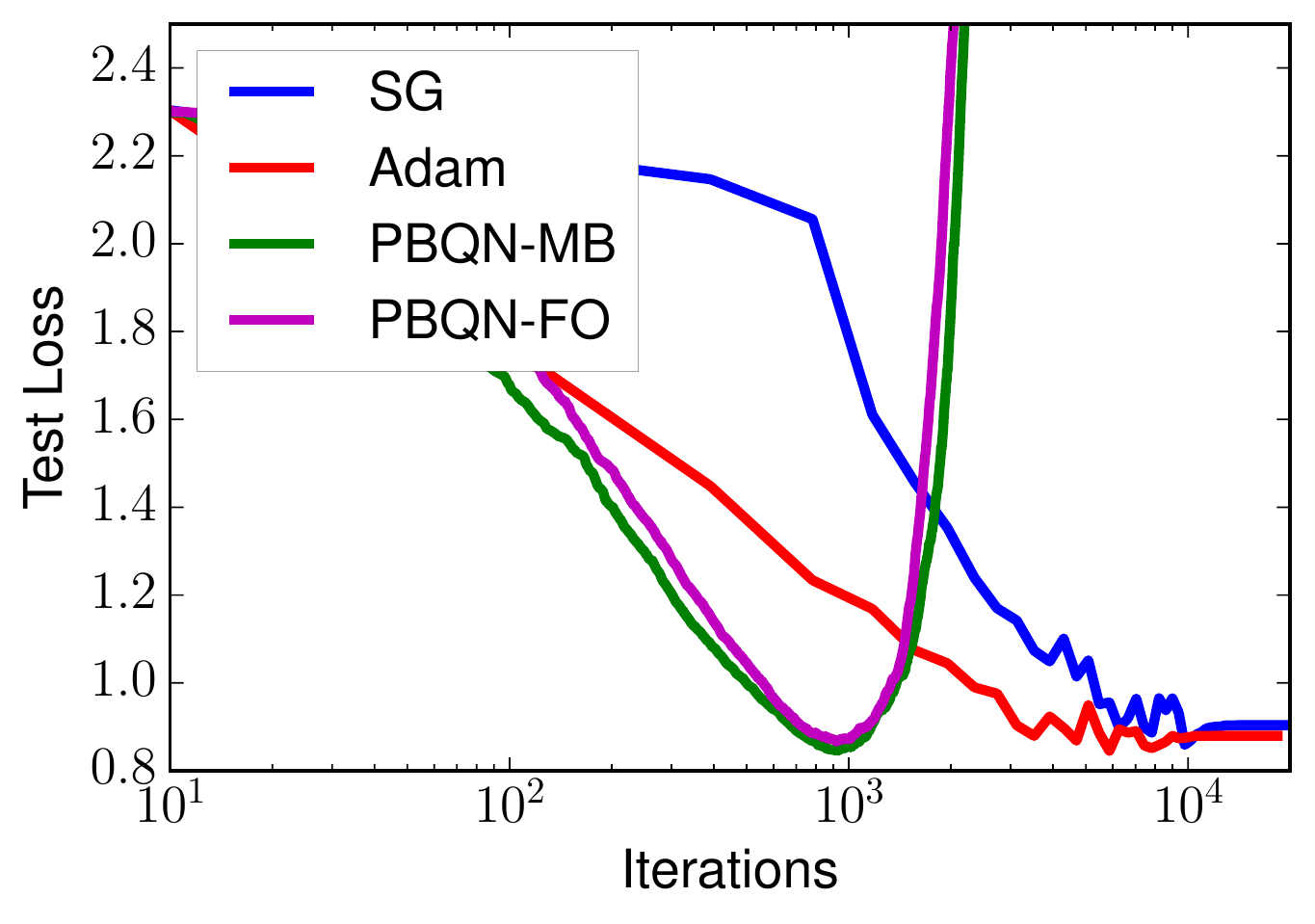}
\includegraphics[width=0.24\linewidth]{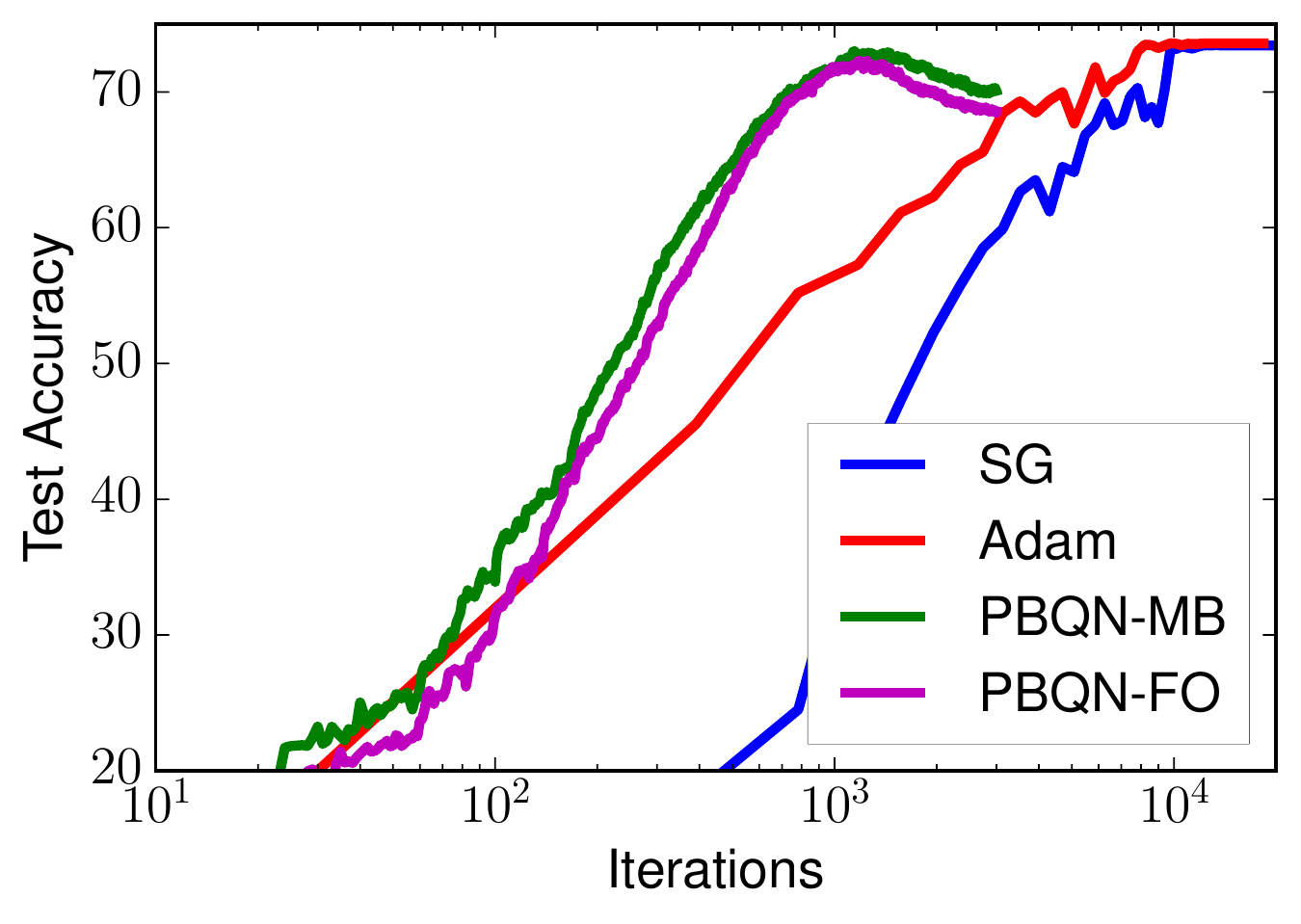}
\includegraphics[width=0.24\linewidth]{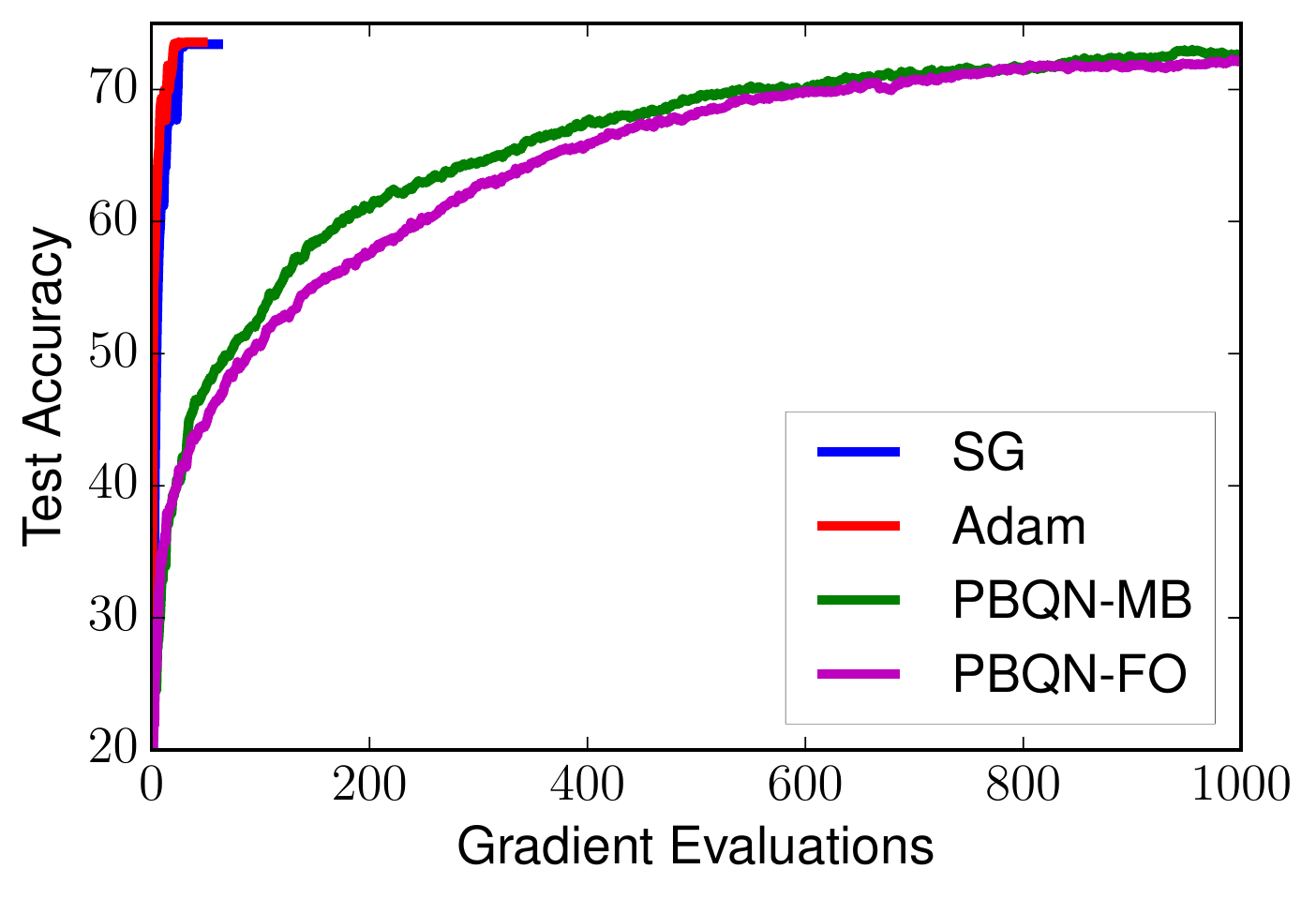}
	\caption{\textbf{CIFAR-10 AlexNet $(\mathcal{A}_2)$:} Performance of the progressive batching L-BFGS methods, with multi-batch (25\% overlap) and full-overlap approaches, and the SG and Adam methods. The best results for L-BFGS are achieved with $\theta = 0.9$.}
\label{exp:cifar10_alexnet}
\end{centering}
\end{figure*}

\begin{figure*}[!htp]
\begin{centering}
	\includegraphics[width=0.24\linewidth]{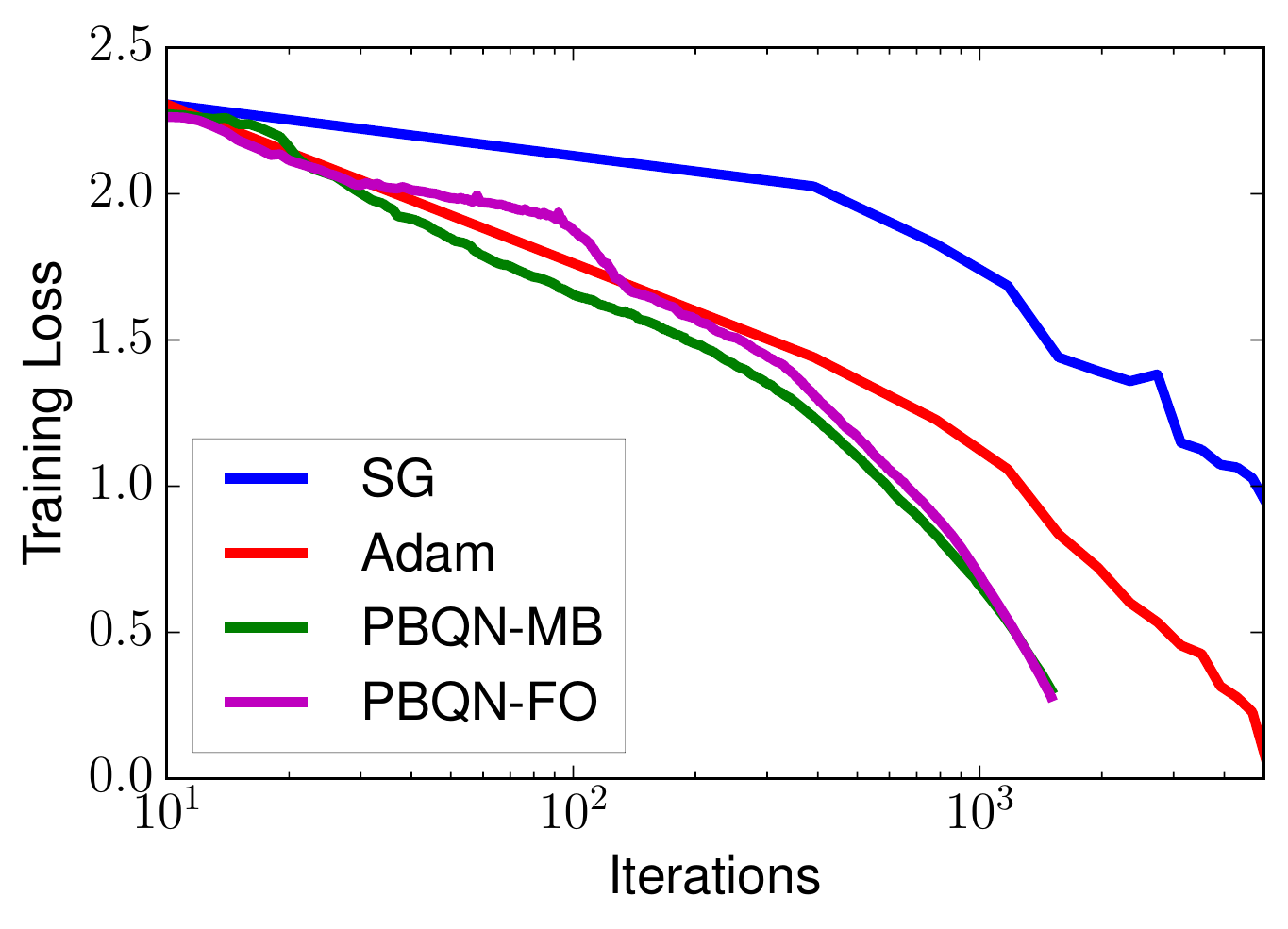}
	\includegraphics[width=0.24\linewidth]{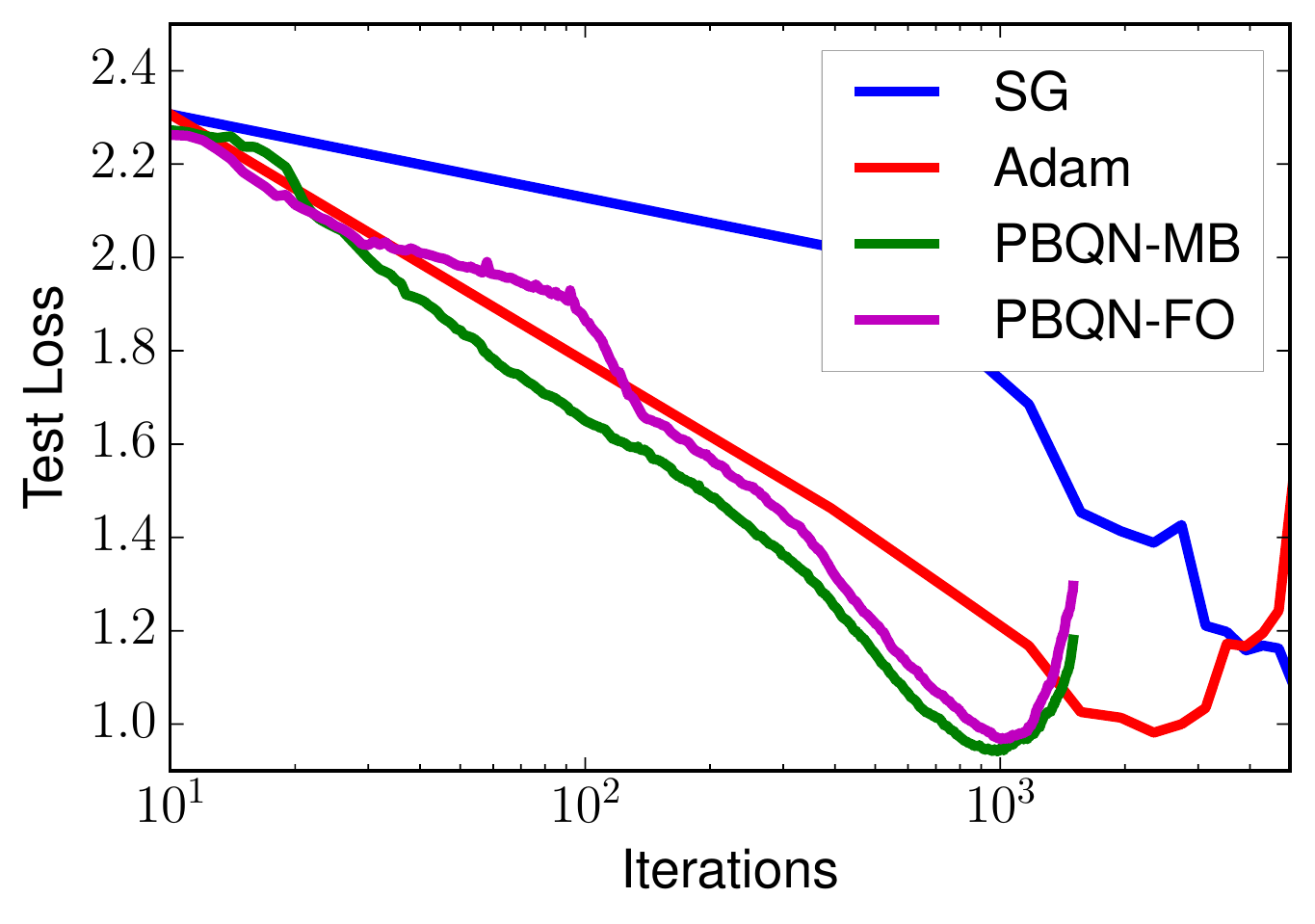}
	\includegraphics[width=0.24\linewidth]{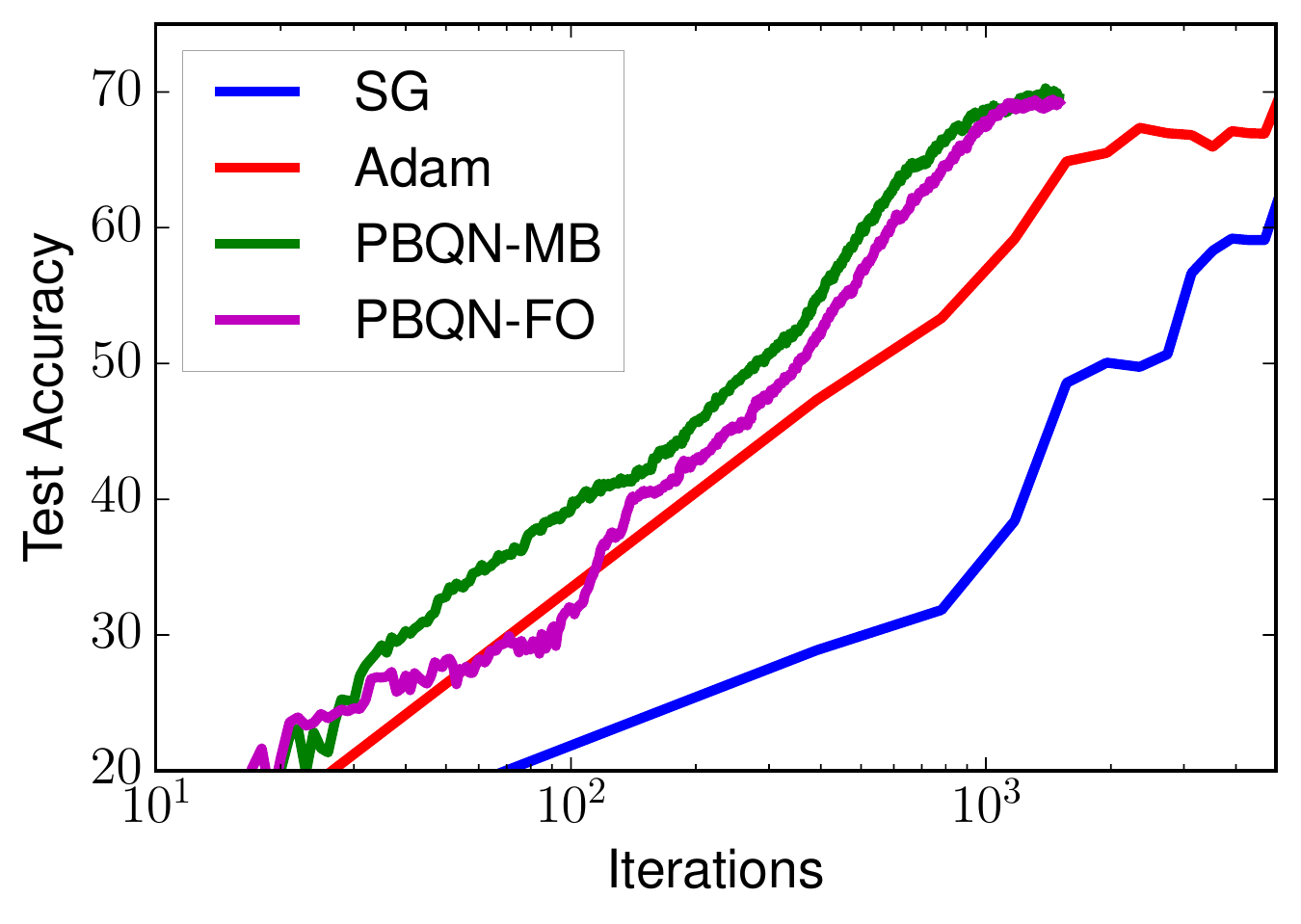}
	\includegraphics[width=0.24\linewidth]{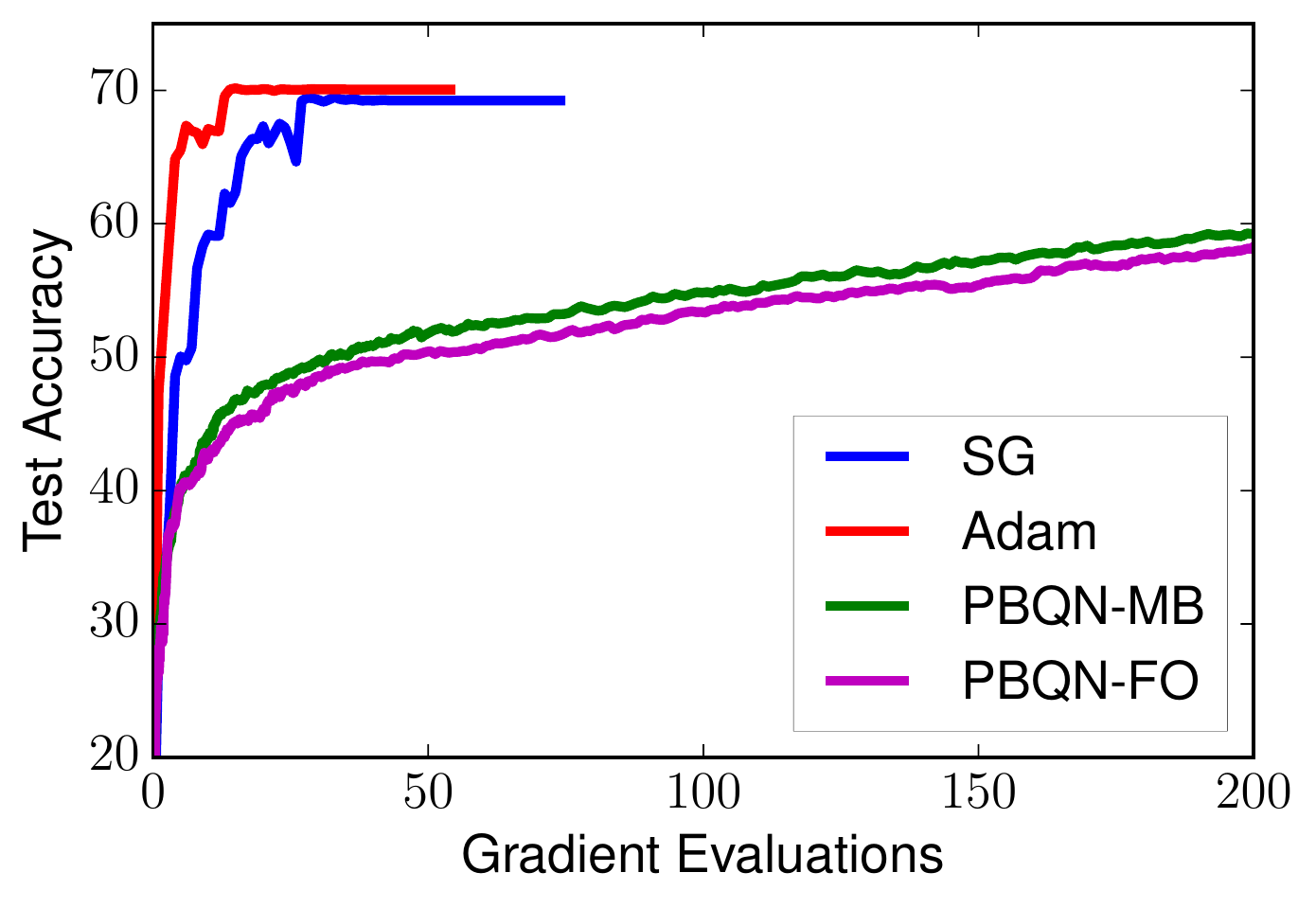}
	\caption{\textbf{CIFAR-10 ResNet18 $(\mathcal{R})$:} Performance of the progressive batching L-BFGS methods, with multi-batch (25\% overlap) and full-overlap approaches, and the SG and Adam methods. The best results for L-BFGS are achieved with $\theta = 2$.}
	\label{exp:cifar10_resnet18}
	\end{centering}
\end{figure*}

We observe from our results that the PBQN method achieves a similar test accuracy as SG and Adam, but requires more gradient evaluations. Improvements in performance can be obtained by ensuring that the PBQN method exerts a finer control on the sample size in the small batch regime --- something that requires further investigation. Nevertheless, the small number of iterations required by the PBQN method,  together with the fact that it employs larger batch sizes than SG during much of the run, suggests that a distributed version similar to a data-parallel distributed implementation of the SG method \cite{chen2016revisiting,das2016distributed} would lead to a highly competitive method.  

Similar to the logistic regression case, we observe that the steplength computed via \eqref{eq: initial step length} is almost always accepted by the Armijo condition, and typically lies within $(0.1, 1)$. 
Once the algorithm has trained for a significant number of iterations using full-batch, the algorithm begins to overfit on the training set, resulting in worsened test loss and accuracy, as observed in the graphs.

\section{Final Remarks}
\label{sec:finalr}

Several types of quasi-Newton methods have been proposed in the literature to address the challenges arising in machine learning. Some of these method operate in the purely stochastic setting (which makes quasi-Newton updating difficult) or in the purely batch regime (which leads to generalization problems). We believe that progressive batching is the right context for designing an L-BFGS method that has good generalization properties, does not expose any free parameters, and has fast convergence. The advantages of our approach are clearly seen in logistic regression experiments. To make the new method competitive with SG and Adam for deep learning, we need to improve several of its components. This includes the design of a more robust progressive batching mechanism,  the redesign of batch normalization and dropout heuristics to improve the generalization performance of our method for training larger networks, and most importantly, the design of a parallelized implementation that takes advantage of the higher granularity of each iteration. We believe that the potential of the proposed approach as an alternative to SG for deep learning is worthy of further investigation.

	
	\section*{Acknowledgements}

We thank Albert Berahas for his insightful comments regarding multi-batch L-BFGS and probabilistic line searches, as well as for his useful feedback on earlier versions of the manuscript. We also thank the anonymous reviewers for their useful feedback. Bollapragada is supported by DOE award DE-FG02-87ER25047. Nocedal is supported by NSF award DMS-1620070. Shi is supported by Intel grant SP0036122.

	\bibliographystyle{icml2018}
	\bibliography{references_arxiv}
	\onecolumn
	\appendix
	
	\section{Initial Step Length Derivation}

To establish our results, recall that the stochastic quasi-Newton method is defined as
\begin{equation} \label{iter}
x_{k+1} = x_k - \alpha_k H_k g_k^{S_k},
\end{equation}
where the batch (or subsampled) gradient is given by
\begin{equation}   \label{batch}
g_k^{S_k} = \nabla F_{S_k}(x_k) = \frac{1}{|S_k|}\sum_{i \in S_k} \nabla F_i (x_k),
\end{equation}
and the set $S_k \subset \{1,2,\cdots\}$ indexes data points $(y^i, z^i)$. The algorithm selects the Hessian approximation $H_k$ through quasi-Newton updating prior to selecting the new sample $S_k$ to define the search direction $p_k$. We will use $\E_k$ to denote the conditional expectation at $x_k$ and use $\E$ to denote the total expectation. 

The primary theoretical mechanism for determining batch sizes is the exact variance inner product quasi-Newton (IPQN) test, which is defined as
\begin{equation}\label{IPgrad: exacttest}
\frac{\E_k \left[\left((H_k\nabla F(x_k))^T (H_k g_k^{i}) -  \|H_k\nabla 
F(x_k)\|^2\right)^2\right]}{|S_k|} \leq \theta^2  \|H_k\nabla F(x_k)\|^4.
\end{equation} 

We establish the inequality  used to determine the initial steplength $\alpha_k$ for the stochastic line search.
\begin{lem}
Assume that $F$ is continuously differentiable with Lipschitz continuous gradient with Lipschitz constant $L$. Then
\begin{equation*}
\E_k \left[ F(x_{k+1}) \right] \leq F(x_k) - \alpha_k \nabla F(x_k)^T H_k^{1/2} W_k H_k^{1/2}\nabla F(x_k) ,\label{sequoia}
\end{equation*}
where 
$$W_k = \left(I - \frac{L\alpha_k}{2}\left(1 +  \frac{{\rm Var}\{H_k g_k^{i}\}}{|S_k| \|H_k\nabla F(x_k)\|^2}\right) H_k \right),$$
and $\Var\{H_k g_k^{i}\} = \E_k\left[\|H_k g_k^{i} - H_k \nabla F(x_k)\|^2\right]$.
\end{lem}
\begin{proof}
By Lipschitz continuity of the gradient, we have that
\begin{align}
\E_k \left[ F(x_{k+1}) \right] & \leq F(x_k) - \alpha_k \nabla F(x_k)^T H_k \E_k\left[g_k^{S_k}\right] + \frac{L\alpha_k^2}{2}\E_k\left[\|H_k g_k^{S_k}\|^2\right] \nonumber \\
& = F(x_k) - \alpha_k \nabla F(x_k)^T H_k \nabla F(x_k) + \frac{L\alpha_k^2}{2} \left( \|H_k \nabla F(x_k)\|^2 + \E_k\left[\|H_k g_k^{S_k} - H_k \nabla F(x_k)\|^2\right] \right)\nonumber \\
& \leq F(x_k) - \alpha_k \nabla F(x_k)^T H_k \nabla F(x_k) + \frac{L\alpha_k^2}{2} \left( \|H_k \nabla F(x_k)\|^2 +\frac{{\rm Var}\{H_k g_k^{i}\}}{|S_k|\|H_k\nabla F(x_k)\|^2}\|H_k \nabla F(x_k)\|^2 \right) \nonumber \\
& = F(x_k) - \alpha_k \nabla F(x_k)^T H_k^{1/2}\left(I - \frac{L\alpha_k}{2}\left(1 +  \frac{{\rm Var}\{H_k g_k^{i} \}}{|S_k| \|H_k\nabla F(x_k)\|^2}\right)H_k\right) H_k^{1/2} \nabla F(x_k) \nonumber\\
& = F(x_k) - \alpha_k \nabla F(x_k)^T H_k^{1/2} W_k H_k^{1/2} \nabla F(x_k) .\nonumber
\end{align}
\end{proof}

\section{Convergence Analysis}  \label{sec:or}

For the rest of our analysis, we make the following two assumptions.
\begin{assum}\label{assum: orth}
The orthogonality condition is satisfied for all $k$, i.e.,
\begin{equation} \label{orth-i}
\frac{\E_k \left[\left\|H_k g_k^i - \frac{(H_k g_k^i)^T (H_k\nabla F(x_k))}{\|H_k\nabla F(x_k)\|^2}H_k\nabla F(x_k)\right\|^2\right]}{|S_k|} \leq \nu^2 \|H_k\nabla F(x_k)\|^2,
\end{equation} 
for some large $\nu > 0$.
\end{assum}

\begin{assum}\label{assum: eigs}
The eigenvalues of $H_k$ are contained in an interval in $\mathbb{R}^+$, i.e., for all $k$ there exist constants $\Lambda_2 \geq \Lambda_1 > 0$ such that
\begin{equation}
\Lambda_1 I \preceq H_k \preceq \Lambda_2 I.
\end{equation}
\end{assum}

Condition \eqref{orth-i} ensures that the stochastic quasi-Newton direction is bounded away from orthogonality to $-H_k \nabla F(x_k)$, with high probability, and prevents the variance in the individual quasi-Newton directions to be too large relative to the variance in the individual quasi-Newton directions along $-H_k \nabla F(x_k)$. Assumption \ref{assum: eigs} holds, for example, when $F$ is convex and a regularization parameter is included so that any subsampled Hessian $\nabla^2 F_{S}(x)$ is positive definite. It can also be shown to hold in the non-convex case by applying cautious BFGS updating; e.g. by updating $H_k$ only when $y_k^T s_k \geq \epsilon \|s_k\|_2^2$
where $\epsilon > 0$ is a predetermined constant \cite{berahas2016multi}.

We begin by establishing a technical descent lemma.

\begin{lem} \label{c-thmlin} 
	 Suppose that $F$ is twice continuously differentiable and that there exists a constant $L > 0$ such that 
	\begin{equation}  \label{c-lip}
	 \nabla^2 F(x) \preceq L I, \quad \forall x \in \R^d.
	\end{equation} 
	Let $\{x_k\}$ be generated by iteration \eqref{iter} for any $x_0$, 
	where $|S_{k}|$ is chosen by the (exact variance) inner product quasi-Newton test \eqref{IPgrad: exacttest} for given constant $\theta > 0$ and suppose that assumptions \eqref{assum: orth} and \eqref{assum: eigs} hold. Then, for any $k$,
	\begin{align}\label{c-sample-norm}
	\E_k \left[\|H_k g_k^{S_k}\|^2\right] 
	& \leq(1 + \theta^2 + \nu^2)\|H_k\nabla F(x_k)\|^2 .
	\end{align}
	Moreover, if $\alpha_k$ satisfies 
	\begin{equation}   \label{c-stepform}
	\alpha_k = \alpha \leq \frac{1}{(1 + \theta^2+\nu^2)L\Lambda_2},
	\end{equation}
	we have that 
	\begin{align} 
	\E_k[F(x_{k+1})] & \leq F(x_k)  - \frac{\alpha}{2} \|H_k^{1/2}\nabla F(x_k)\|^2   \label{c-lin-ineq} .
	\end{align}  
\end{lem}
\begin{proof} By Assumption \eqref{assum: orth}, the orthogonality condition, we have that
	\begin{align} 
		 \E_k \left[\left\|H_k g_k^{S_k} -  \frac{(H_k g_k^{S_k})^T (H_k\nabla F(x_k))}{\|H_k\nabla F(x_k)\|^2}H_k\nabla F(x_k)\right\|^2\right] 
		&\leq \frac{\E_k \left[\left\|H_k g_k^i - \frac{(H_k g_k^i)^T (H_k\nabla F(x_k))}{\|H_k\nabla F(x_k)\|^2}H_k\nabla F(x_k)\right\|^2\right]}{|S_k|} \label{eq:lhs_inequality}\\
		&\leq \nu^2 \,\|H_k\nabla F(x_k)\|^2. \nonumber
	\end{align}
	Now, expanding the left hand side of inequality \eqref{eq:lhs_inequality}, we get
	\begin{align} \nonumber
	& \E_k \left[\left\|H_k g_k^{S_k} -  \frac{(H_k g_k^{S_k})^T (H_k\nabla F(x_k))}{\|H_k\nabla F(x_k)\|^2}H_k\nabla F(x_k)\right\|^2\right] \nonumber \\
	& ~ = ~ \E_k \left[\|H_k g_k^{S_k}\|^2\right] -\frac{2 \E_k \left[\left( (H_k g_k^{S_k})^T (H_k\nabla F(x_k))\right)^2\right]}{\|H_k\nabla F(x_k)\|^2} 
	 + \frac{\E_k \left[\left( (H_k g_k^{S_k})^T (H_k\nabla F(x_k))\right)^2\right]}{\|H_k\nabla F(x_k)\|^2} \nonumber\\
	& ~ = ~ \E_k \left[\|H_k g_k^{S_k} \|^2\right] - \frac{ \E_k \left[\left( (H_k g_k^{S_k})^T (H_k\nabla F(x_k))\right)^2\right]}{\|H_k\nabla F(x_k)\|^2}  \nonumber \\
	&  ~ \leq ~ \nu^2 \,\|H_k\nabla F(x_k)\|^2. \nonumber
	\end{align}
	Therefore, rearranging gives the inequality
	\begin{equation}\label{step-two-term}
	\E_k \left[\|H_k g_k^{S_k}\|^2\right] \leq \frac{ \E_k \left[\left( (H_k g_k^{S_k})^T (H_k\nabla F(x_k)) \right)^2\right]}{\|H_k\nabla F(x_k)\|^2} + \nu^2 \|H_k\nabla F(x_k)\|^2.
	\end{equation}
	To bound the first term on the right side of this inequality, we use the inner product quasi-Newton test; in particular, $|S_k|$  satisfies 
	\begin{align}
	\E_k \left[ \left( (H_k \nabla F(x_k))^T (H_k g_k^{S_k})) - \| H_k \nabla F(x_k) \|^2 \right)^2 \right] \nonumber & \leq \frac{ \E_k \left[\left((H_k\nabla F(x_k))^T (H_k g_k^{i}) -  \|H_k\nabla F(x_k)\|^2\right)^2\right]}{|S_k|} \nonumber \\
	& \leq ~ \theta^2  \|H_k\nabla F(x_k)\|^4, \label{eq:theta_bound}
	\end{align}
	where the second inequality holds by the IPQN test. Since
	\begin{equation}\label{eq:eq:decomposition}
	\E_k \left[ \left( (H_k \nabla F(x_k))^T (H_k g_k^{S_k}) - \| H_k \nabla F(x_k) \|^2 \right)^2 \right] = \E_k \left[ \left( (H_k \nabla F(x_k))^T (H_k g_k^{S_k}) \right)^2 \right] - \|H_k \nabla F(x_k) \|^4,
	\end{equation}
	we have
	\begin{align} 
	\E_k \left[\left( (H_k g_k^{S_k})^T (H_k\nabla F(x_k)) \right)^2\right] & \leq  \|H_k\nabla F(x_k)\|^4 + \theta^2 \|H_k\nabla F(x_k)\|^4 \nonumber\\
	&= (1 + \theta^2) \|H_k\nabla F(x_k)\|^4, \label{eq:1+theta_bound}
	\end{align}
	by \eqref{eq:theta_bound} and \eqref{eq:eq:decomposition}.
	Substituting \eqref{eq:1+theta_bound} into \eqref{step-two-term}, we get the following bound on the length of the search direction:
	\begin{align*}
	\E_k \left[\|H_k g_k^{S_k}\|^2\right] 
	& \leq(1 + \theta^2 + \nu^2)\|H_k\nabla F(x_k)\|^2 ,
	\end{align*}
	which proves \eqref{c-sample-norm}.
	Using this inequality, Assumption \ref{assum: eigs}, and bounds on the Hessian and steplength \eqref{c-lip} and \eqref{c-stepform}, we have
	\begin{align} 
	\E_k [F(x_{k+1}) ] & \leq F(x_k)  - \E_k \left[\alpha (H_k g_k^{S_k})^T\nabla F(x_k)\right] + \E_k \left[\frac{L\alpha^2}{2} \|H_k g_k^{S_k}\|^2\right] \nonumber\\
	& = F(x_k)  -\alpha \nabla F(x_k)^TH_k\nabla F(x_k) + \frac{L\alpha^2}{2} \E_k [\|H_k g_k^{S_k} \|^2] \nonumber\\
	&\leq F(x_k)  - \alpha \nabla F(x_k)^TH_k\nabla F(x_k) +\frac{L \alpha^2}{2} (1 + \theta^2+\nu^2)\|H_k\nabla F(x_k)\|^2\nonumber\\	
	& = F(x_k)  - \alpha (H_k^{1/2}\nabla F(x_k))^T\left(I - \frac{L\alpha(1 + \theta^2 + \nu^2)}{2}H_k\right)H_k^{1/2}\nabla F(x_k) \nonumber\\
	& \leq F(x_k)  - \alpha \left(1 - \frac{L \Lambda_2 \alpha (1 + \theta^2 + \nu^2)}{2} \right) \|H_k^{1/2}\nabla F(x_k)\|^2 \nonumber \\
	&\leq F(x_k) - \frac{\alpha}{2} \|H_k^{1/2}\nabla F(x_k)\|^2  . \nonumber
	\end{align}
\end{proof}

We now show that the stochastic quasi-Newton iteration \eqref{iter} with a fixed steplength $\alpha$ is linearly convergent when $F$ is strongly convex. In the following discussion, $x^*$ denotes the  minimizer of $F$.

\begin{thm} \label{thmlin} 
	Suppose that $F$ is twice continuously differentiable and that there exist constants $0 < \mu \leq L$ such that 
	\begin{equation}
	\mu I   \preceq \nabla^2 F(x) \preceq L I, \quad \forall x \in \R^d.
	\end{equation} 
	Let $\{x_k\}$ be generated by iteration \eqref{iter}, for any $x_0$, 
	where $|S_{k}|$ is chosen by the (exact variance) inner product quasi-Newton test \eqref{IPgrad: exacttest} and suppose that the assumptions \eqref{assum: orth} and \eqref{assum: eigs} hold. Then, if $\alpha_k$ satisfies \eqref{c-stepform}
	we have that   
	\begin{equation} \label{linear}
	\E[F(x_k) - F(x^*)] \leq \rho^k (F(x_0) - F(x^*)),
	\end{equation}
where  $x^*$ denotes the  minimizer of $F$, and
$
\rho= 1 -\mu \Lambda_1\alpha .
$

\end{thm}


\begin{proof} It is well-known \cite{bertsekas2003convex} that for strongly convex functions,   
\[
\|\nabla F(x_k)\|^2 \geq 2\mu [F(x_k) - F(x^*)].
\]
Substituting this into \eqref{c-lin-ineq} and subtracting $F(x^*)$ from both sides and using Assumption \ref{assum: eigs}, we obtain
\begin{align}
\E_k[F(x_{k+1}) - F(x^*)] &\leq  F(x_k) - F(x^*) - \frac{\alpha}{2}\|H_k^{1/2}\nabla F(x_k)\|^2 \nonumber \\
&\leq F(x_k) - F(x^*) - \frac{\alpha}{2}\Lambda_1\|\nabla F(x_k)\|^2 \nonumber\\
&\leq  (1 -\mu \Lambda_1 \alpha ) (F(x_k) - F(x^*)). \nonumber
\end{align}
The theorem follows from taking total expectation.

\end{proof}

We now consider the case when $F$ is nonconvex and bounded below.

\begin{thm} \label{thmsublin}
	Suppose that $F$ is twice continuously 
	differentiable and bounded below, and that there exists a constant $ L > 0$ such that 
	\begin{equation}
	\nabla^2 F(x) \preceq L I, \quad \forall x \in \R^d.
	\end{equation} 
	Let $\{x_k\}$ be generated by iteration \eqref{iter}, for any $x_0$, 
	where $|S_{k}|$ is chosen by 
	the (exact variance) inner product quasi-Newton test \eqref{IPgrad: exacttest} and suppose that the assumptions \eqref{assum: orth} and \eqref{assum: eigs} hold. Then, if $\alpha_k$ satisfies
	\eqref{c-stepform},
	%
	we have
	\begin{equation}   \label{convergence}
	\lim_{k \rightarrow \infty} \E [\|\nabla F(x_k)\|^2] \rightarrow 0.
	\end{equation}  
	Moreover, for any positive integer $T$ we have that 
	\begin{align*}
	\min_{0\leq k \leq T-1} \E [\|\nabla F(x_k)\|^2] &\leq \frac{2}{\alpha 
		T\Lambda_1}  (F(x_0) - F_{min}),
	\end{align*}
	where $F_{min}$ is a lower bound on  $F$ in $\mathbb{R}^d$.
\end{thm}
   					\begin{proof}
				From Lemma~\ref{c-thmlin} and by taking total expectation, we have
   					\begin{align*}
   						\E[F(x_{k+1})] &\leq \E[F(x_k)]  - \frac{\alpha}{2}\E [\|H_k^{1/2}\nabla F(x_k)\|^2],
   					\end{align*}	
				    and hence
				    \[
				    \E [\|H_k^{1/2}\nabla F(x_k)\|^2] \leq \frac{2}{\alpha} \E [F(x_k) - F(x_{k+1})] .
				    \]
   					Summing both sides of this inequality from $k= 0$ to $T-1$, and since
				$F$ is bounded below by  $F_{min}$, we get
   					\[
   					\sum_{k=0}^{T-1}\E [\|H_k^{1/2}\nabla F(x_k)\|^2] \leq \frac{2}{\alpha} \E [F(x_0) - F(x_{\mbox{\sc t}})] 
				\leq \frac{2}{\alpha}  [F(x_0) - F_{min}].
   					\]
				Using the bound on the eigenvalues of $H_k$ and taking limits, we obtain
\begin{equation*}
\Lambda_1\lim_{T \rightarrow \infty}\sum_{k=0}^{T-1}\E [\|\nabla F(x_k)\|^2] \leq  \lim_{T \rightarrow \infty} \sum_{k=0}^{T-1}\E [\|H_k^{1/2}\nabla F(x_k)\|^2] < \infty,
\end{equation*}
which implies \eqref{convergence}.
   					We can also conclude that
   					\begin{align*}
   					\min_{0\leq k \leq T-1} \E [\|\nabla F(x_k)\|^2] \leq & \frac{1}{T}\sum_{k=0}^{T}\E [\|\nabla F(x_k)\|^2] 
				\leq \frac{2}{\alpha T\Lambda_1}  (F(x_0) - F_{min}).
   					\end{align*}
   				\end{proof}	

	\newpage
	\section{Additional Numerical Experiments}
\label{sec:additional_exp}

\subsection{Datasets}

Table \ref{datasets} summarizes the datasets used for the experiments. Some of these datasets divide the data into training and testing sets; for the rest, we randomly divide the data so that the training set constitutes 90\% of the total. 
\begin{table}[htp] \label{datasets}
	\centering
		\caption{Characteristics of all datasets used in the experiments.}	\label{datasets}
	\vskip 0.15in
	\begin{tabular}{l r r c c c} \toprule
		Dataset & \# Data Points (train; test) & \# Features & \# Classes
		& Source
		\\ \midrule
		\texttt{gisette} & (6,000; 1,000) & 5,000 & 2 &\cite{CC01a}\\
		\texttt{mushrooms} & (7,311; 813) & 112 & 2 & \cite{CC01a}\\
		\texttt{sido} & (11,410; 1,268) & 4,932 & 2 & \cite{guyon2008design} \\
		\texttt{ijcnn} & (35,000; 91701) & 22 & 2 & \cite{CC01a}\\
		\texttt{spam} & (82,970; 9,219) & 823,470 & 2 & \cite{cormack2005spam,carbonetto2009new}\\
		\texttt{alpha} & (450,000; 50,000) & 500 & 2 & synthetic\\
		\texttt{covertype} & (522,910; 58,102) & 54 & 2 & \cite{CC01a}\\
		\texttt{url} & (2,156,517; 239,613) & 3,231,961 & 2 & \cite{CC01a}\\
		\texttt{MNIST} & (60,000; 10,000) & $28 \times 28$ & 10 & \cite{lecun1998gradient} \\
		\texttt{CIFAR-10} & (50,000; 10,000) & $32 \times 32$ & 10 & \cite{krizhevsky2009learning} \\
		\bottomrule
	\end{tabular} 
\vskip -0.1in
\end{table} 

The alpha dataset is a synthetic dataset that is available at \url{ftp://largescale.ml.tu-berlin.de}.

\subsection{Logistic Regression Experiments}

We report the numerical results on binary classification logistic regression problems on the $8$ datasets given in Table \ref{datasets}. We plot the performance measured in terms of training error, test loss and test accuracy against  gradient evaluations. We also report the behavior of the batch sizes and steplengths for both variants of the PBQN method.

\begin{figure*}[!htp]
	\begin{centering}
		\includegraphics[width=0.33\linewidth]{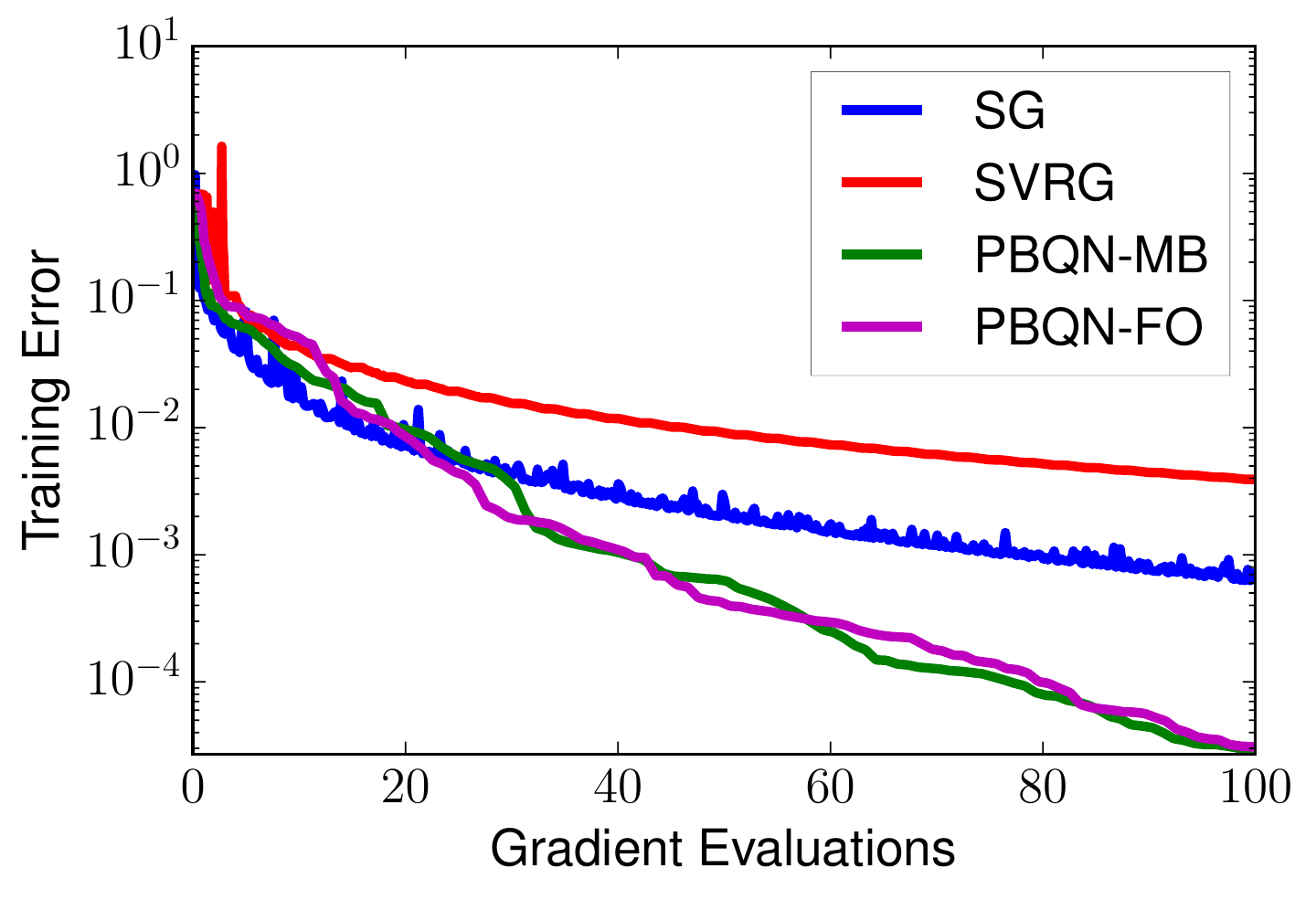}
		\includegraphics[width=0.33\linewidth]{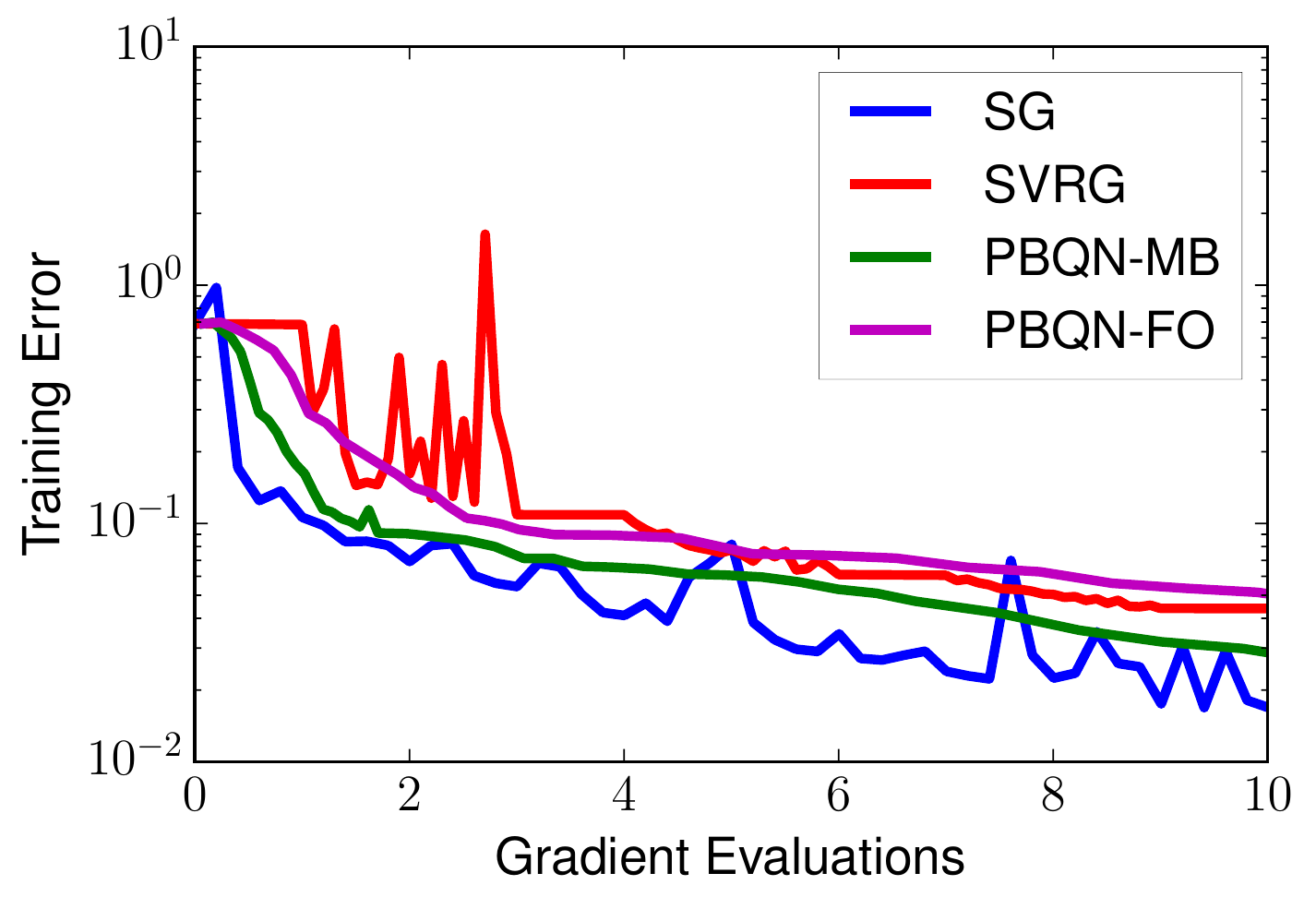}
		\includegraphics[width=0.33\linewidth]{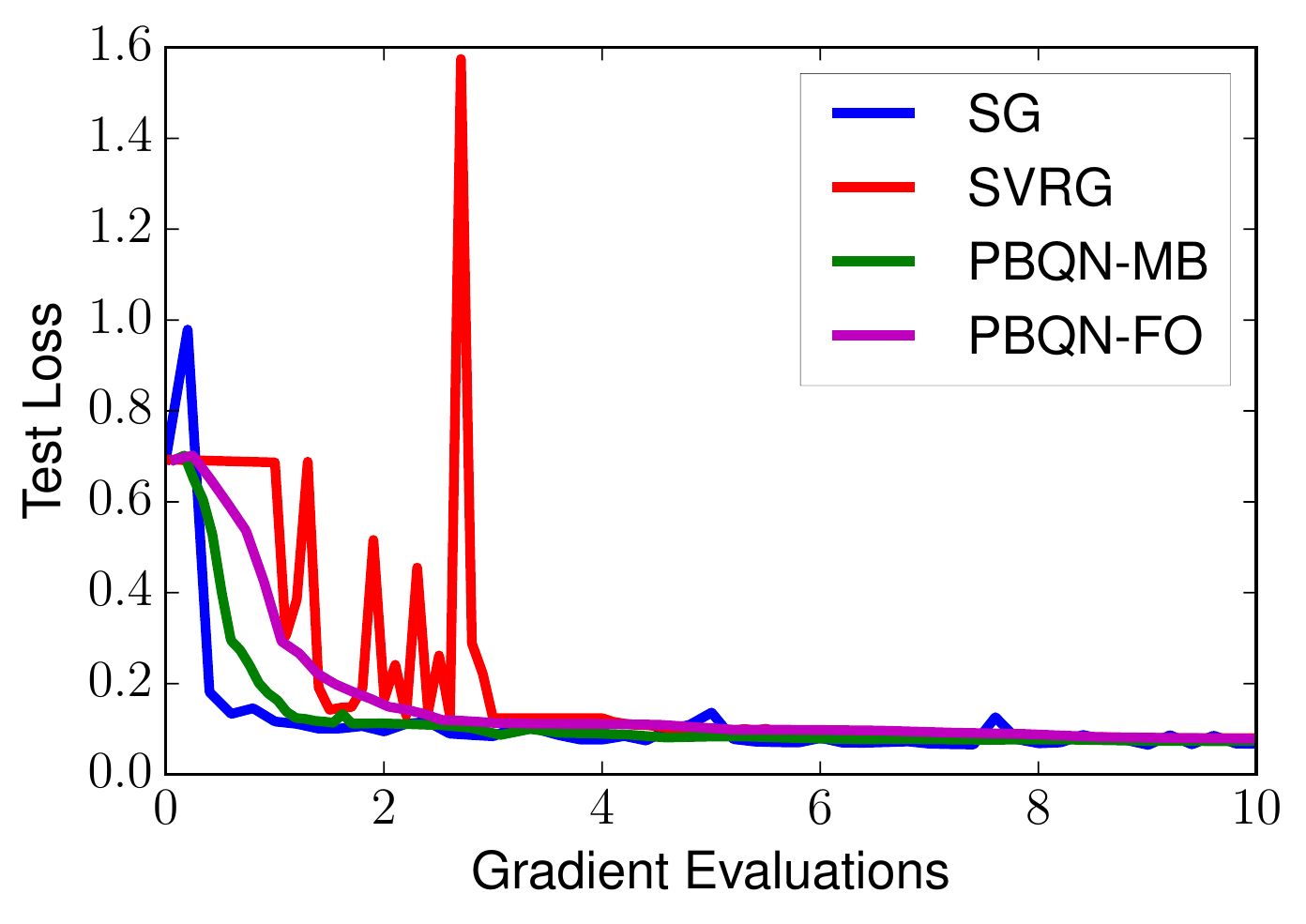}
		\includegraphics[width=0.33\linewidth]{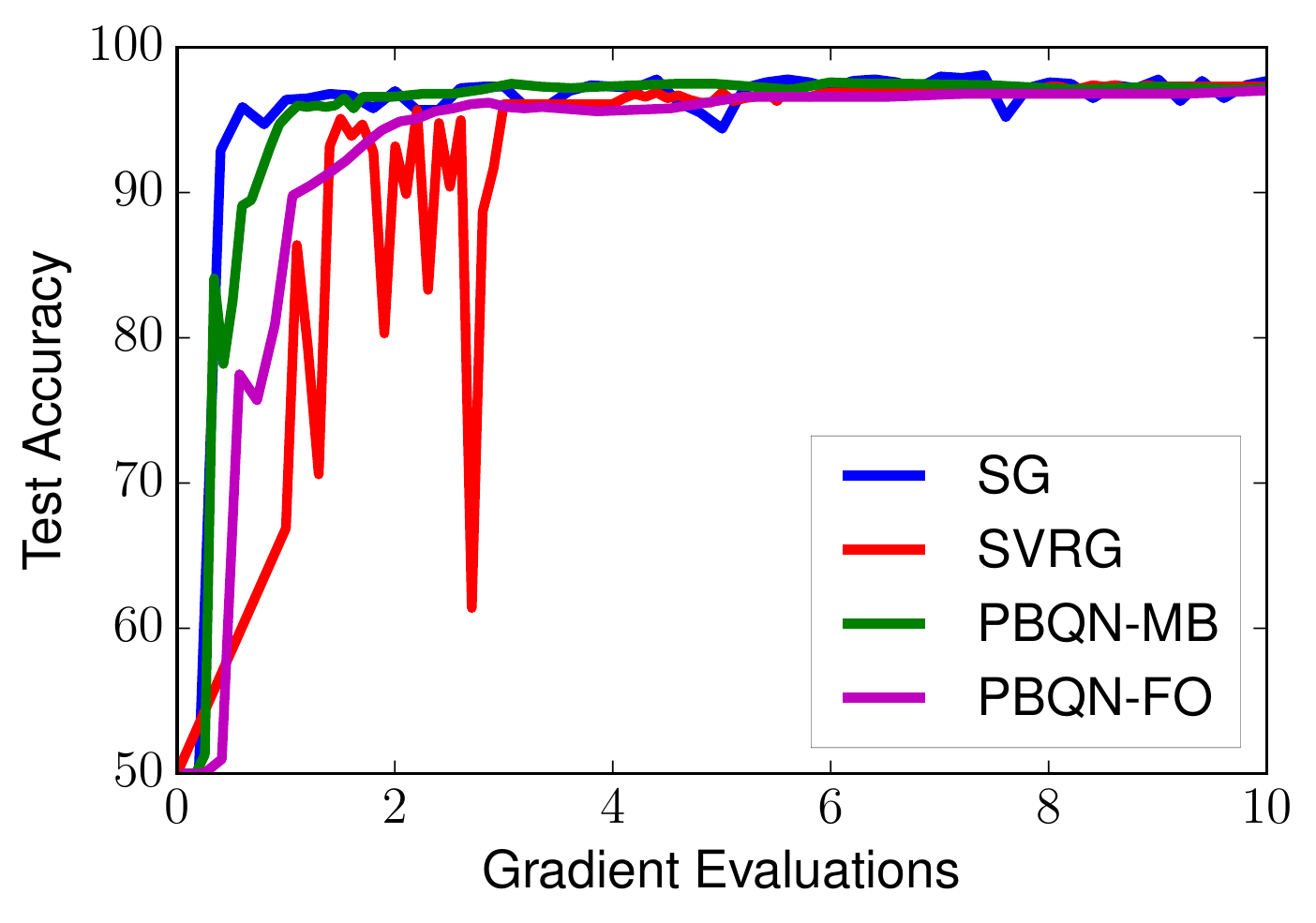}
		\includegraphics[width=0.33\linewidth]{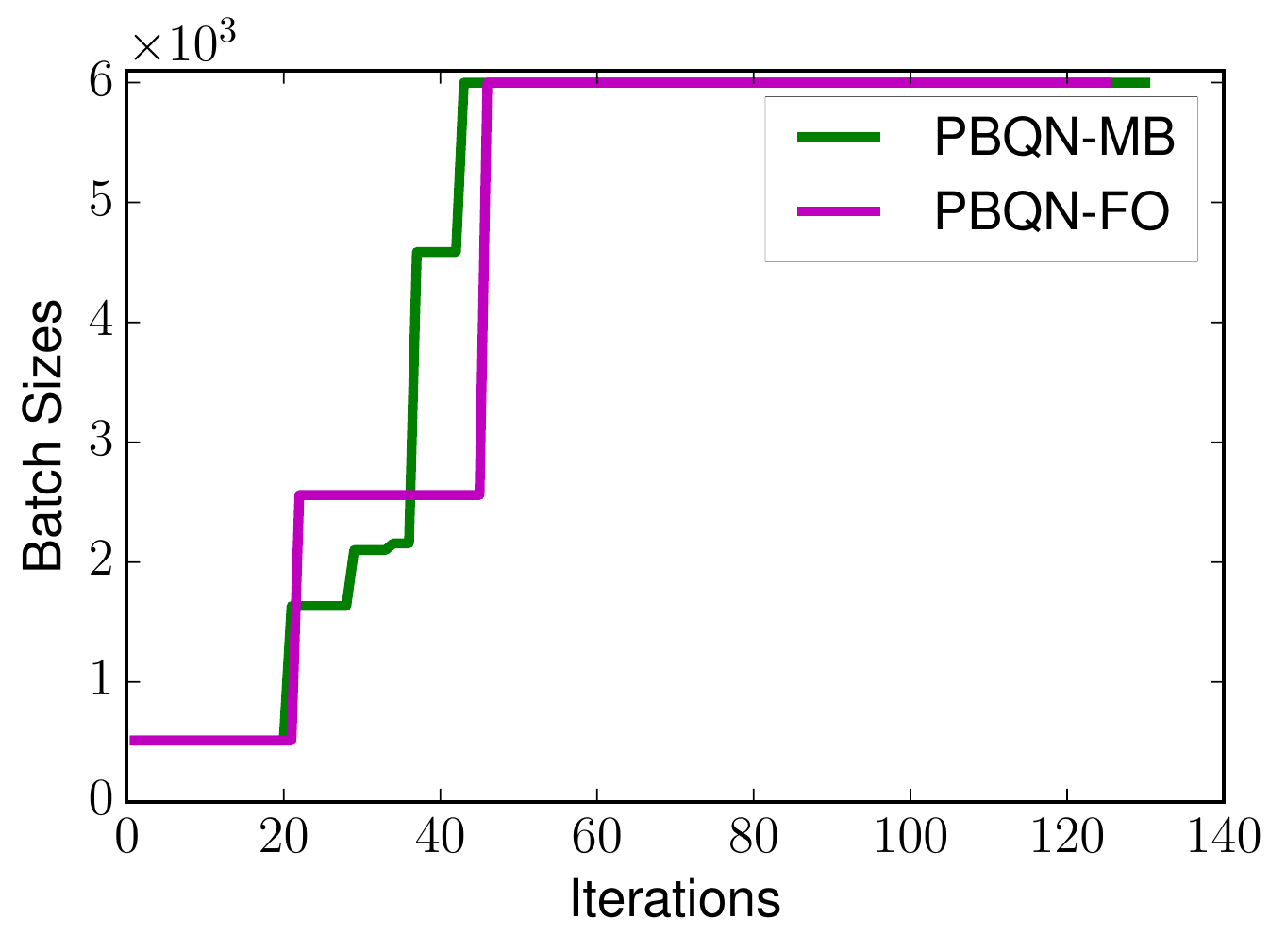}
		\includegraphics[width=0.33\linewidth]{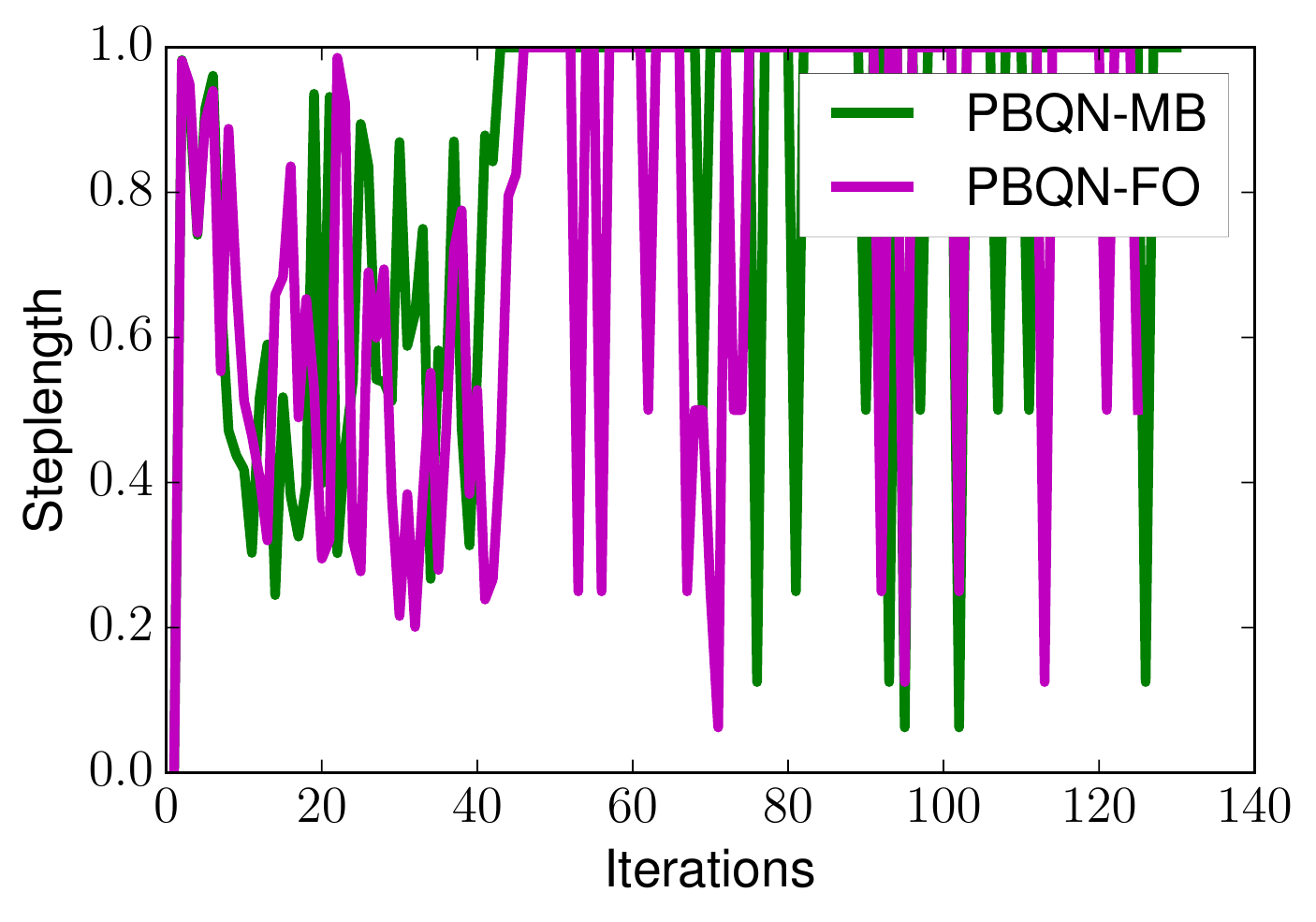}
		\par\end{centering}
	\caption{ \textbf{ gisette dataset:} Performance of the progressive batching L-BFGS methods, with multi-batch (MB) (25\% overlap) and full-overlap (FO) approaches, and the SG and SVRG methods.}
	\label{exp:gisette} 
\end{figure*}

\begin{figure*}[!htp]
	\begin{centering}
		\includegraphics[width=0.33\linewidth]{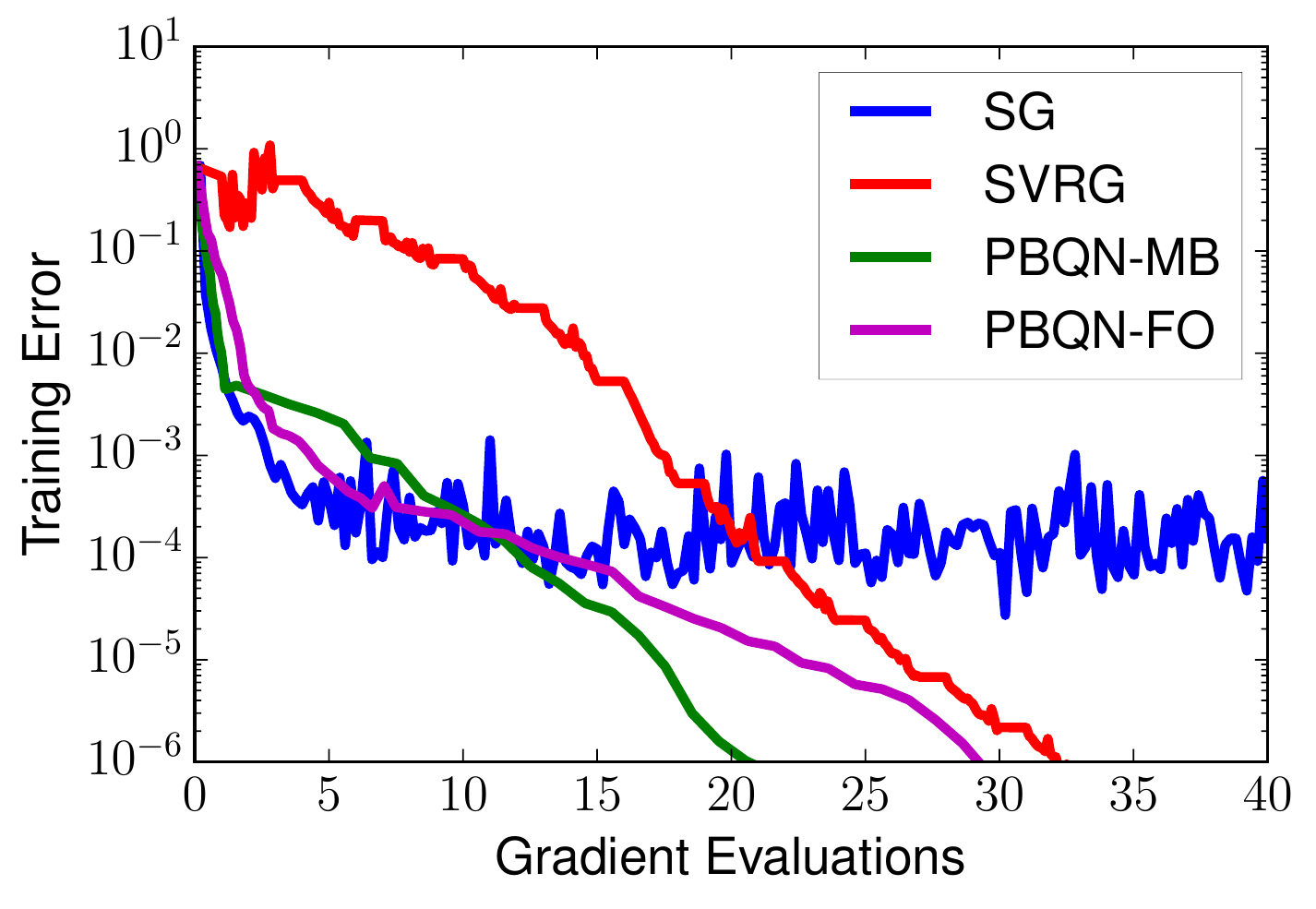}
		\includegraphics[width=0.33\linewidth]{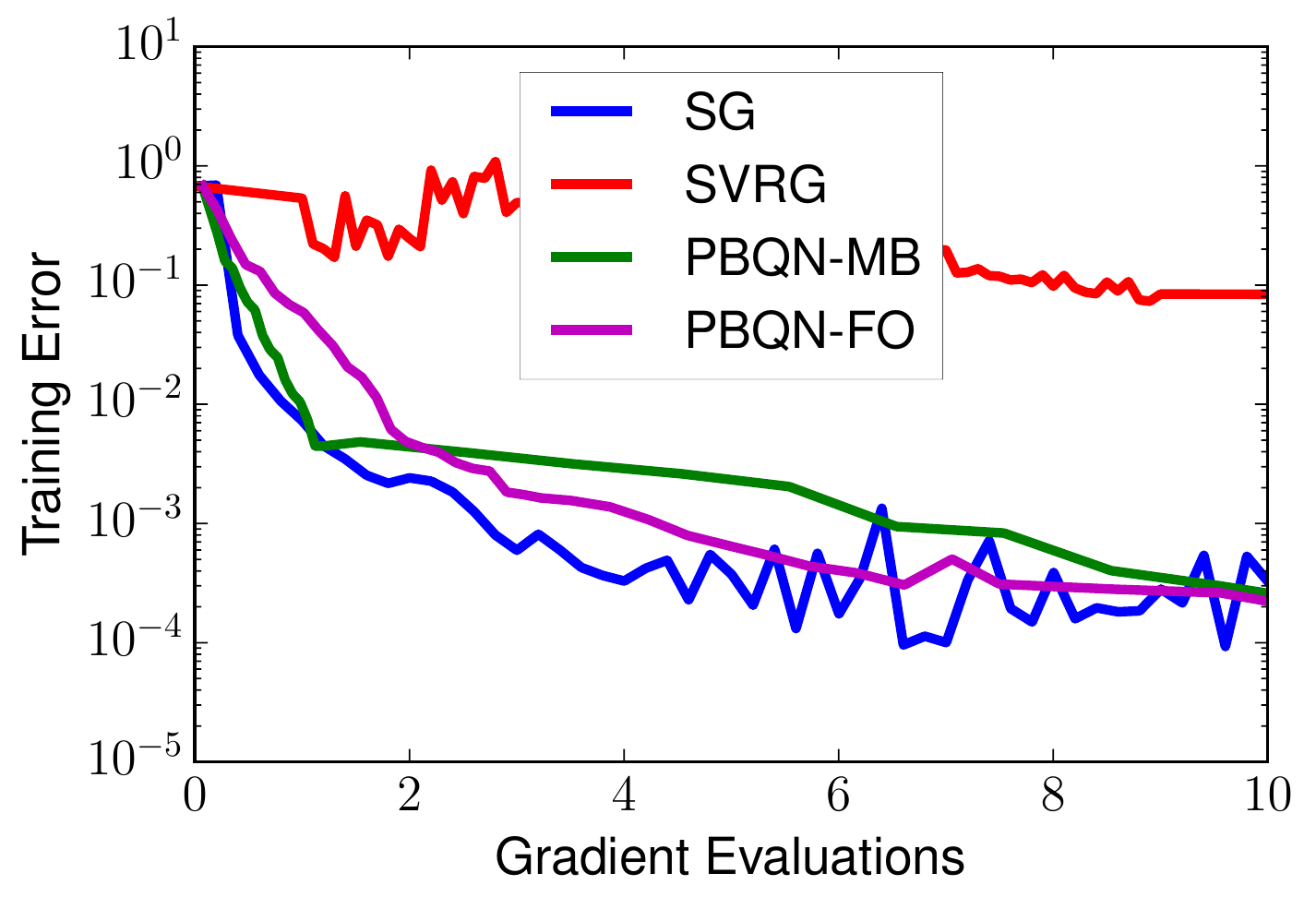}
		\includegraphics[width=0.33\linewidth]{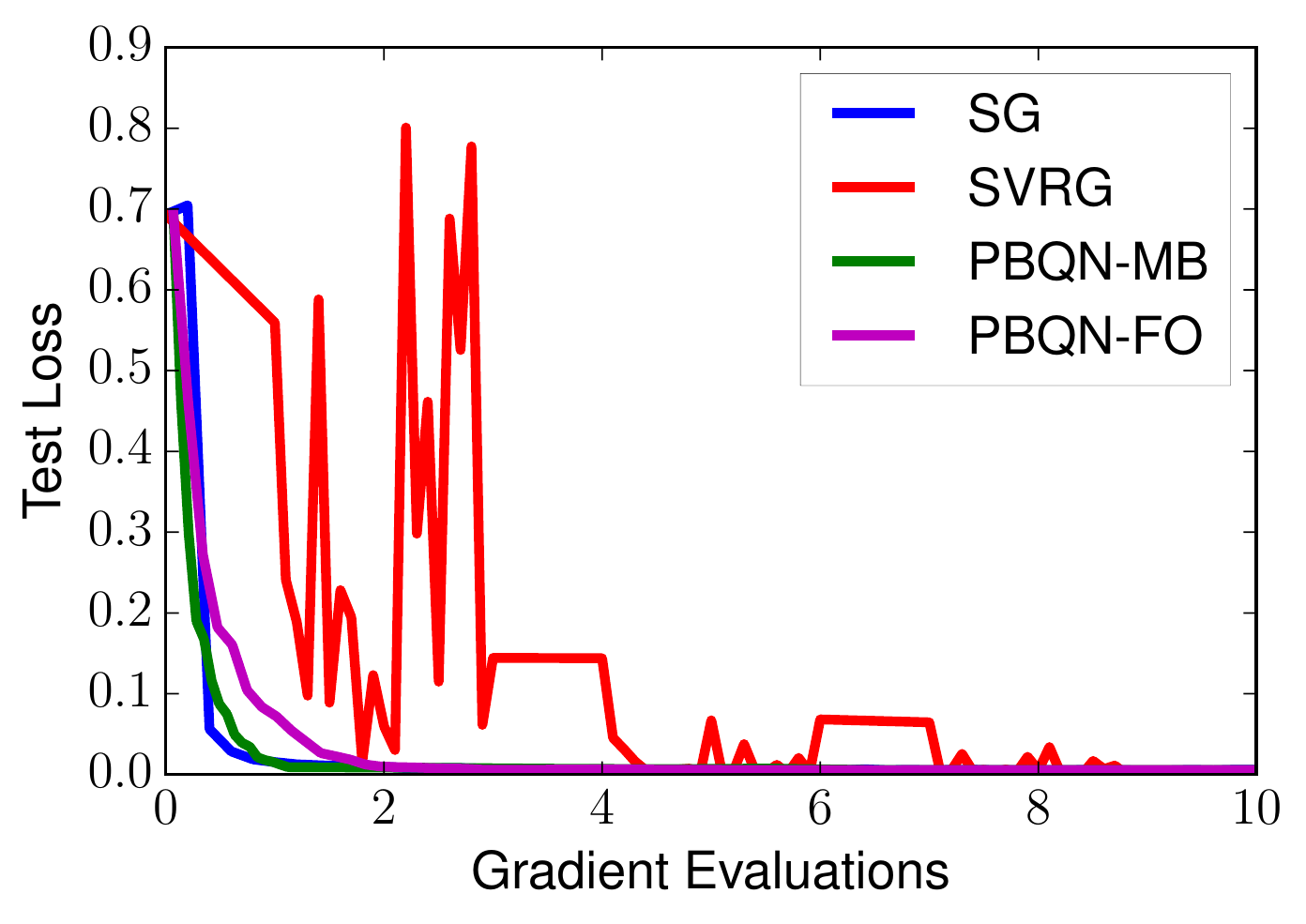}
		\includegraphics[width=0.33\linewidth]{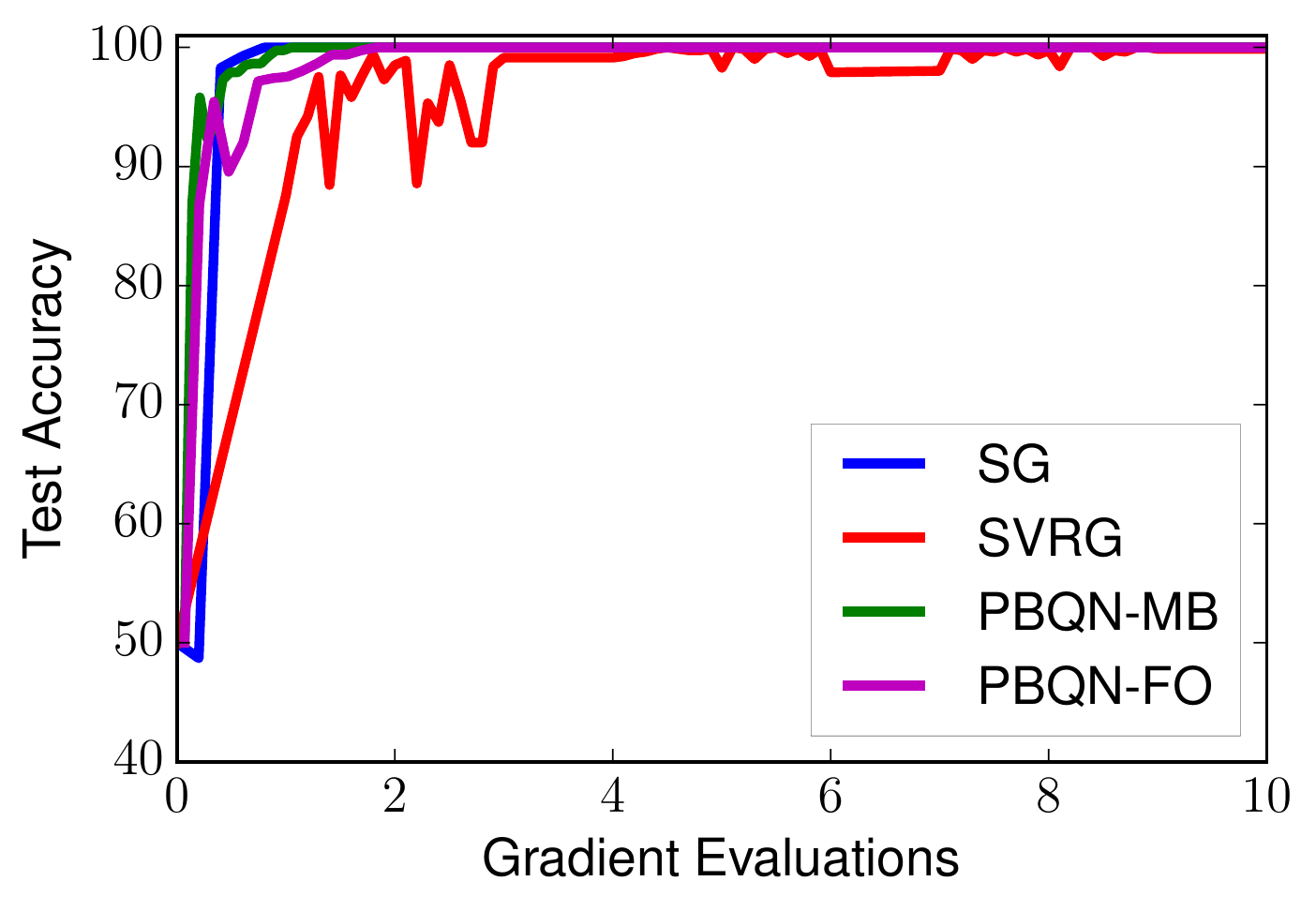}
		\includegraphics[width=0.33\linewidth]{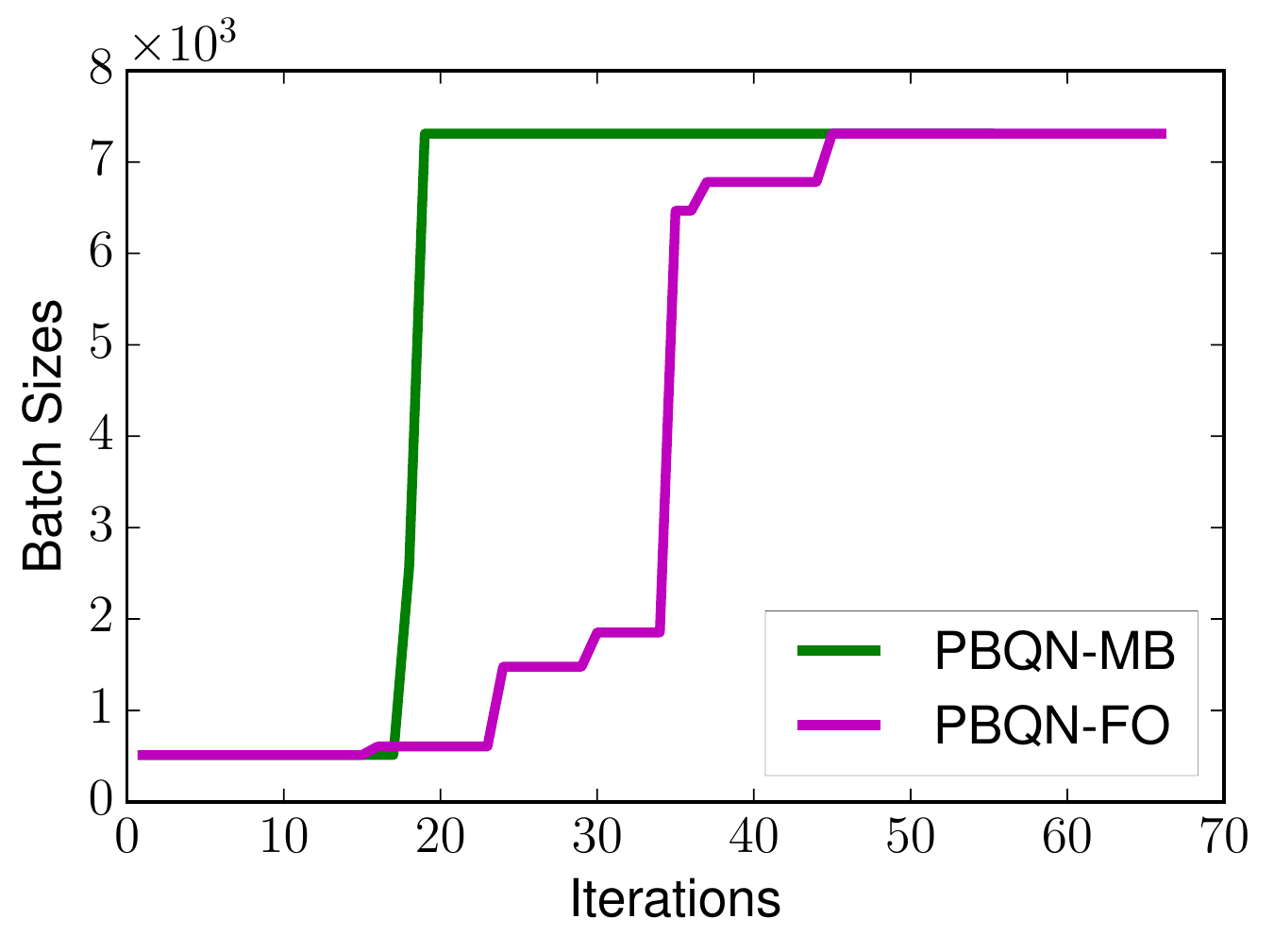}
		\includegraphics[width=0.33\linewidth]{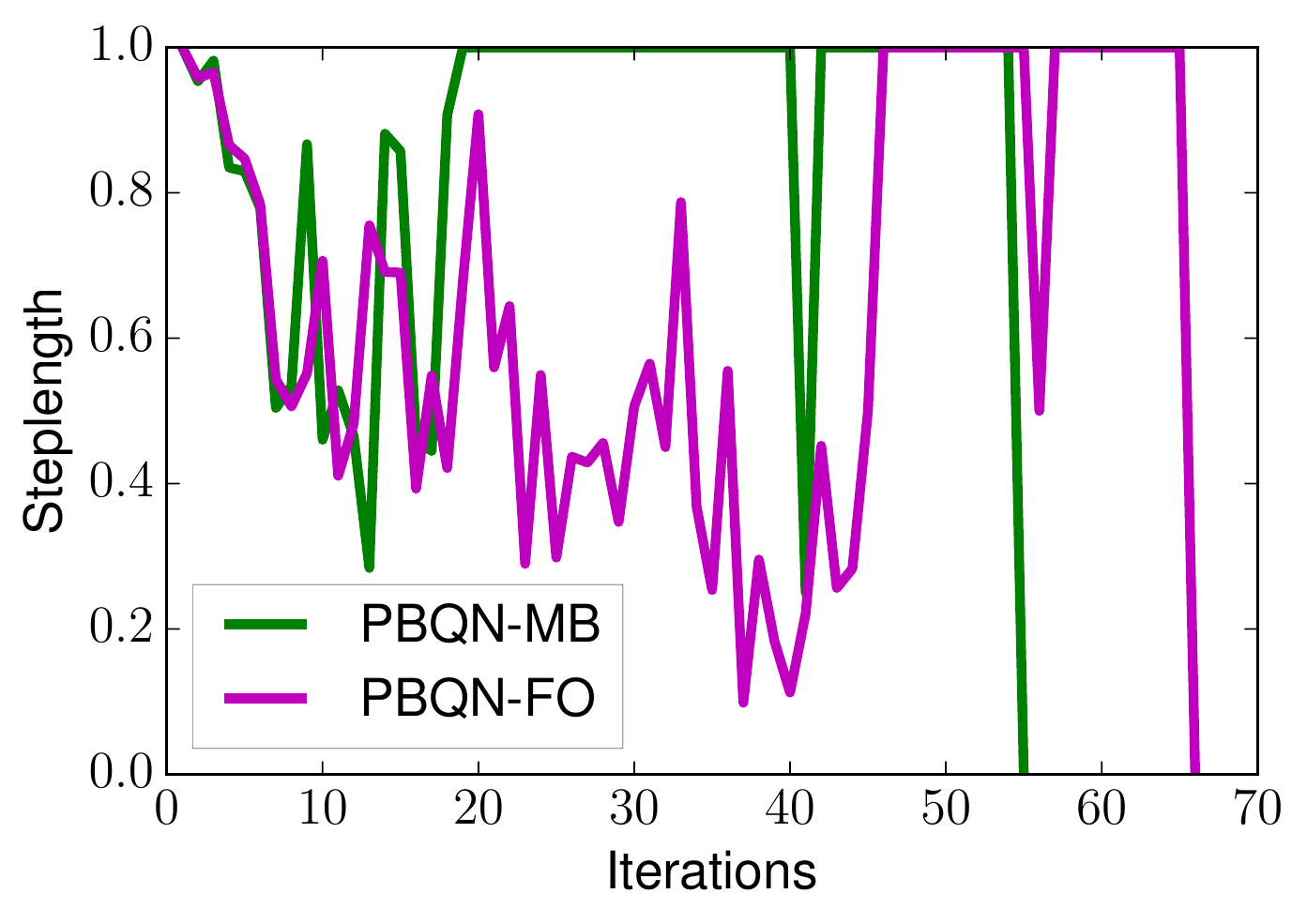}
		\par\end{centering}
	\caption{ \textbf{mushrooms dataset:} Performance of the progressive batching L-BFGS methods, with multi-batch (MB) (25\% overlap) and full-overlap (FO) approaches, and the SG and SVRG methods.}
	\label{exp:mushrooms} 
\end{figure*}

\begin{figure*}[!htp]
	\begin{centering}
		\includegraphics[width=0.33\linewidth]{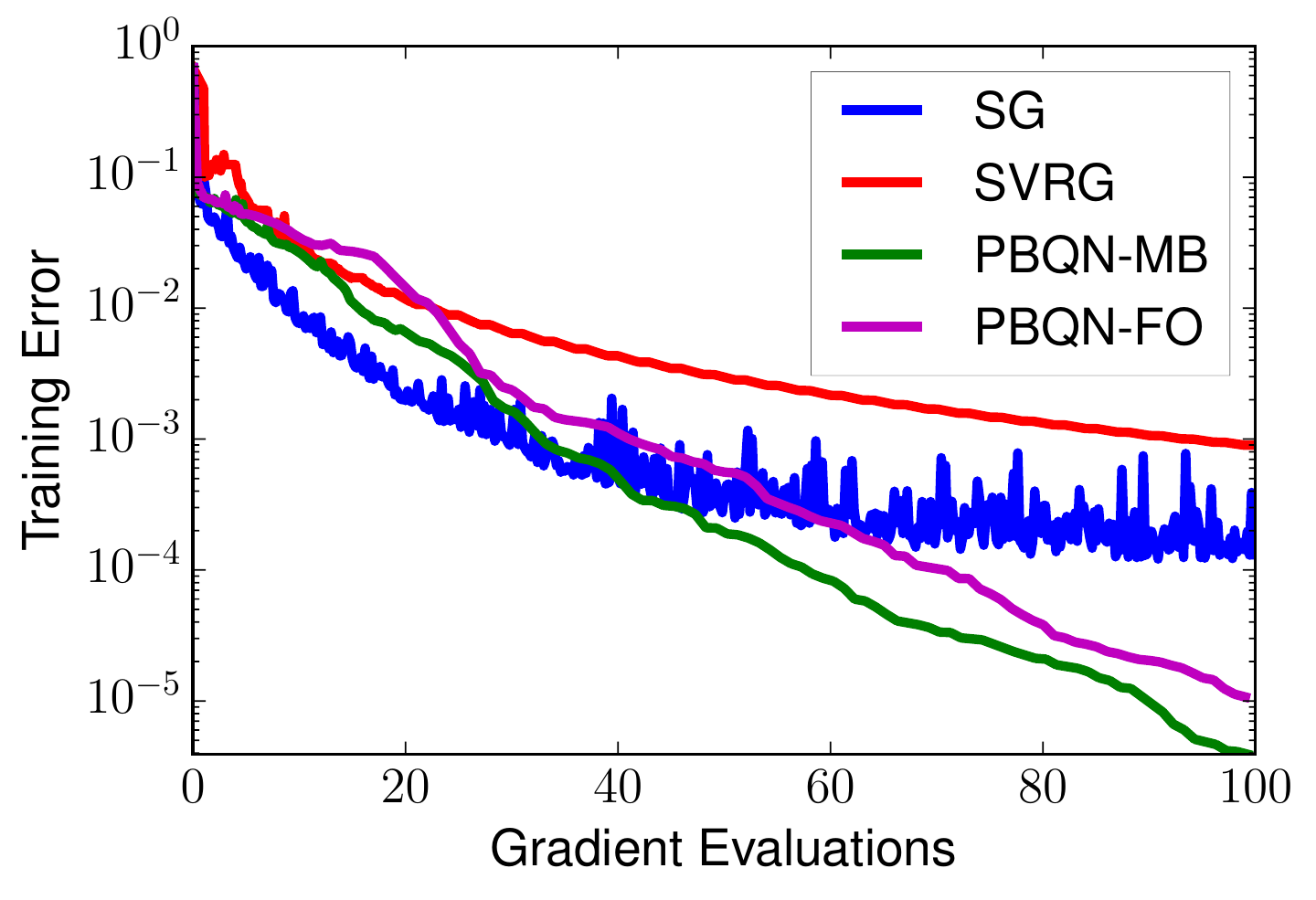}
		\includegraphics[width=0.33\linewidth]{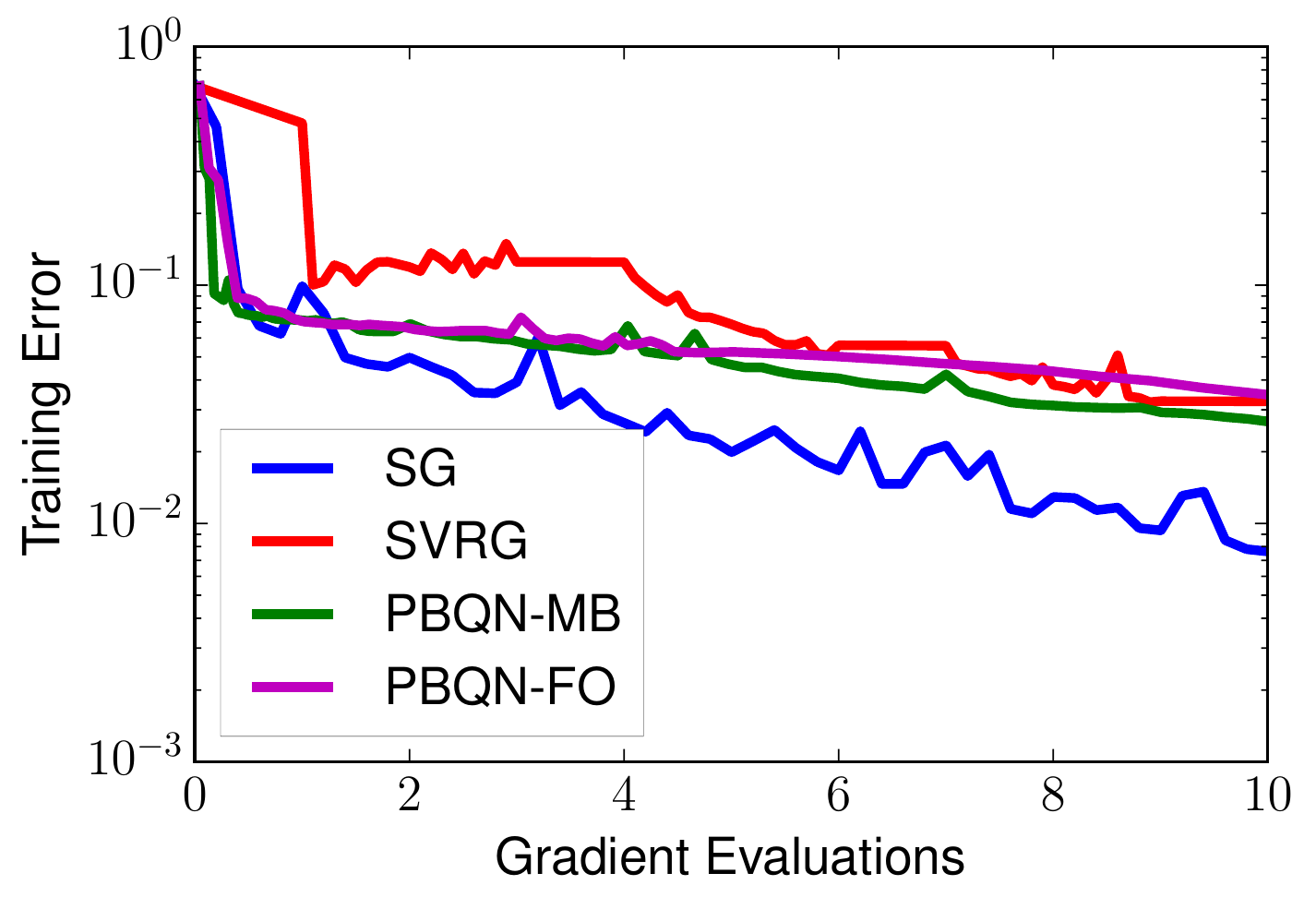}
		\includegraphics[width=0.33\linewidth]{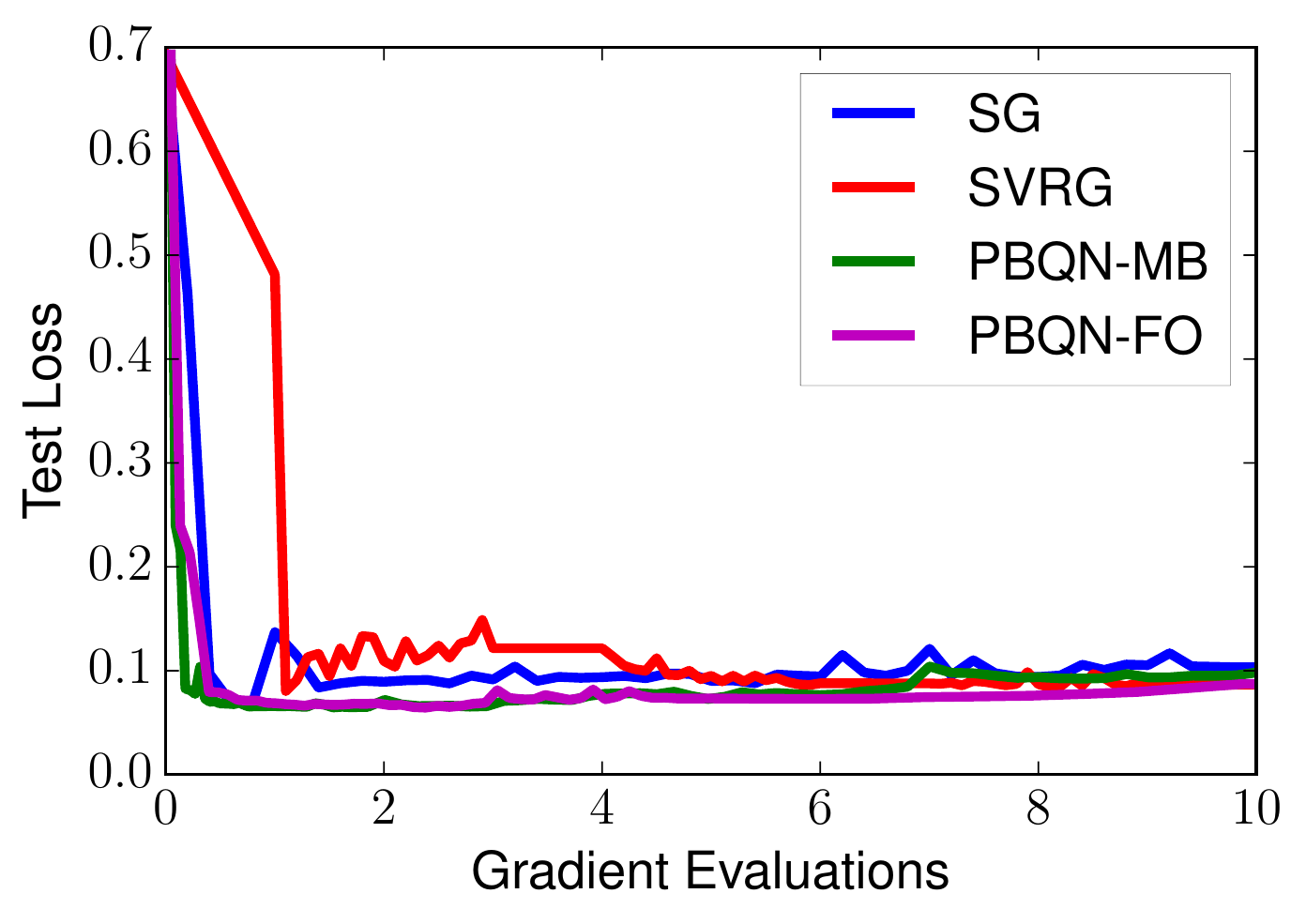}
		\includegraphics[width=0.33\linewidth]{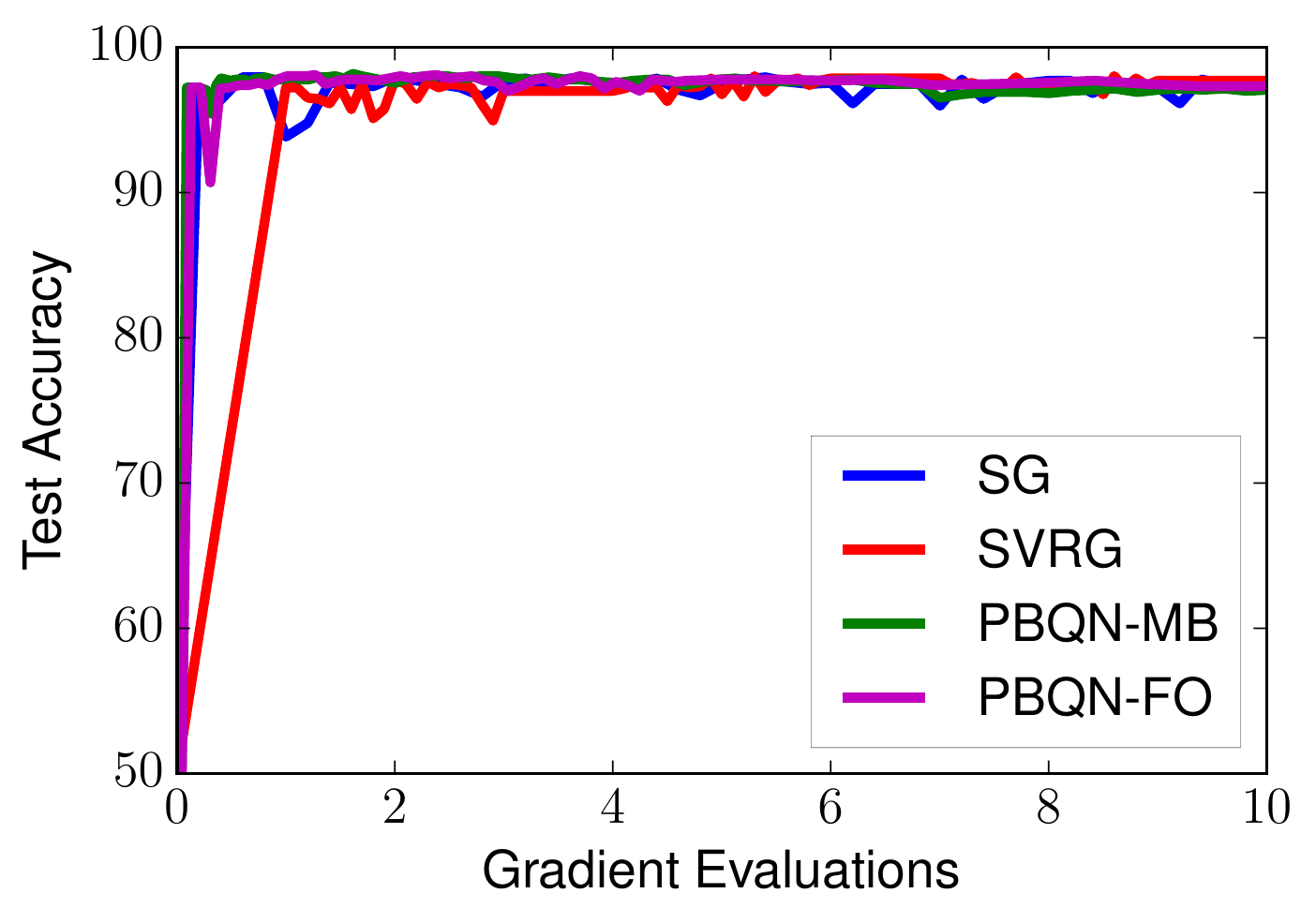}
		\includegraphics[width=0.33\linewidth]{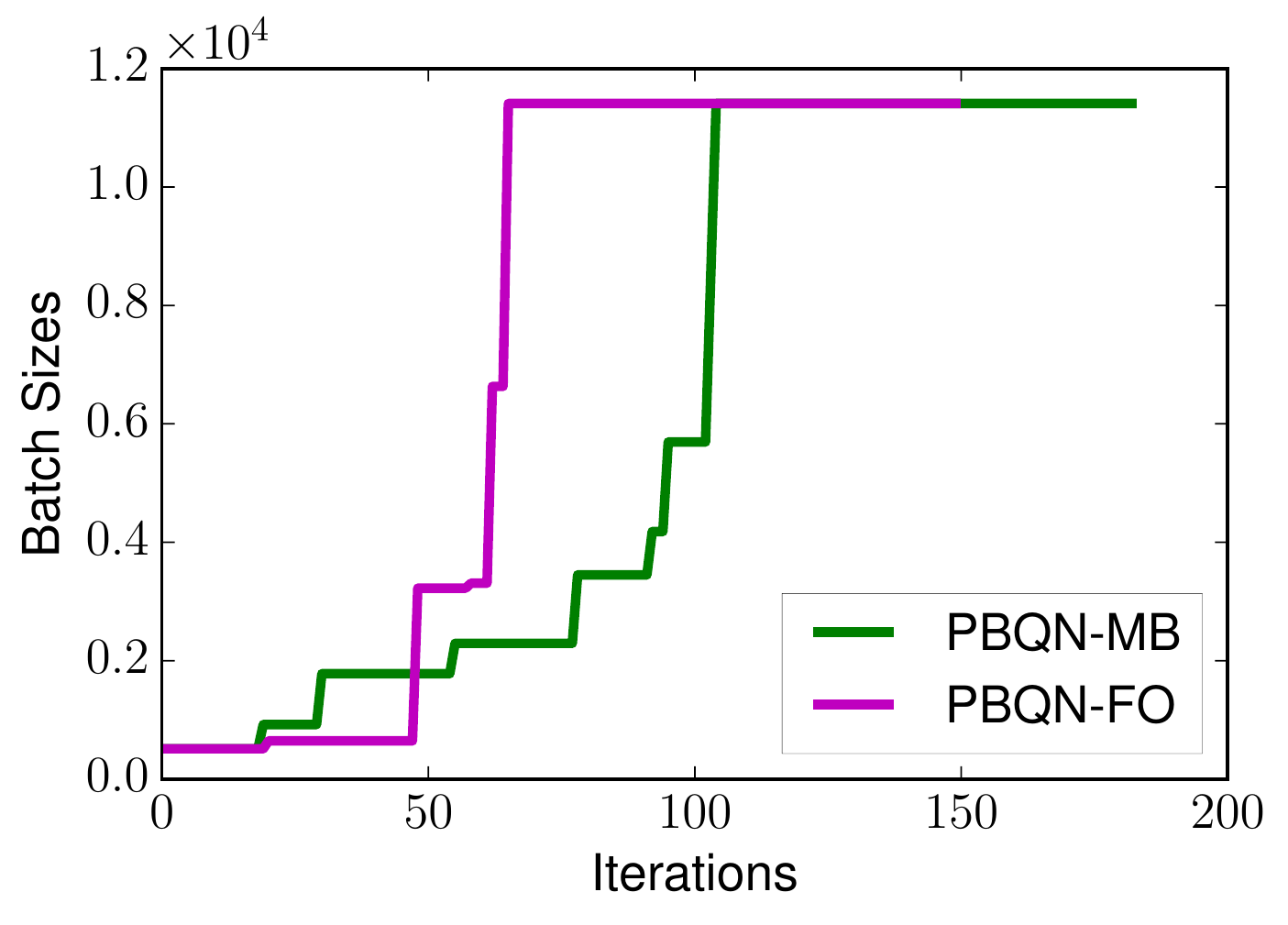}
		\includegraphics[width=0.33\linewidth]{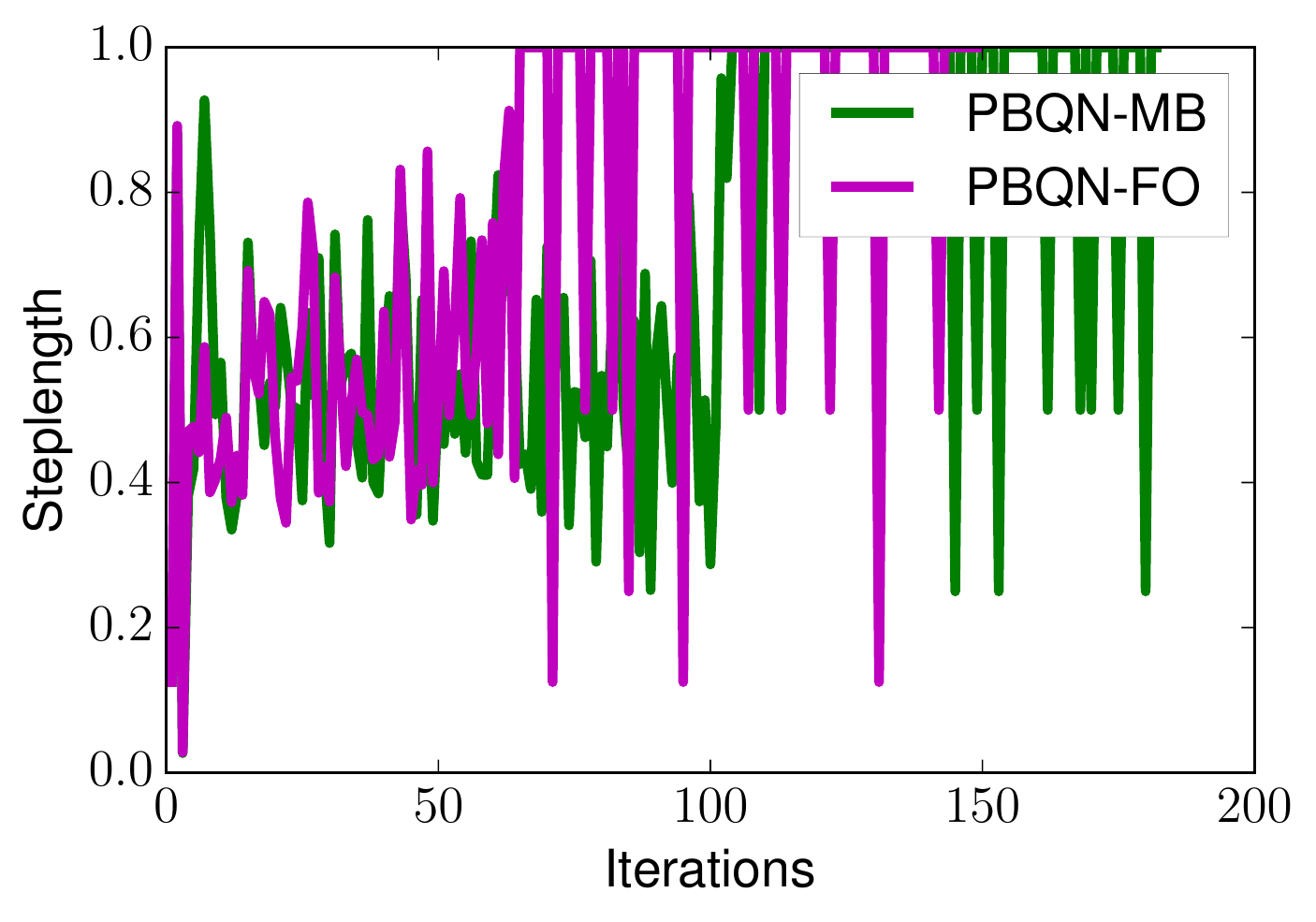}
		\par\end{centering}
	\caption{ \textbf{sido dataset:} Performance of the progressive batching L-BFGS methods, with multi-batch (MB) (25\% overlap) and full-overlap (FO) approaches, and the SG and SVRG methods.}
	\label{exp:sido} 
\end{figure*}

\begin{figure*}[!htp]
	\begin{centering}
		\includegraphics[width=0.33\linewidth]{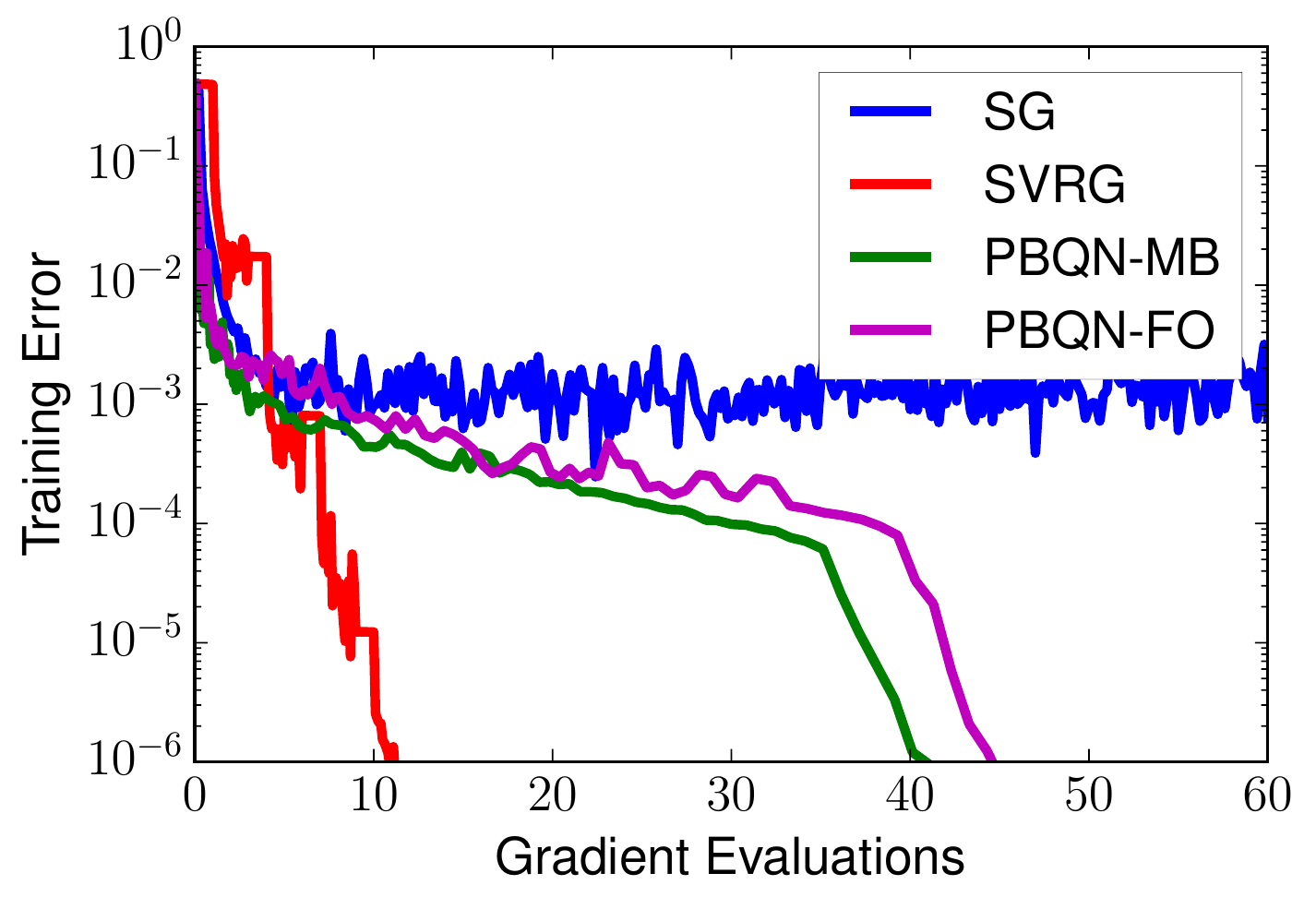}
		\includegraphics[width=0.33\linewidth]{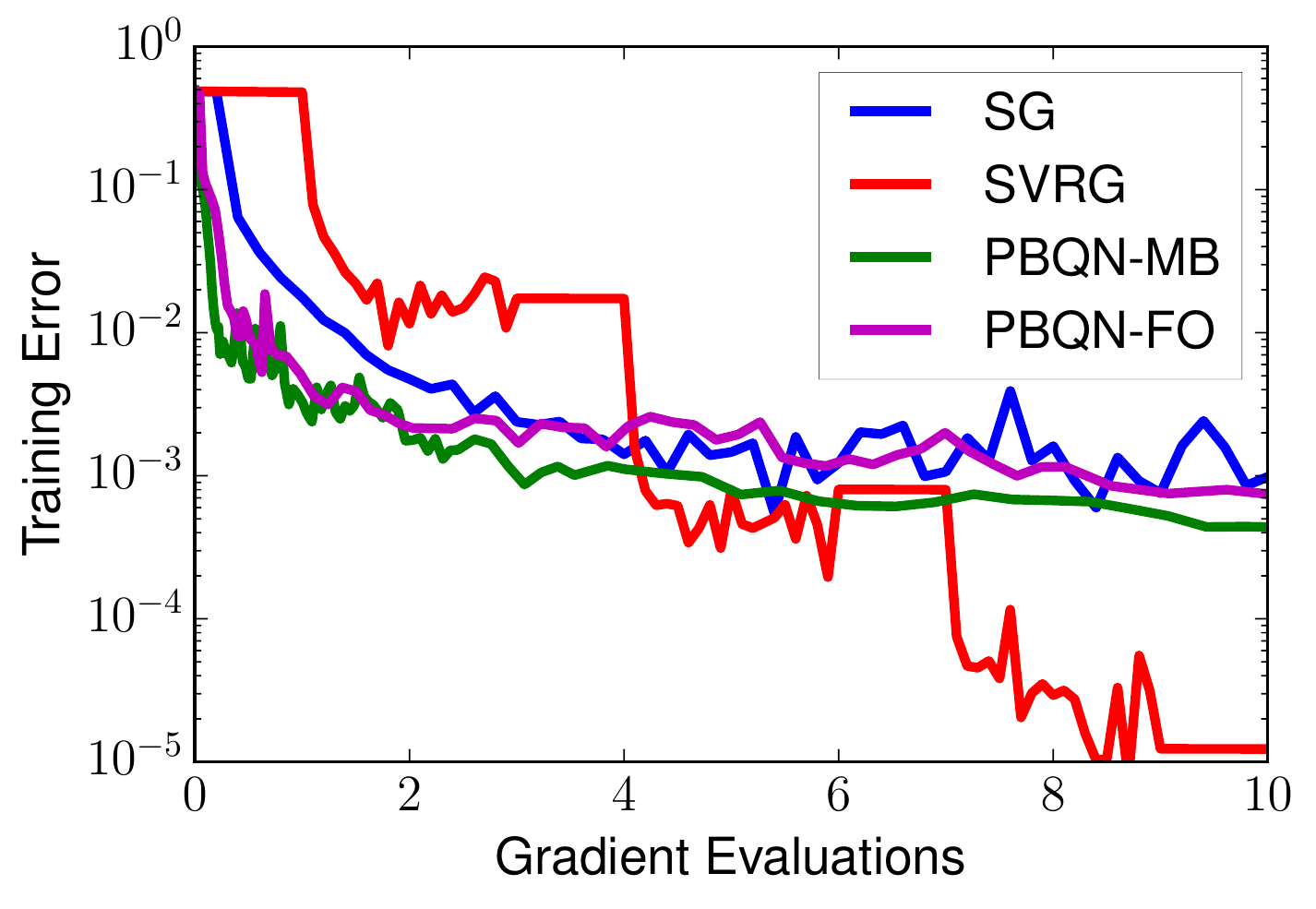}
		\includegraphics[width=0.33\linewidth]{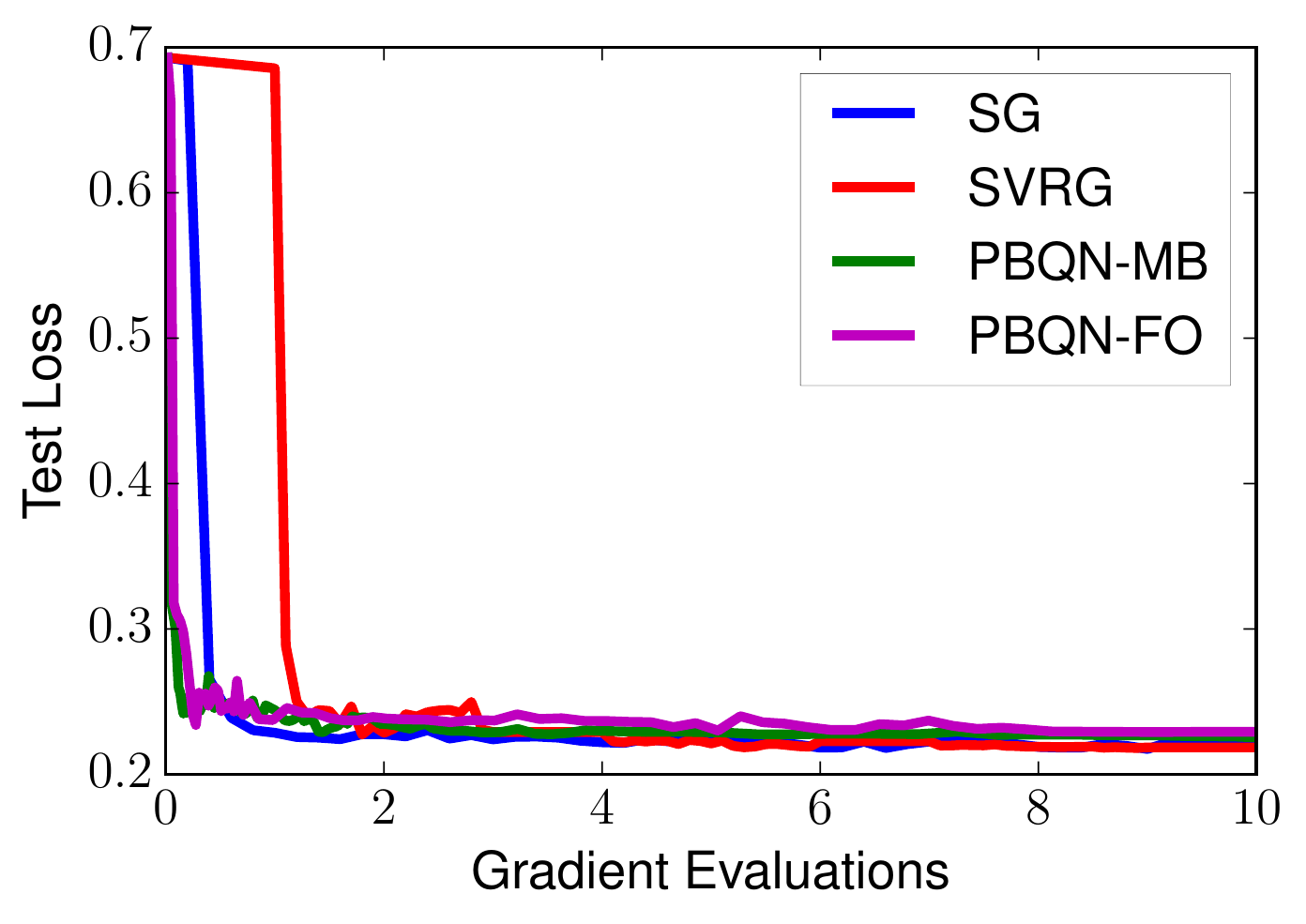}
		\includegraphics[width=0.33\linewidth]{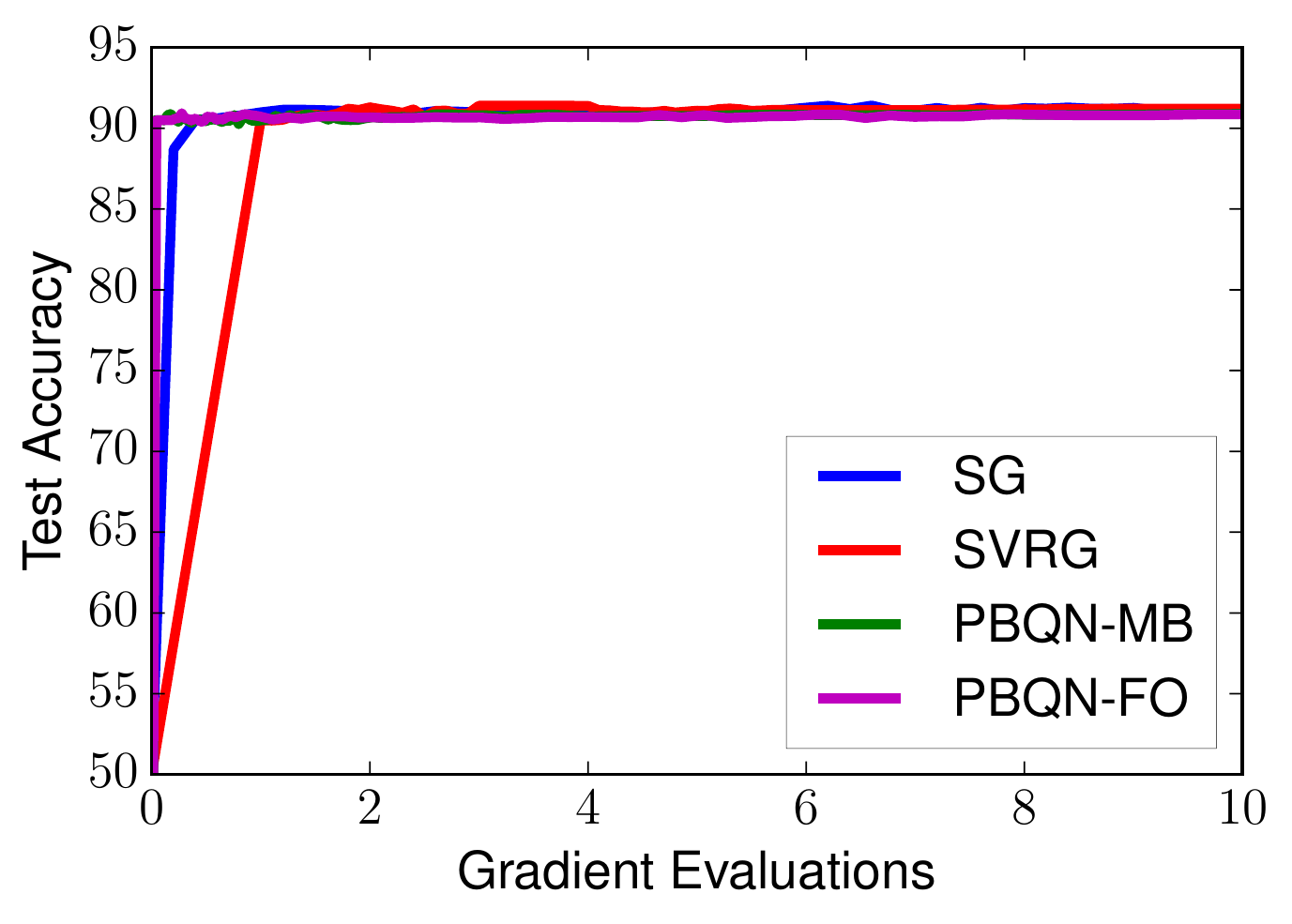}
		\includegraphics[width=0.33\linewidth]{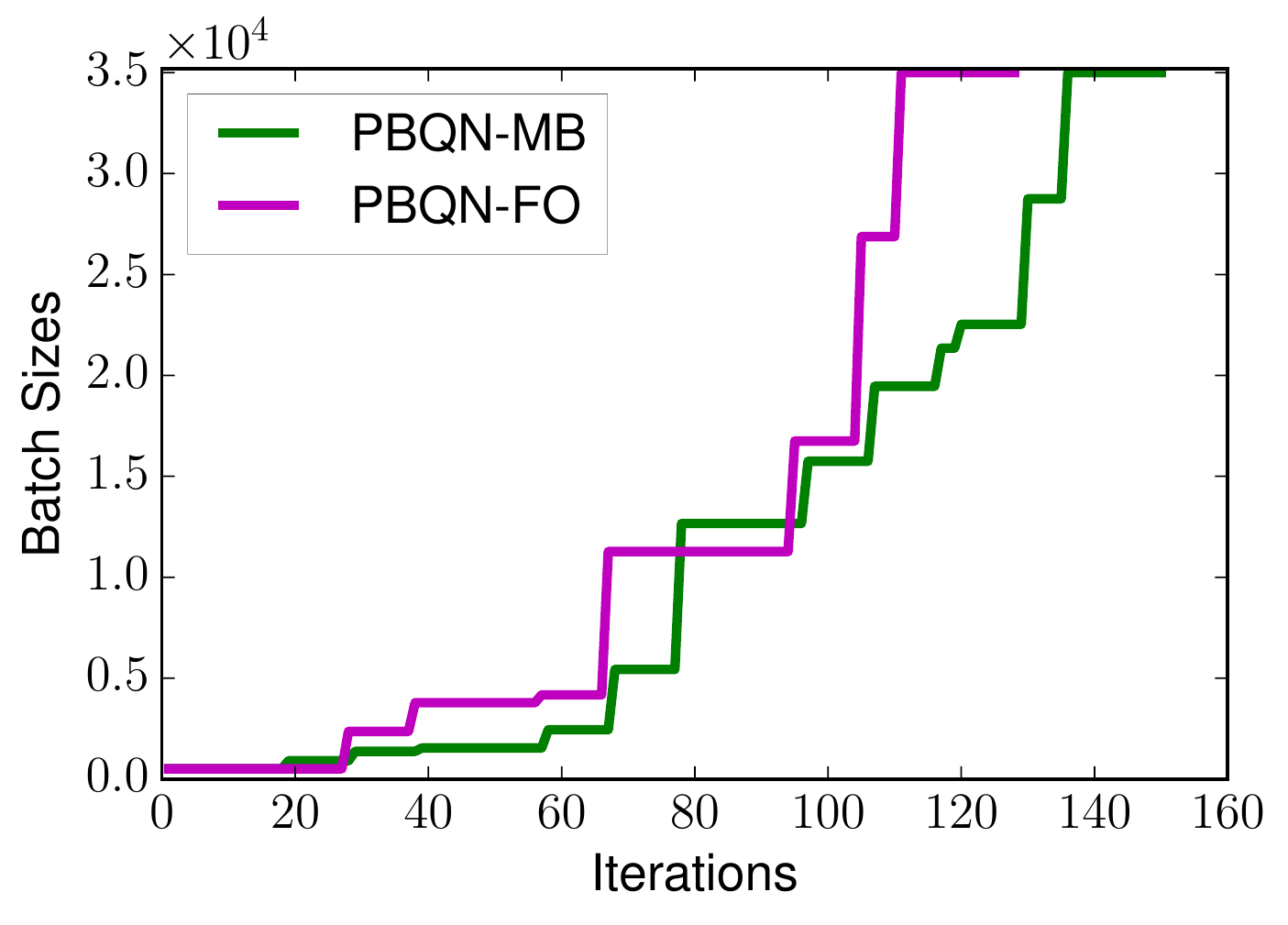}
		\includegraphics[width=0.33\linewidth]{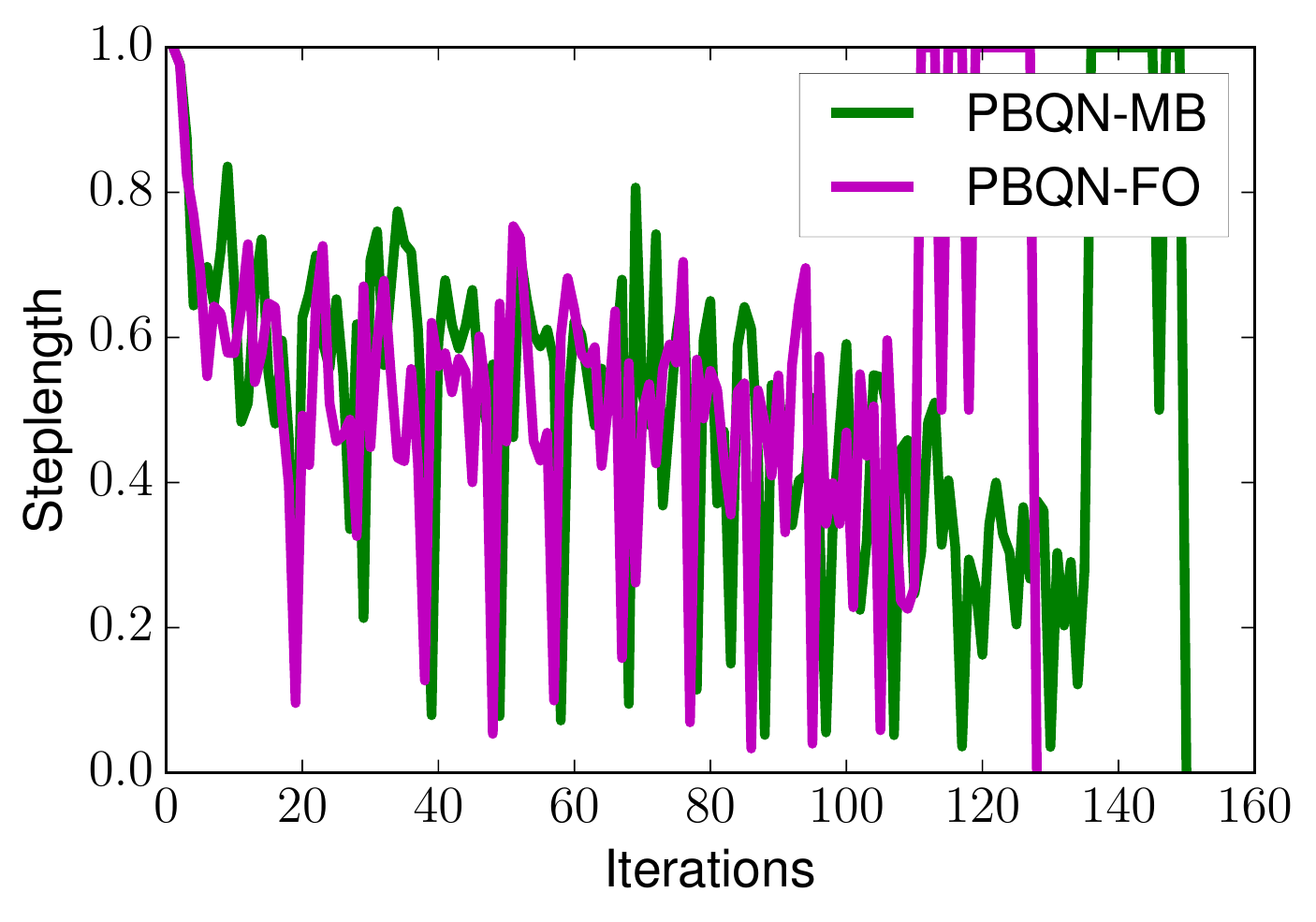}
		\par\end{centering}
	\caption{ \textbf{ijcnn dataset:} Performance of the progressive batching L-BFGS methods, with multi-batch (MB) (25\% overlap) and full-overlap(FO) approaches, and the SG and SVRG methods.}
	\label{exp:ijcnn} 
\end{figure*}

\begin{figure*}[!htp]
	\begin{centering}
		\includegraphics[width=0.33\linewidth]{trec2005_gradevals_train_loss_wide.pdf}
		\includegraphics[width=0.33\linewidth]{trec2005_gradevals_train_loss_tight.pdf}
		\includegraphics[width=0.33\linewidth]{trec2005_gradevals_test_loss_tight.pdf}
		\includegraphics[width=0.33\linewidth]{trec2005_gradevals_test_acc_tight.pdf}
		\includegraphics[width=0.33\linewidth]{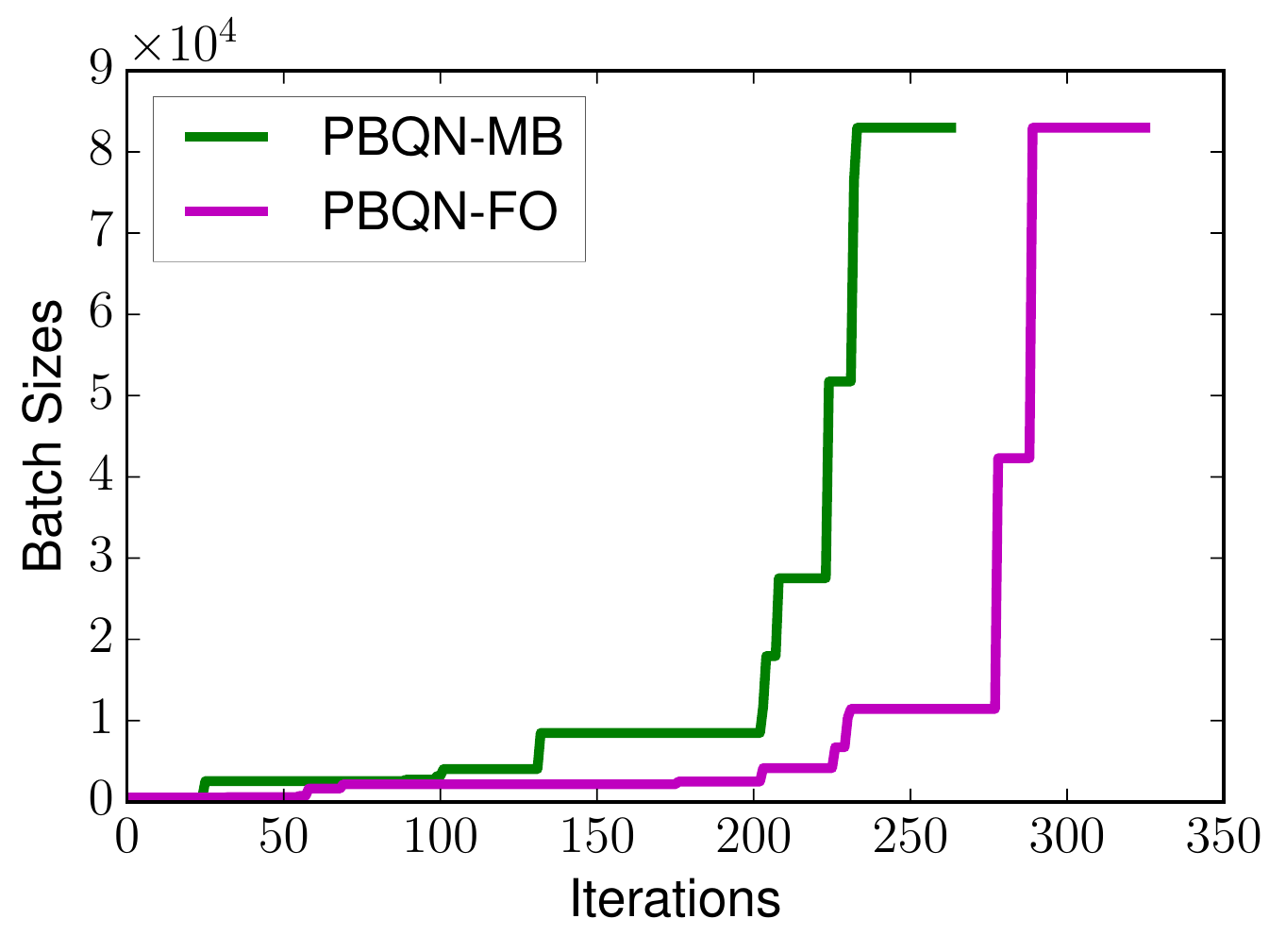}
		\includegraphics[width=0.33\linewidth]{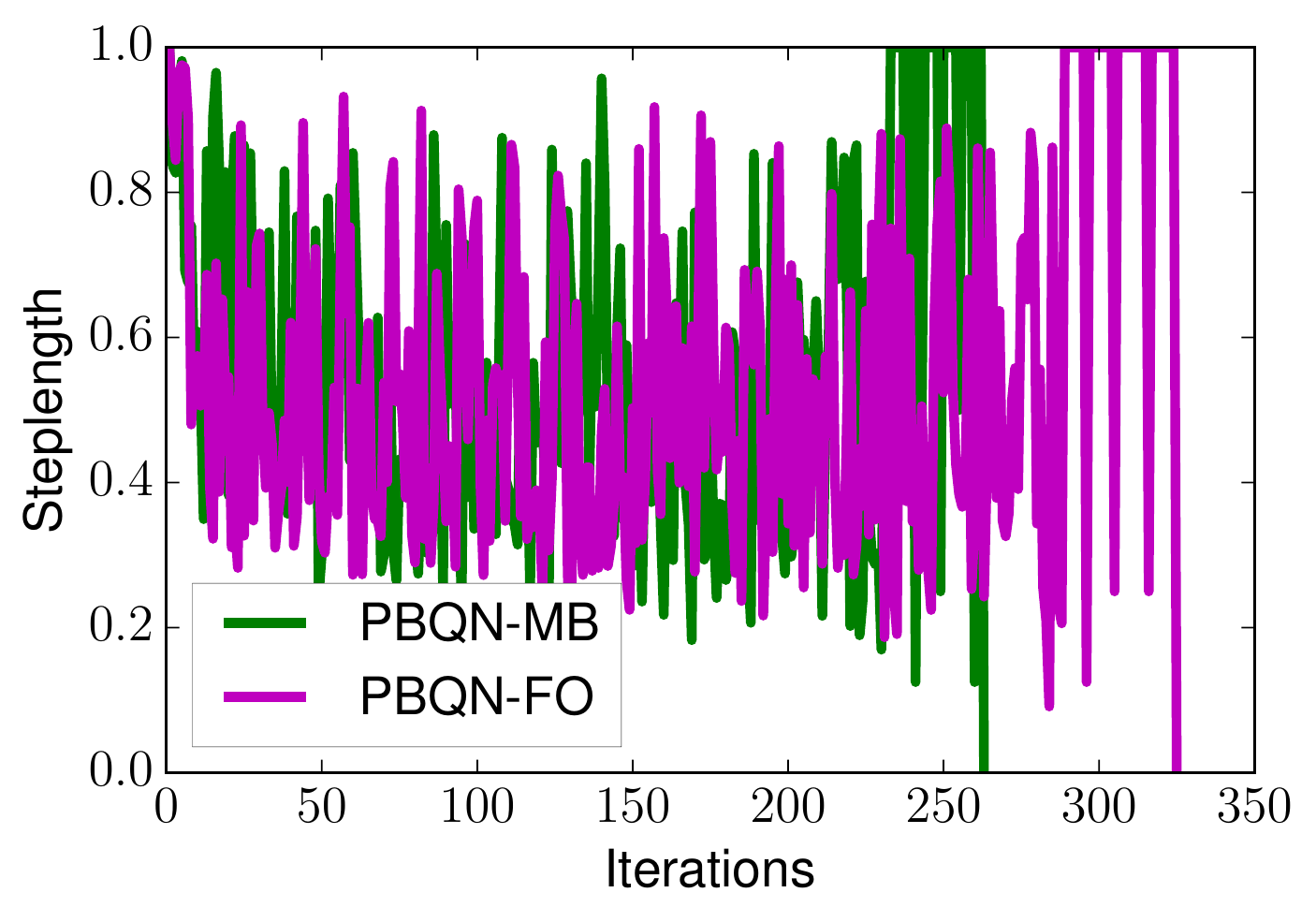}
		\par\end{centering}
	\caption{ \textbf{spam dataset:} Performance of the progressive batching L-BFGS methods, with multi-batch (MB) (25\% overlap) and full-overlap (FO) approaches, and the SG and SVRG methods.}
	\label{exp:spam} 
\end{figure*}

\begin{figure*}[!htp]
	\begin{centering}
		\includegraphics[width=0.33\linewidth]{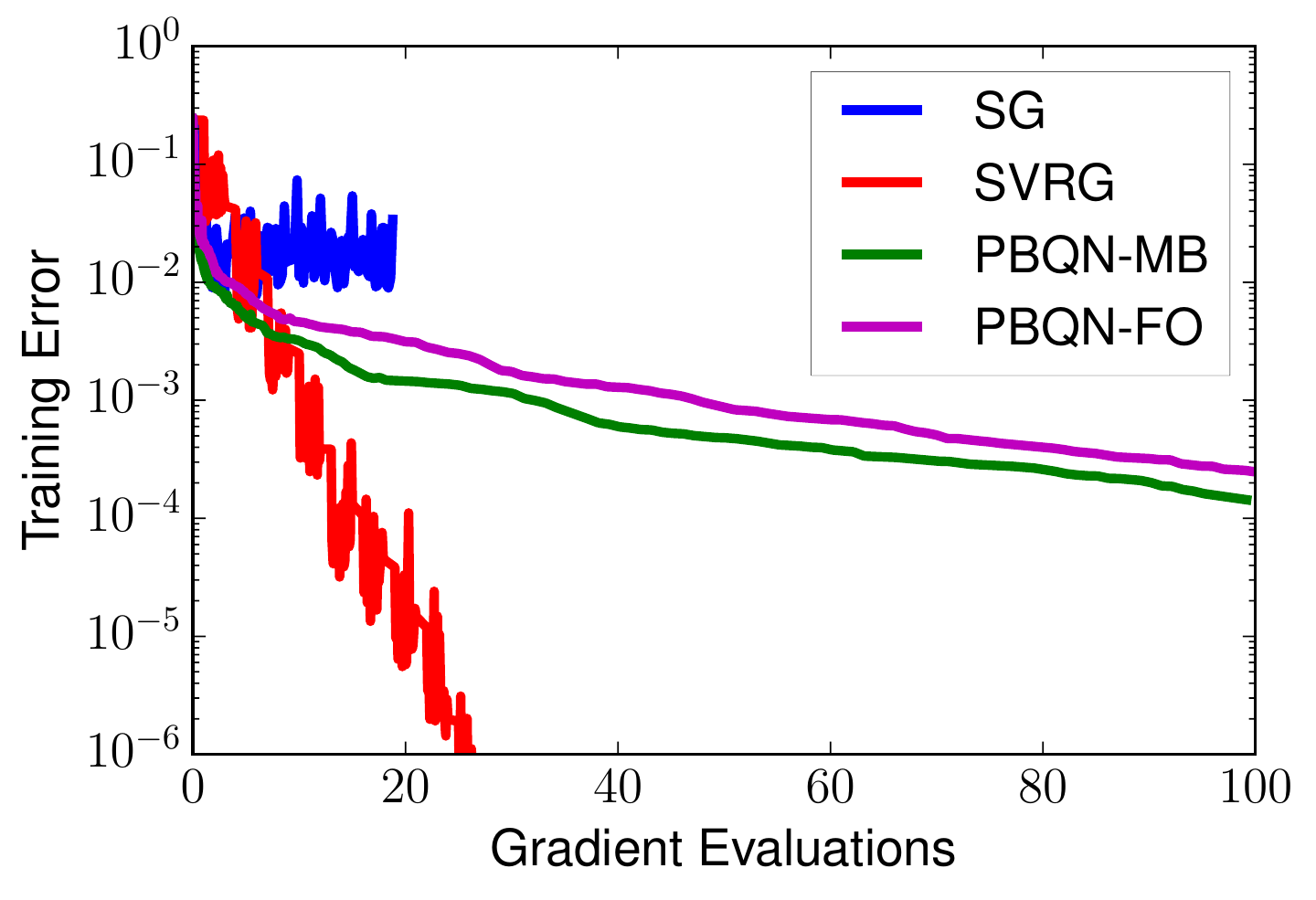}
		\includegraphics[width=0.33\linewidth]{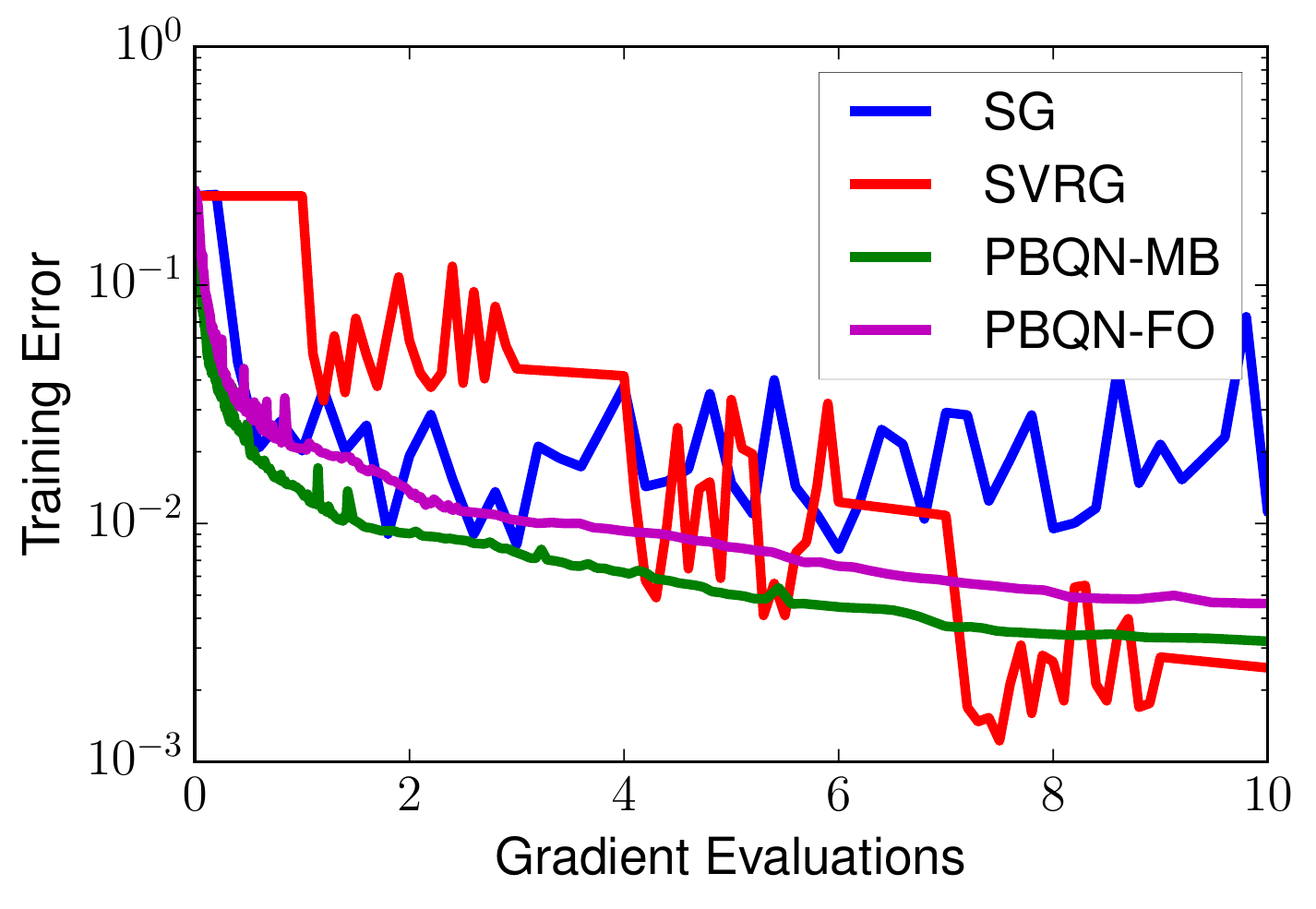}
		\includegraphics[width=0.33\linewidth]{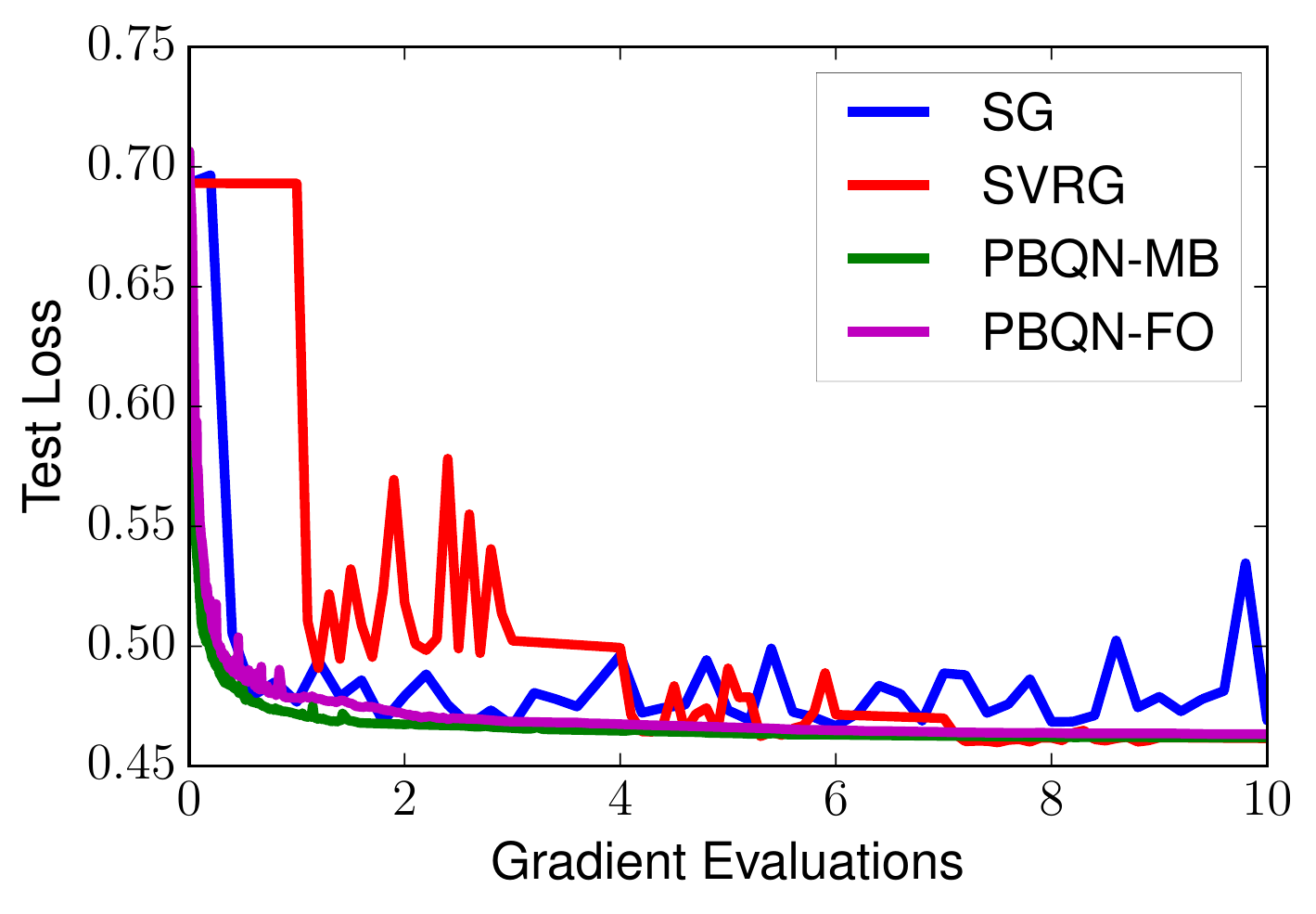}
		\includegraphics[width=0.33\linewidth]{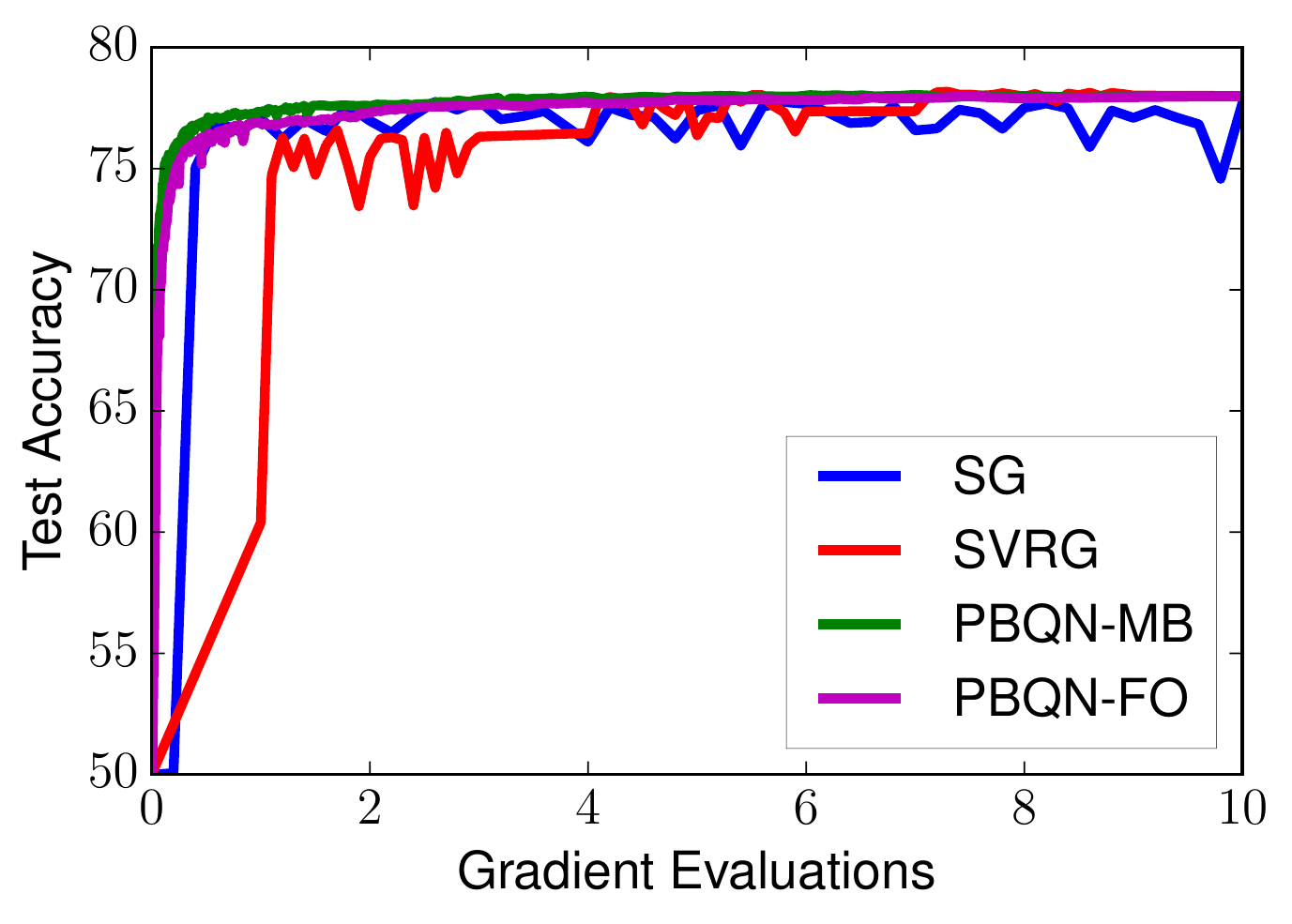}
		\includegraphics[width=0.33\linewidth]{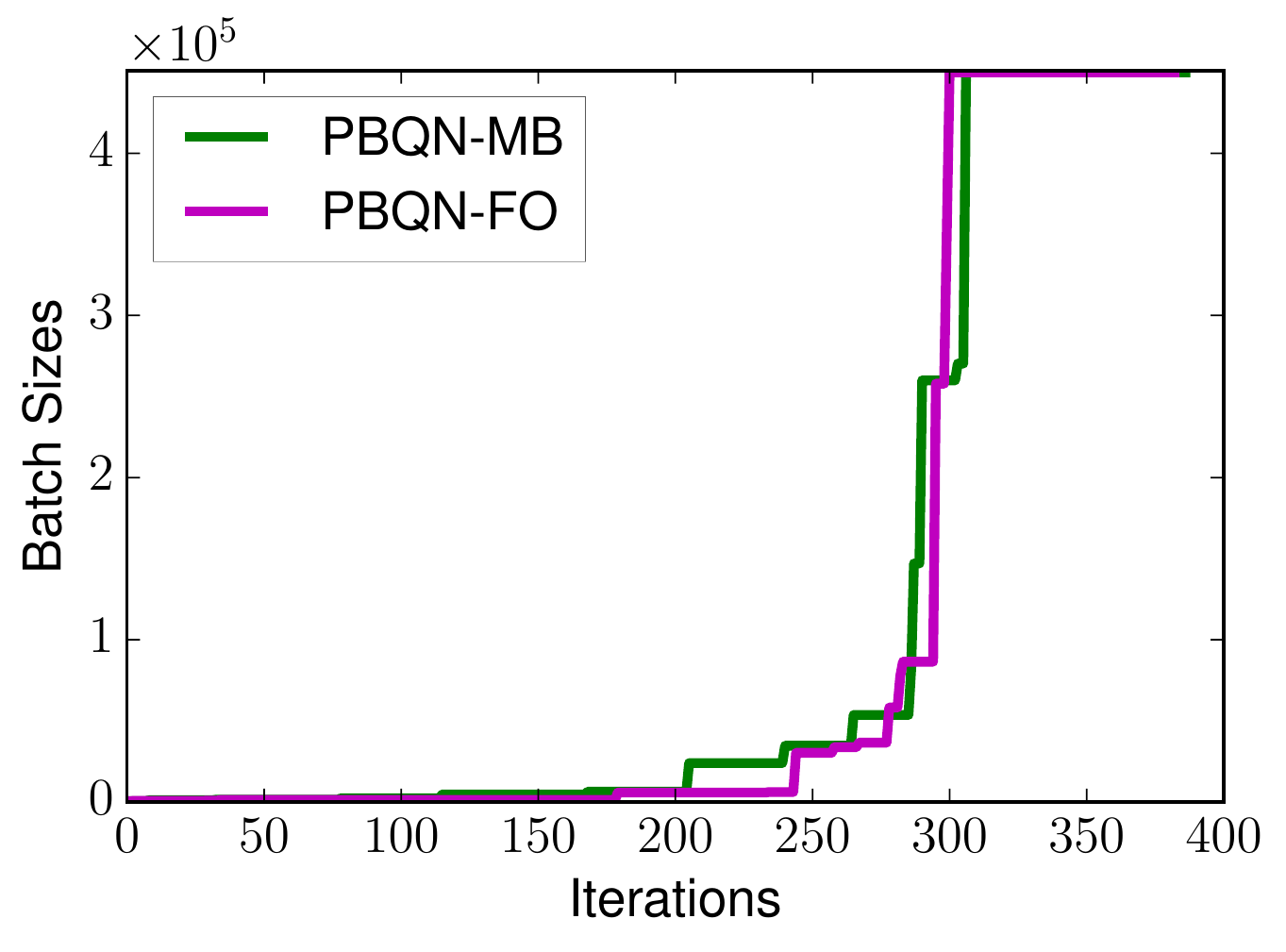}
		\includegraphics[width=0.33\linewidth]{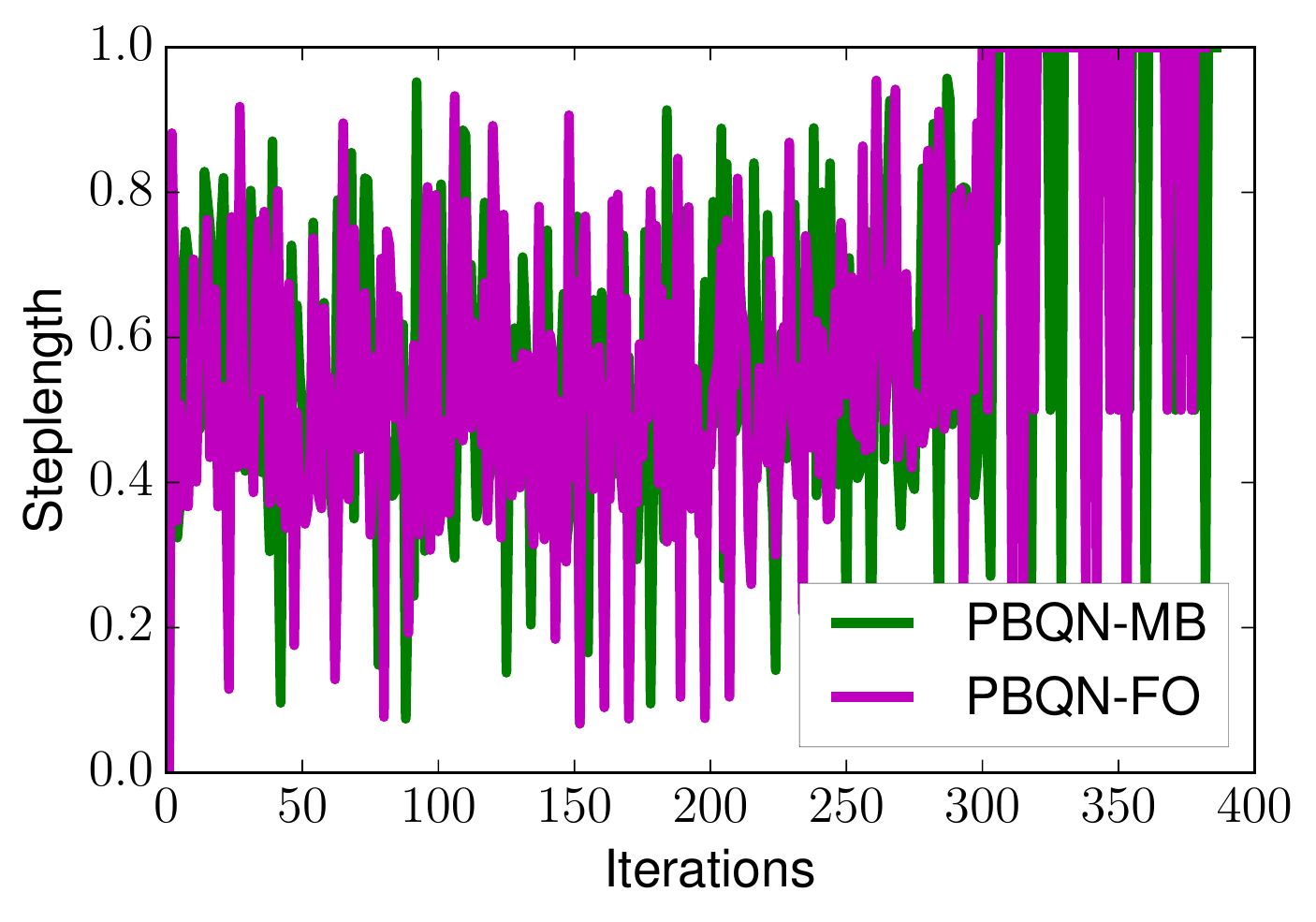}
		\par\end{centering}
	\caption{ \textbf{alpha dataset:} Performance of the progressive batching L-BFGS methods, with multi-batch (MB) (25\% overlap) and full-overlap (FO) approaches, and the SG and SVRG methods.}
	\label{exp:alpha} 
\end{figure*}

\begin{figure*}[!htp]
	\begin{centering}
		\includegraphics[width=0.33\linewidth]{covtype_gradevals_train_loss_wide.pdf}
		\includegraphics[width=0.33\linewidth]{covtype_gradevals_train_loss_tight.pdf}
		\includegraphics[width=0.33\linewidth]{covtype_gradevals_test_loss_tight.pdf}
		\includegraphics[width=0.33\linewidth]{covtype_gradevals_test_acc_tight.pdf}
		\includegraphics[width=0.33\linewidth]{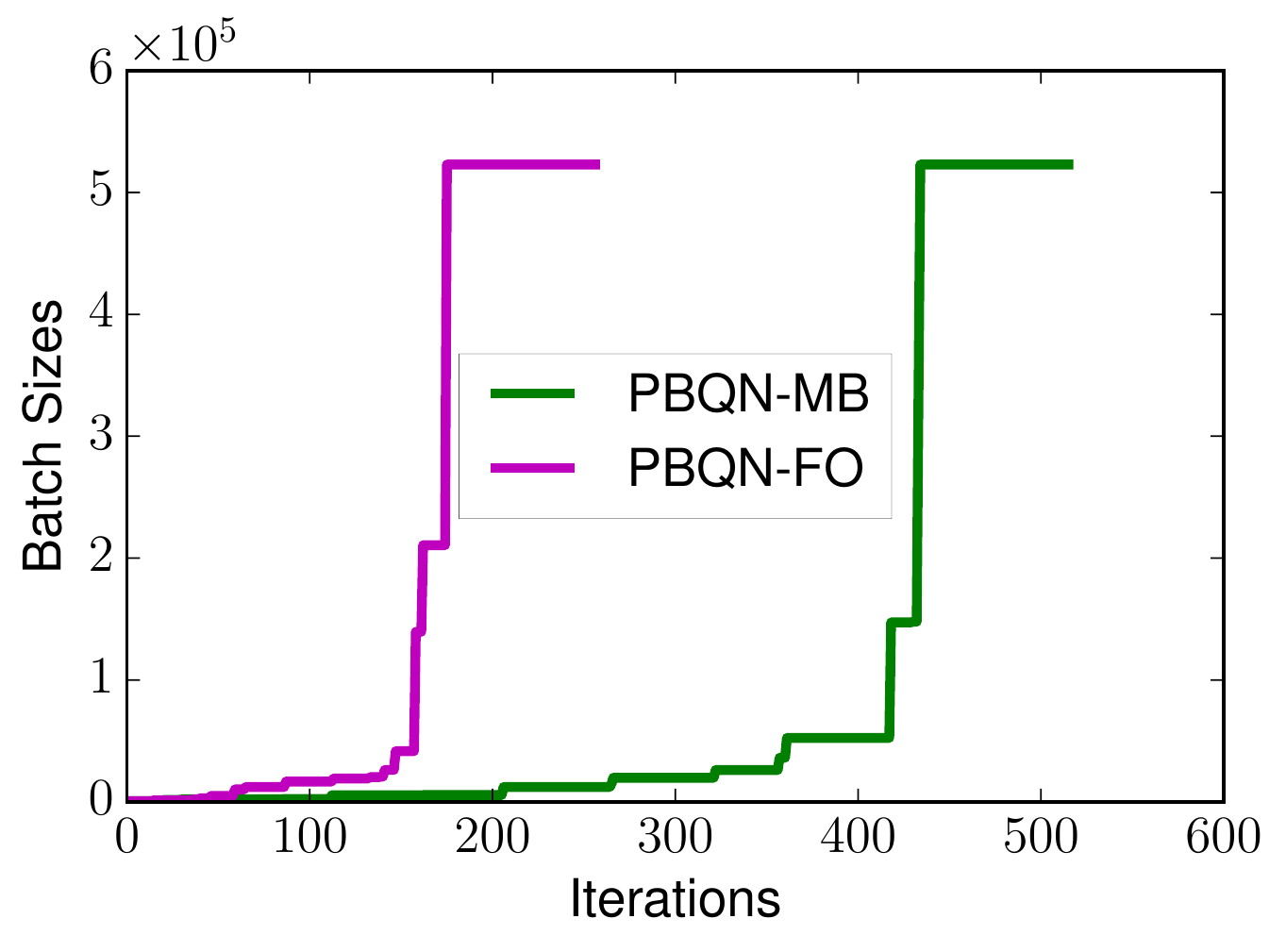}
		\includegraphics[width=0.33\linewidth]{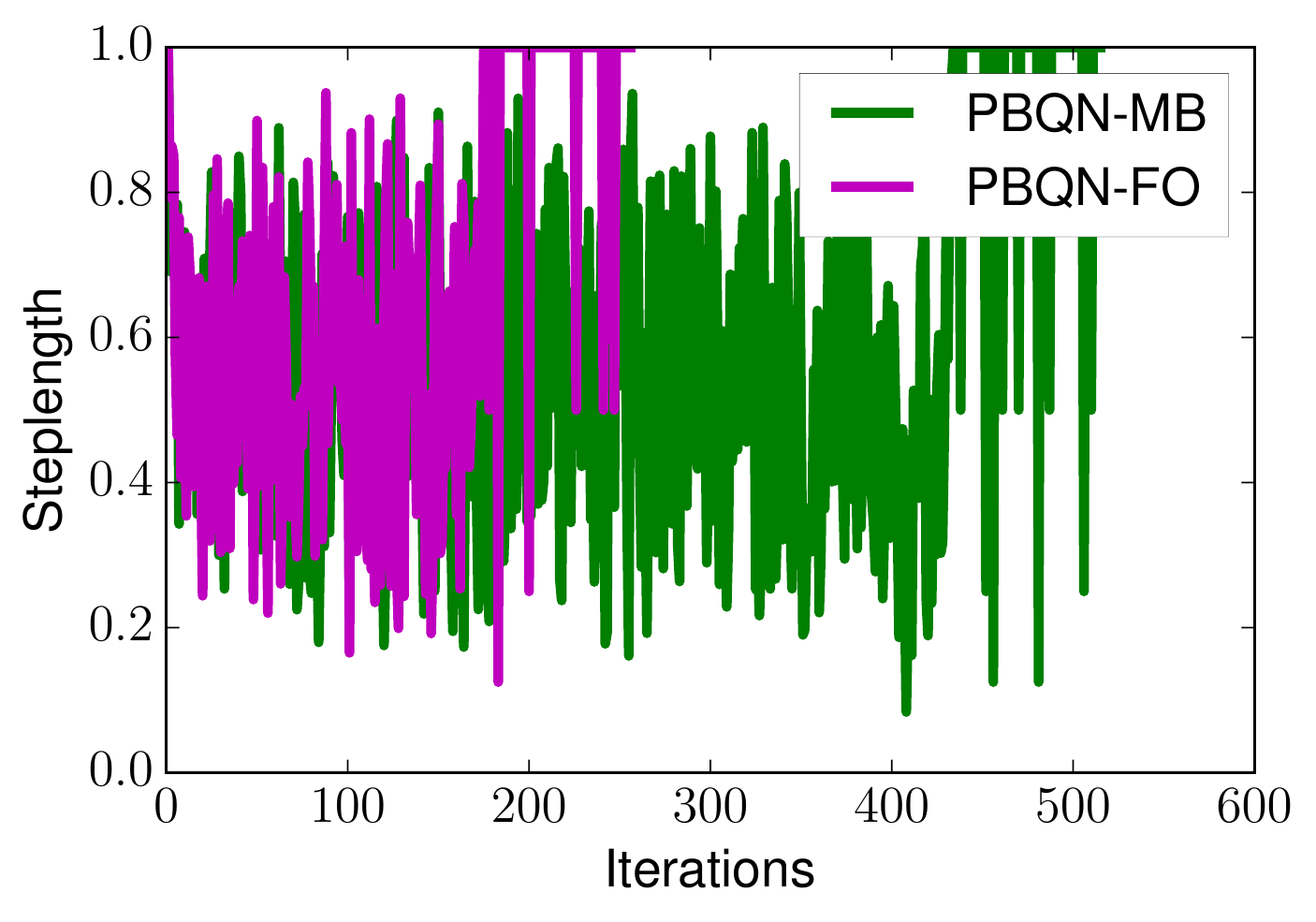}
		\par\end{centering}
	\caption{ \textbf{covertype dataset:} Performance of the progressive batching L-BFGS methods, with multi-batch (MB) (25\% overlap) and full-overlap (FO) approaches, and the SG and SVRG methods.}
	\label{exp:covertype} 
\end{figure*}

\begin{figure*}[!htp]
	\begin{centering}
		\includegraphics[width=0.33\linewidth]{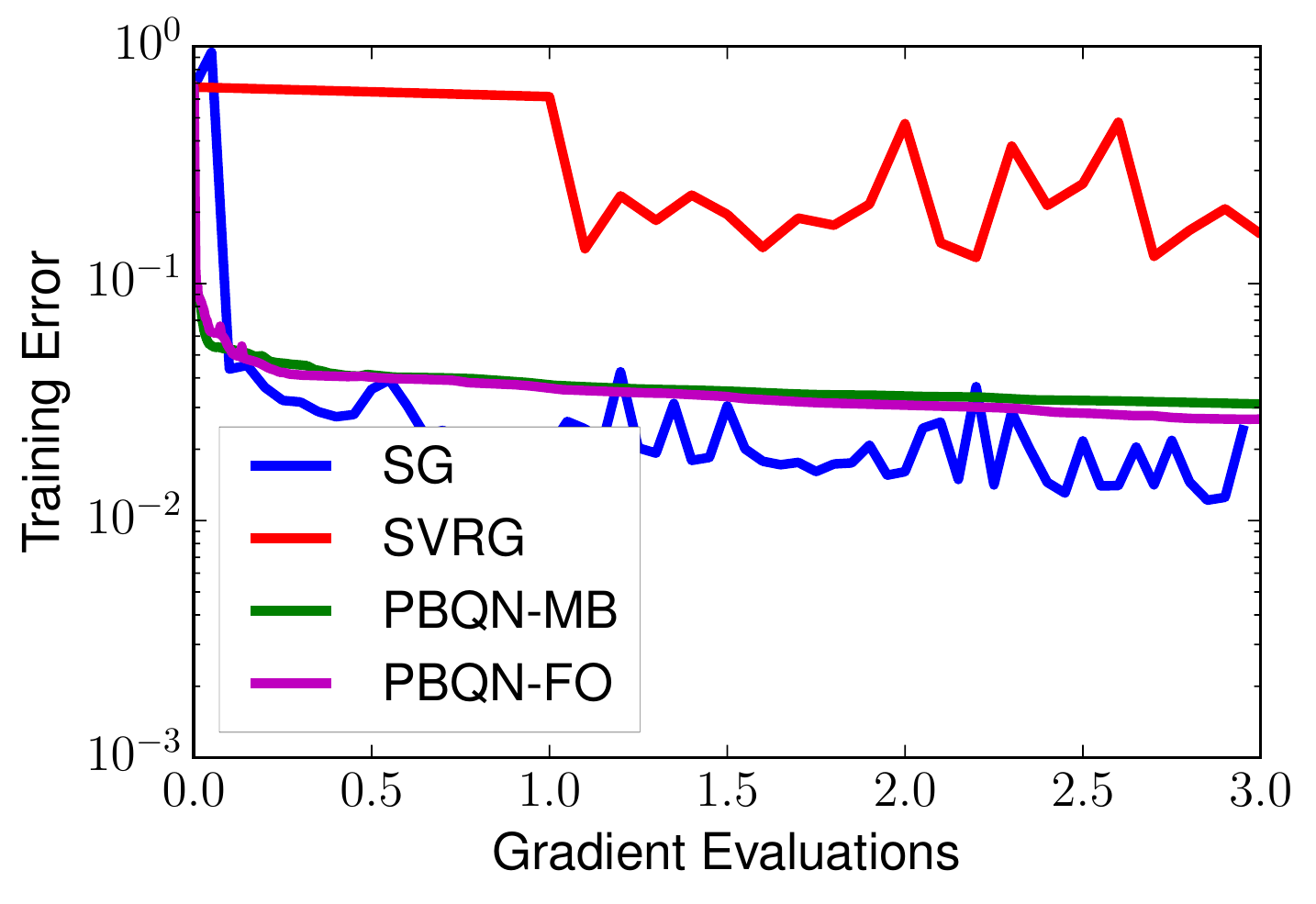}
		\includegraphics[width=0.33\linewidth]{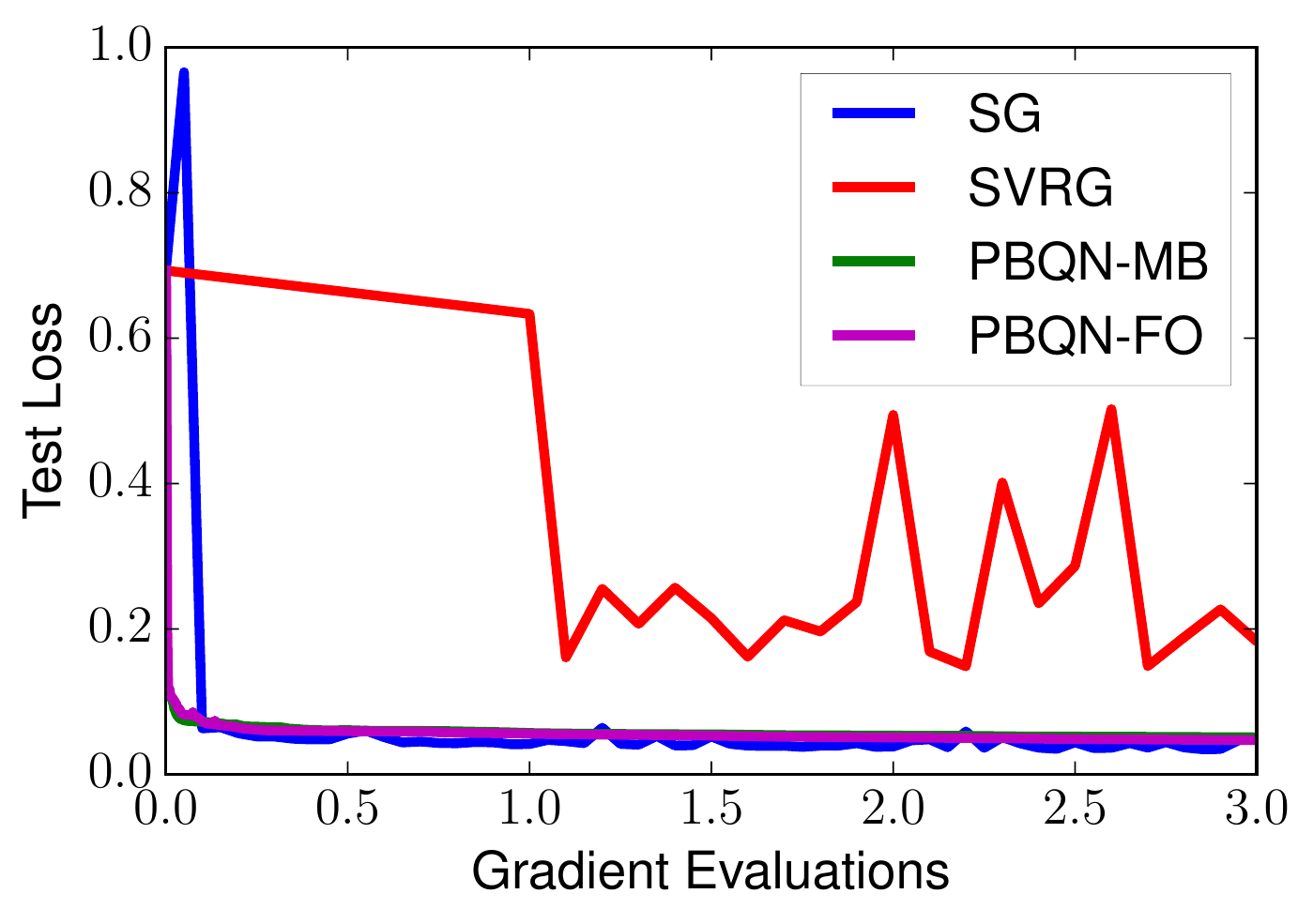}
		
		\includegraphics[width=0.33\linewidth]{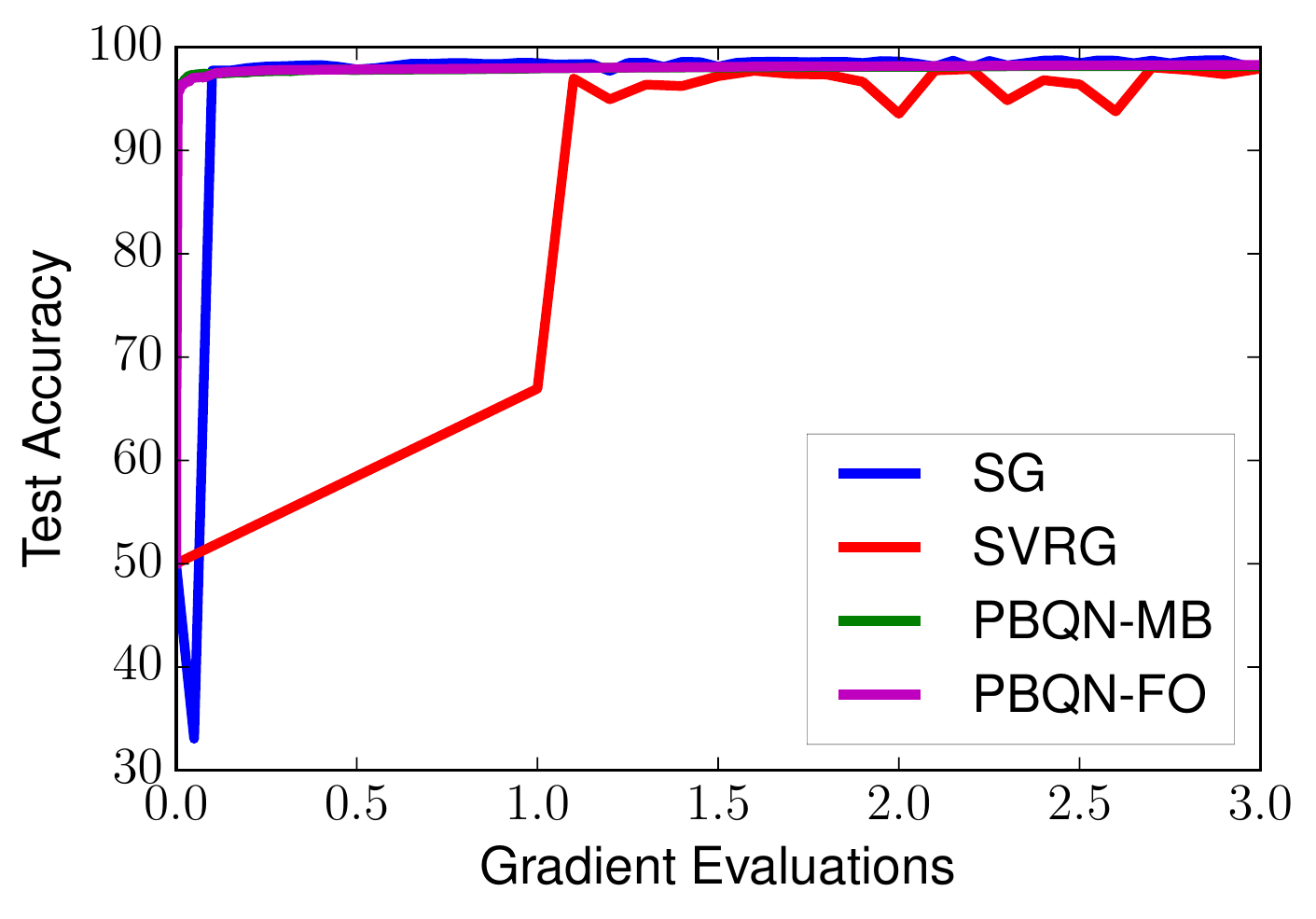}
		\includegraphics[width=0.33\linewidth]{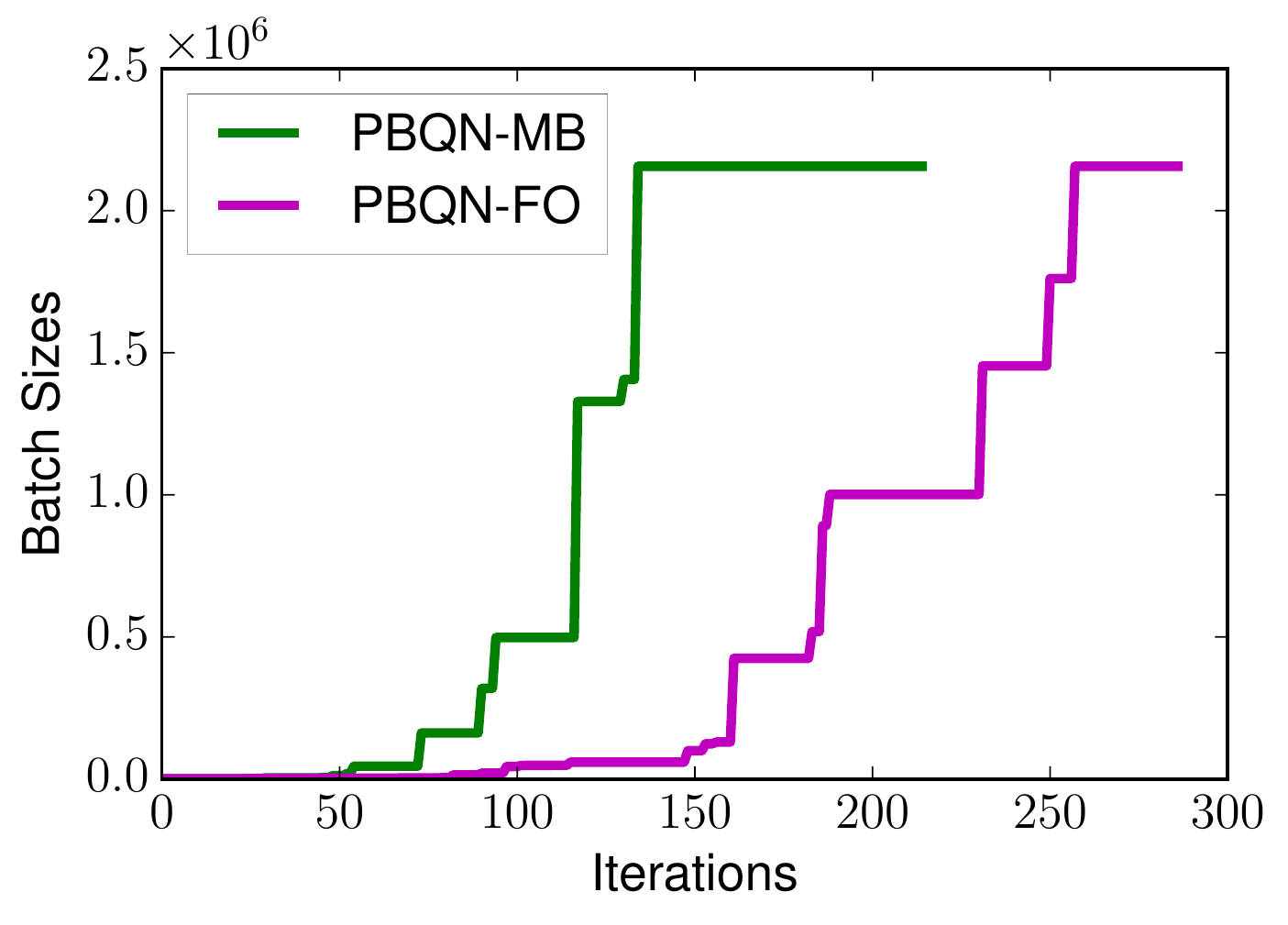}
		\includegraphics[width=0.33\linewidth]{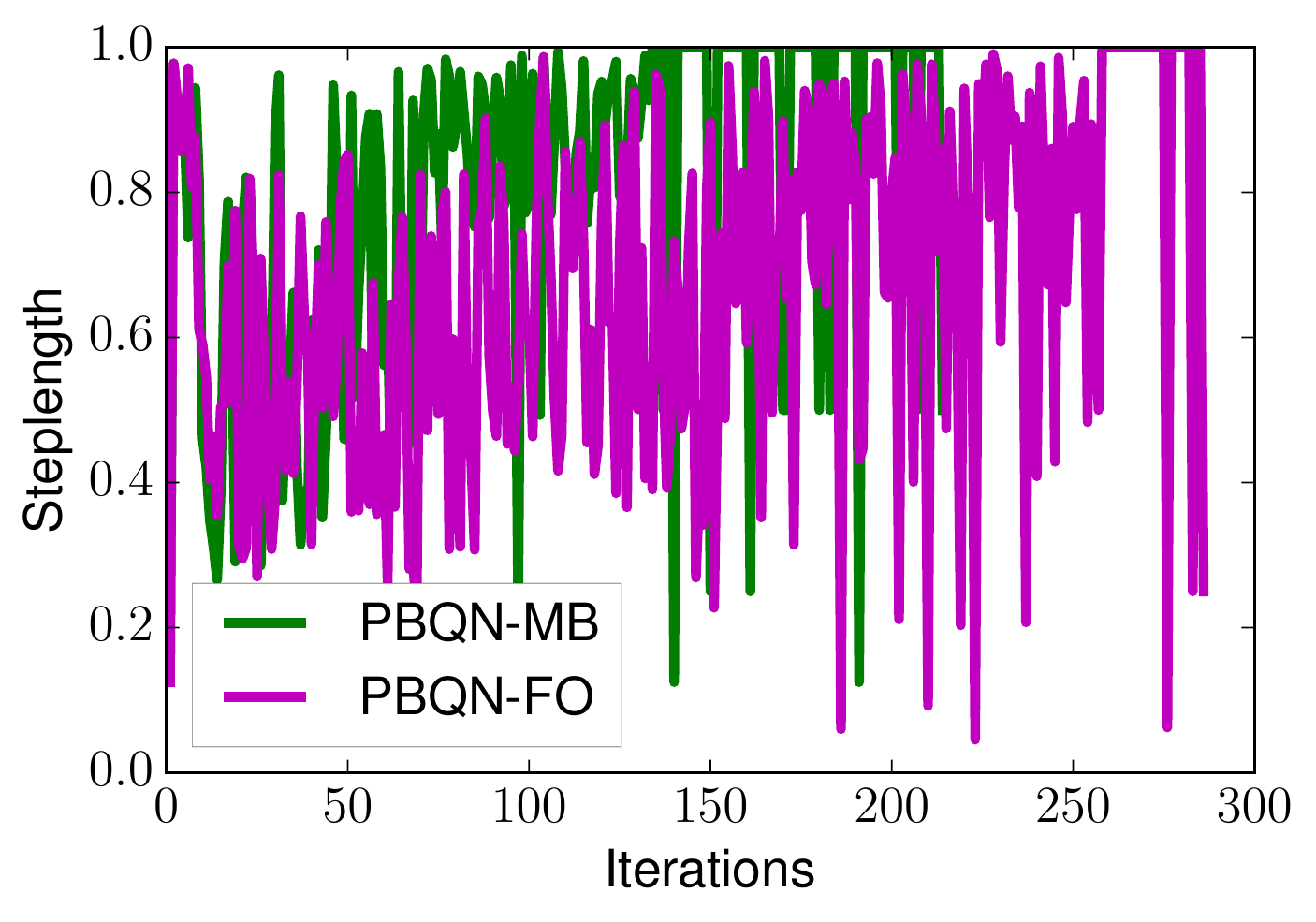}
		\par\end{centering}
	\caption{ \textbf{url dataset:} Performance of the progressive batching L-BFGS methods, with multi-batch (MB) (25\% overlap) and full-overlap (FO) approaches, and the SG and SVRG methods. Note that we only ran the SG and SVRG algorithms for 3 gradient evaluations since the equivalent number of iterations already reached of order of magnitude $10^7$.}
	\label{exp:url} 
\end{figure*}

\newpage
\subsection{Neural Network Experiments}

We describe each neural network architecture below. We  plot the training loss, test loss and test accuracy against the total number of iterations and gradient evaluations. We also report the behavior of the batch sizes and steplengths for both variants of the PBQN method.

\subsubsection{CIFAR-10 Convolutional Network ($\mathcal{C}$) Architecture}

The small convolutional neural network (ConvNet) is a 2-layer convolutional network with two alternating stages of $5 \times 5$ kernels and $2 \times 2$ max pooling followed by a fully connected layer with 1000 ReLU units. The first convolutional layer yields 6 output channels and the second convolutional layer yields 16 output channels.

\subsubsection{CIFAR-10 and MNIST AlexNet-like Network ($\mathcal{A}_1, \mathcal{A}_2$) Architecture}

The larger convolutional network (AlexNet) is an adaptation of the AlexNet architecture \cite{krizhevsky2012imagenet} for CIFAR-10 and MNIST. The CIFAR-10 version consists of three convolutional layers with max pooling followed by two fully-connected layers. The first convolutional layer uses a $5 \times 5$ kernel with a stride of 2 and 64 output channels. The second and third convolutional layers use a $3 \times 3$ kernel with a stride of 1 and 64 output channels. Following each convolutional layer is a set of ReLU activations and $3 \times 3$ max poolings with strides of 2. This is all followed by two fully-connected layers with 384 and 192 neurons with ReLU activations, respectively. The MNIST version of this network modifies this by only using a $2 \times 2$ max pooling layer after the last convolutional layer. 

\subsubsection{CIFAR-10 Residual Network ($\mathcal{R}$) Architecture}

The residual network (ResNet18) is a slight modification of the ImageNet ResNet18 architecture for CIFAR-10 \cite{he2016deep}. It follows the same architecture as ResNet18 for ImageNet but removes the global average pooling layer before the 1000 neuron fully-connected layer. ReLU activations and max poolings are included appropriately.

\begin{figure*}
\begin{centering}
\includegraphics[width=0.33\linewidth]{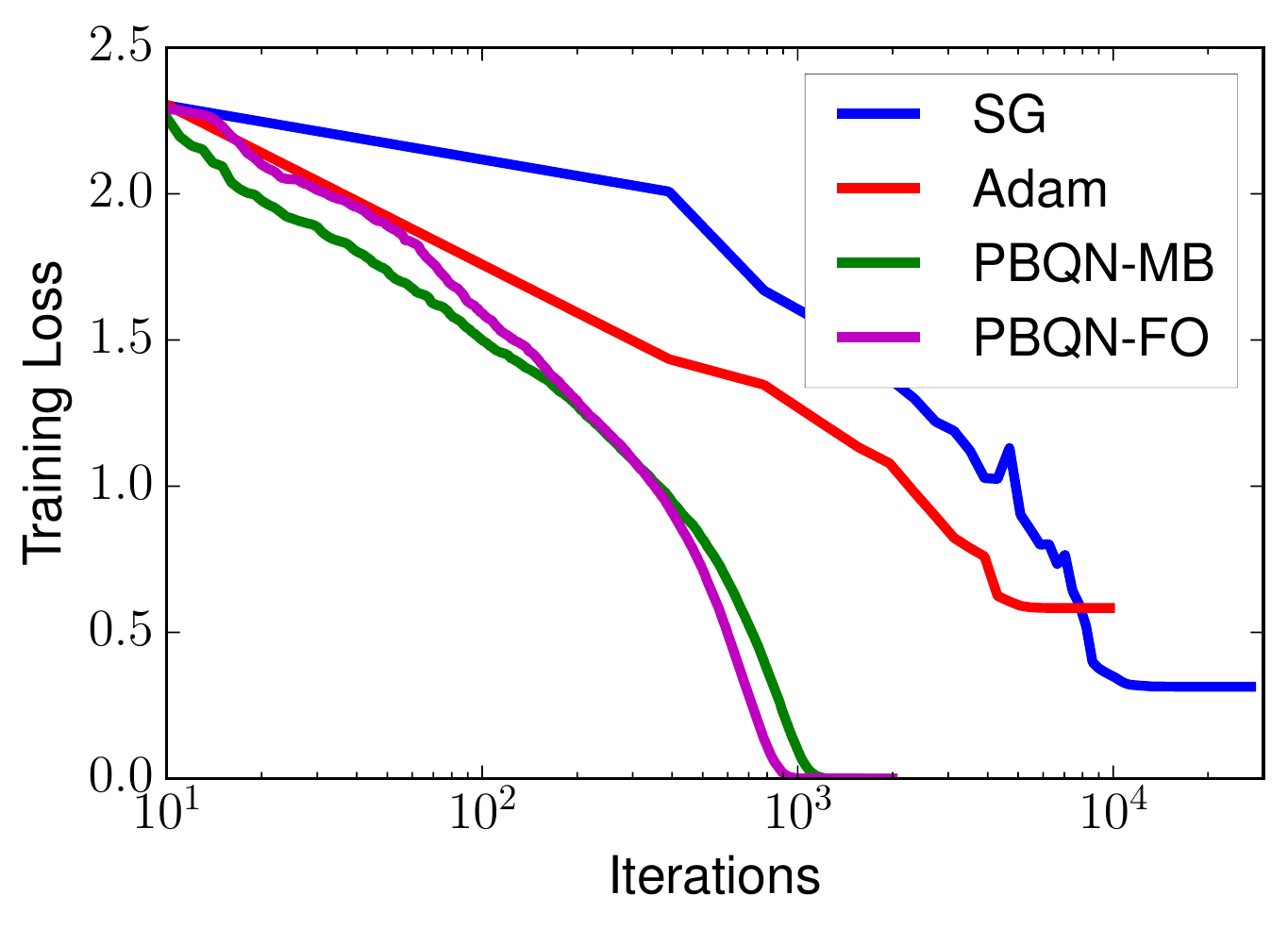}
\includegraphics[width=0.33\linewidth]{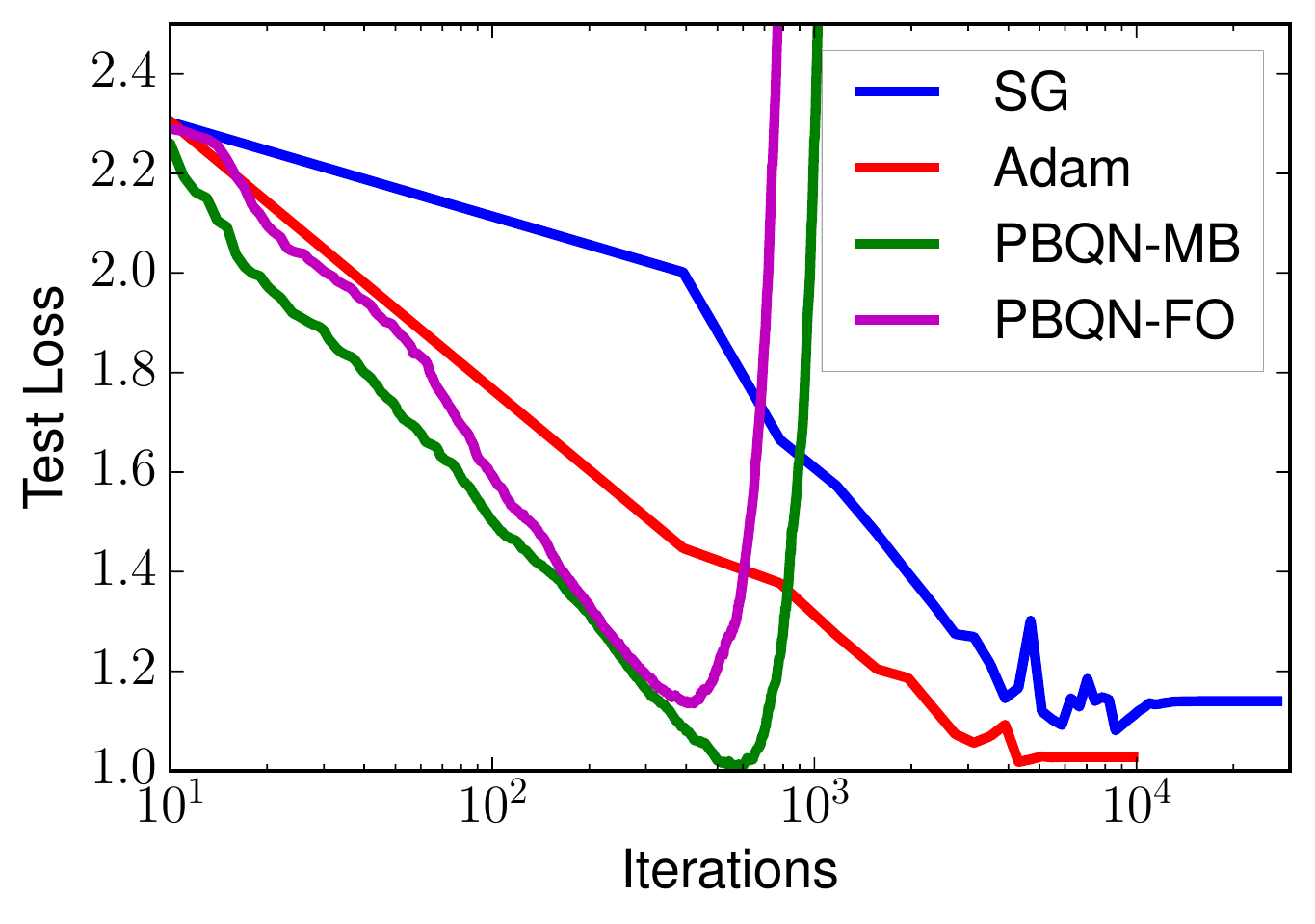}
\includegraphics[width=0.33\linewidth]{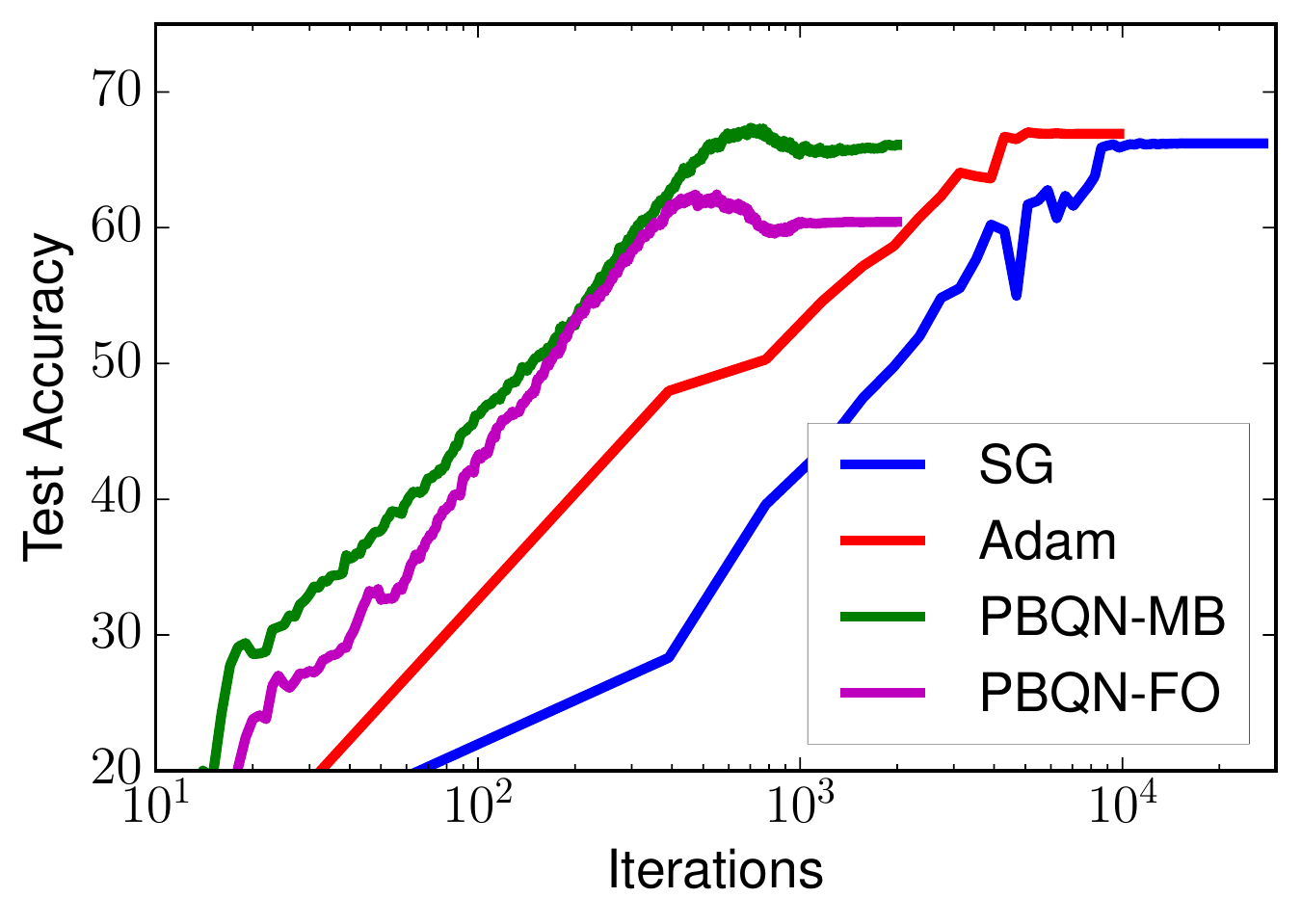}
\includegraphics[width=0.33\linewidth]{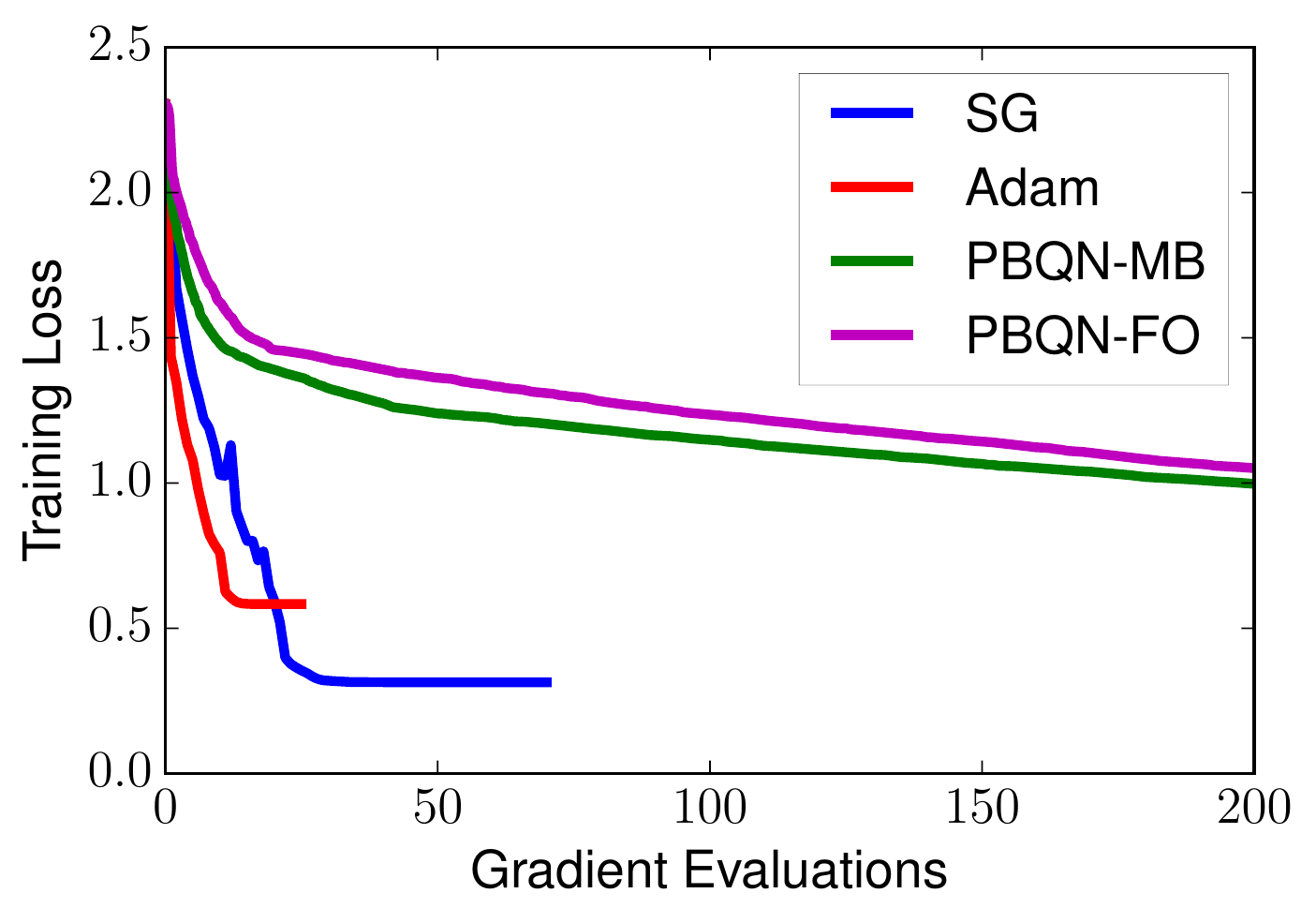}
\includegraphics[width=0.33\linewidth]{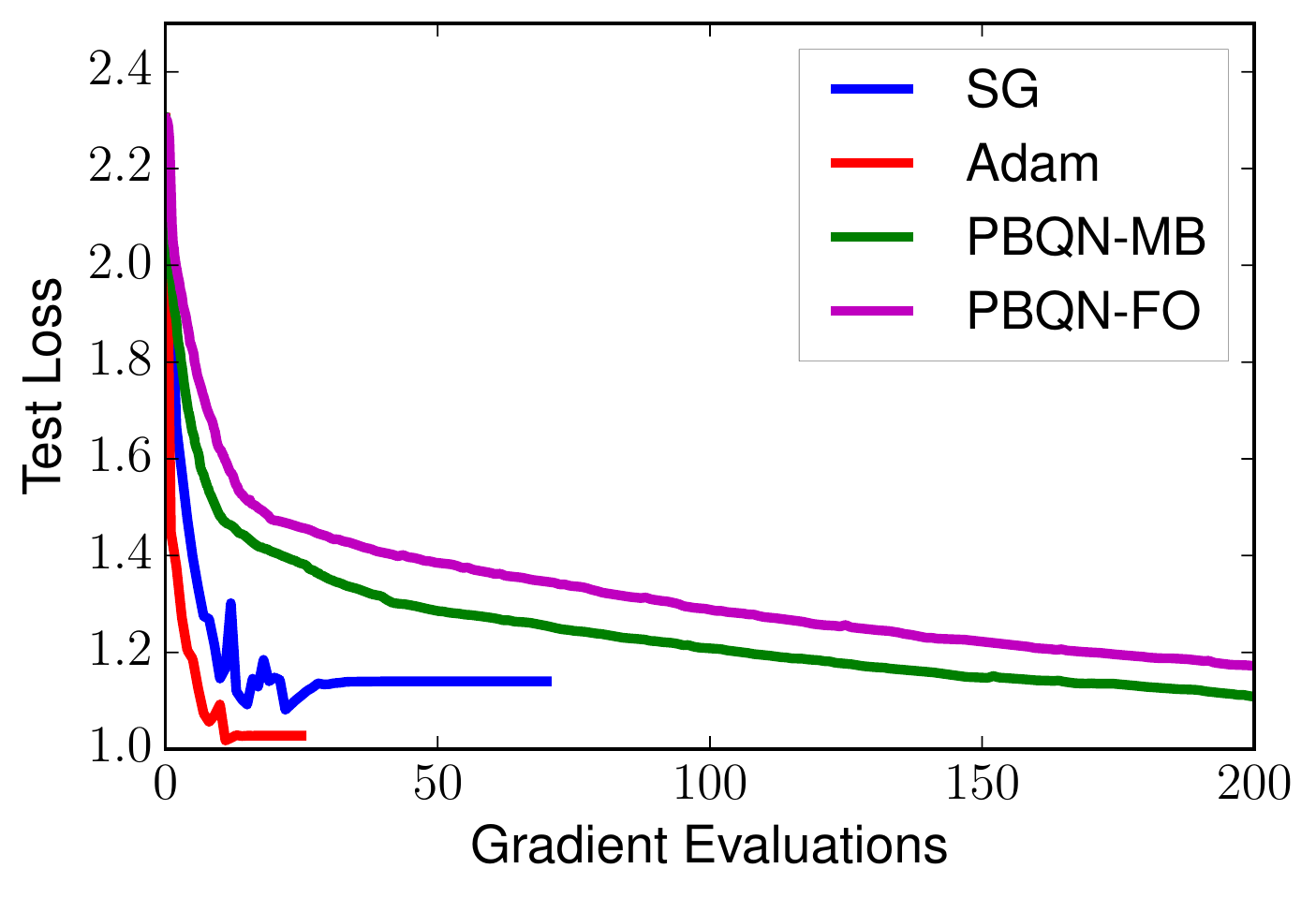}
\includegraphics[width=0.33\linewidth]{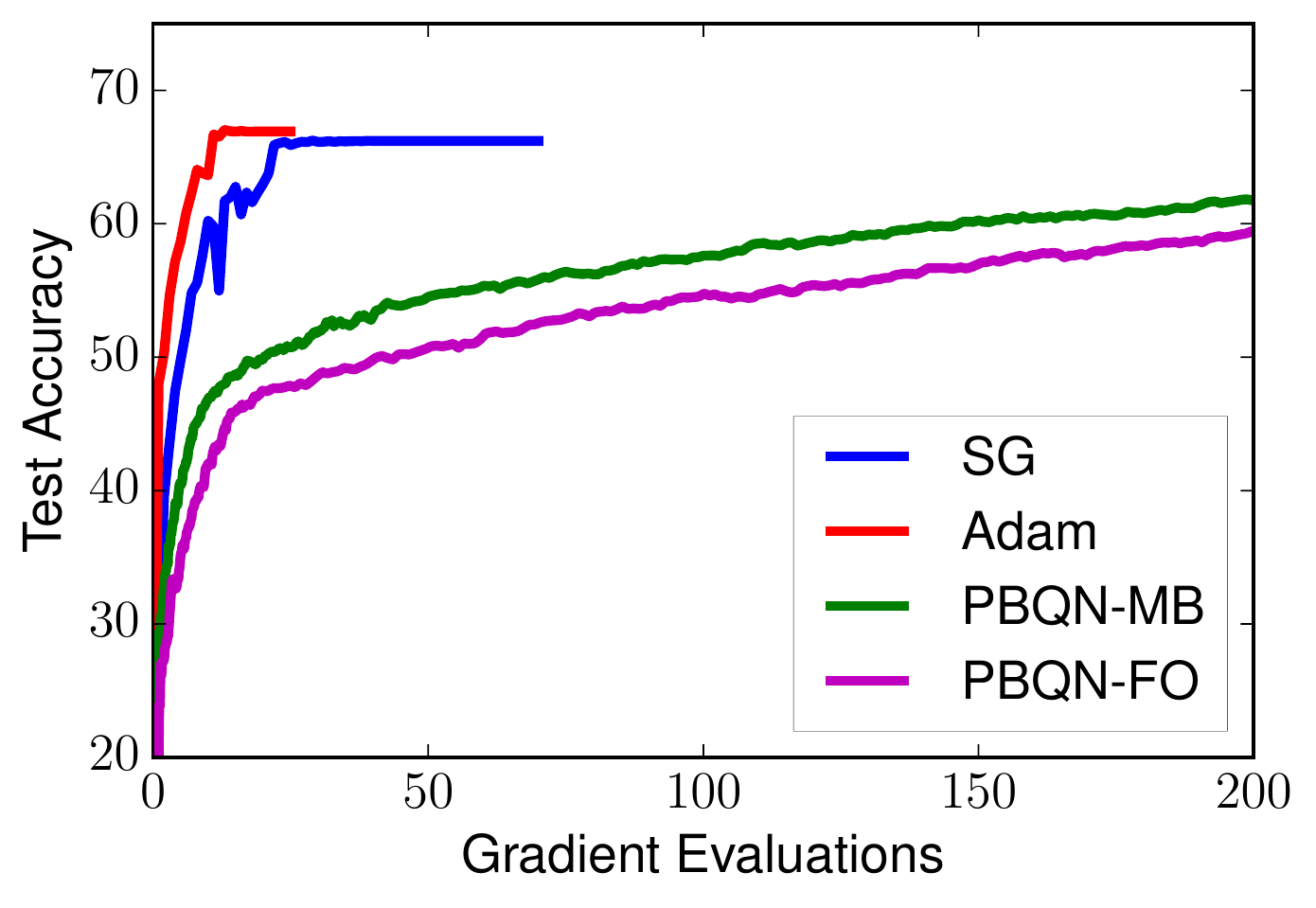}
\includegraphics[width=0.33\linewidth]{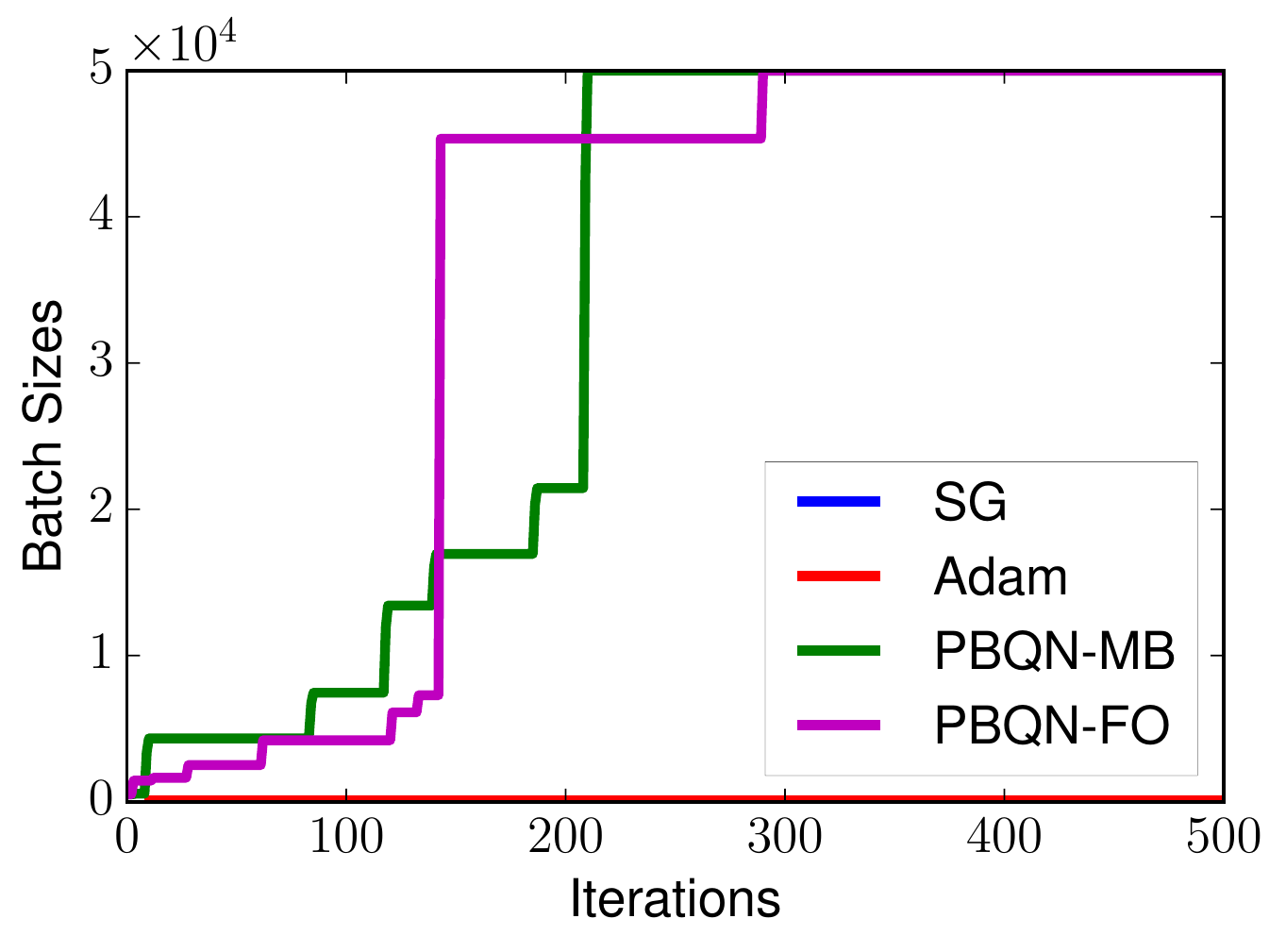}
\includegraphics[width=0.33\linewidth]{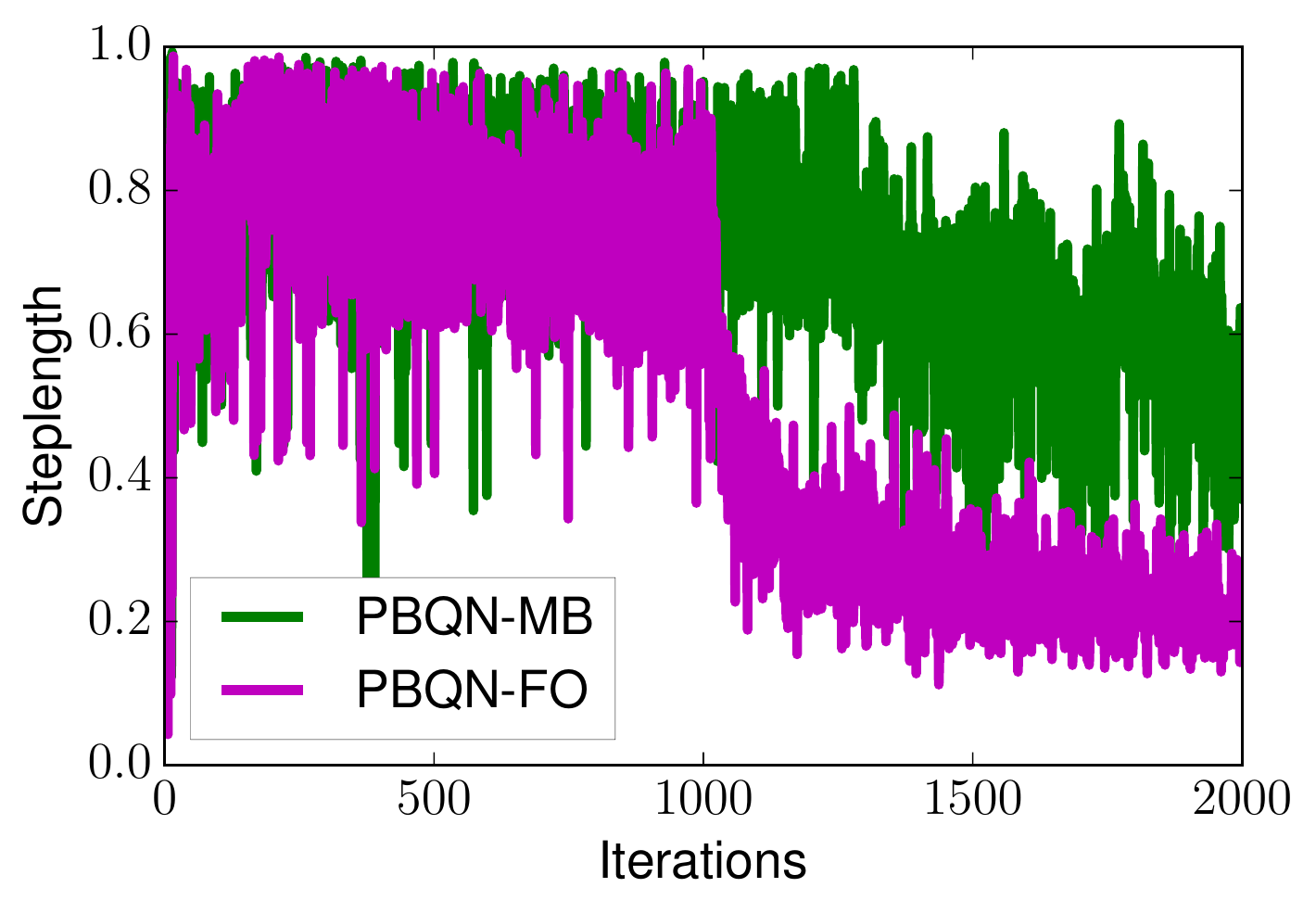}
\caption{\textbf{CIFAR-10 ConvNet $(\mathcal{C})$:} Performance of the progressive batching L-BFGS methods, with multi-batch (MB) (25\% overlap) and full-overlap (FO) approaches, and the SG and Adam methods. The best results for L-BFGS are achieved with $\theta = 0.9$.}
\label{exp:cifar10_cnn}
\end{centering}
\end{figure*}

\begin{figure*}
\begin{centering}
\includegraphics[width=0.33\linewidth]{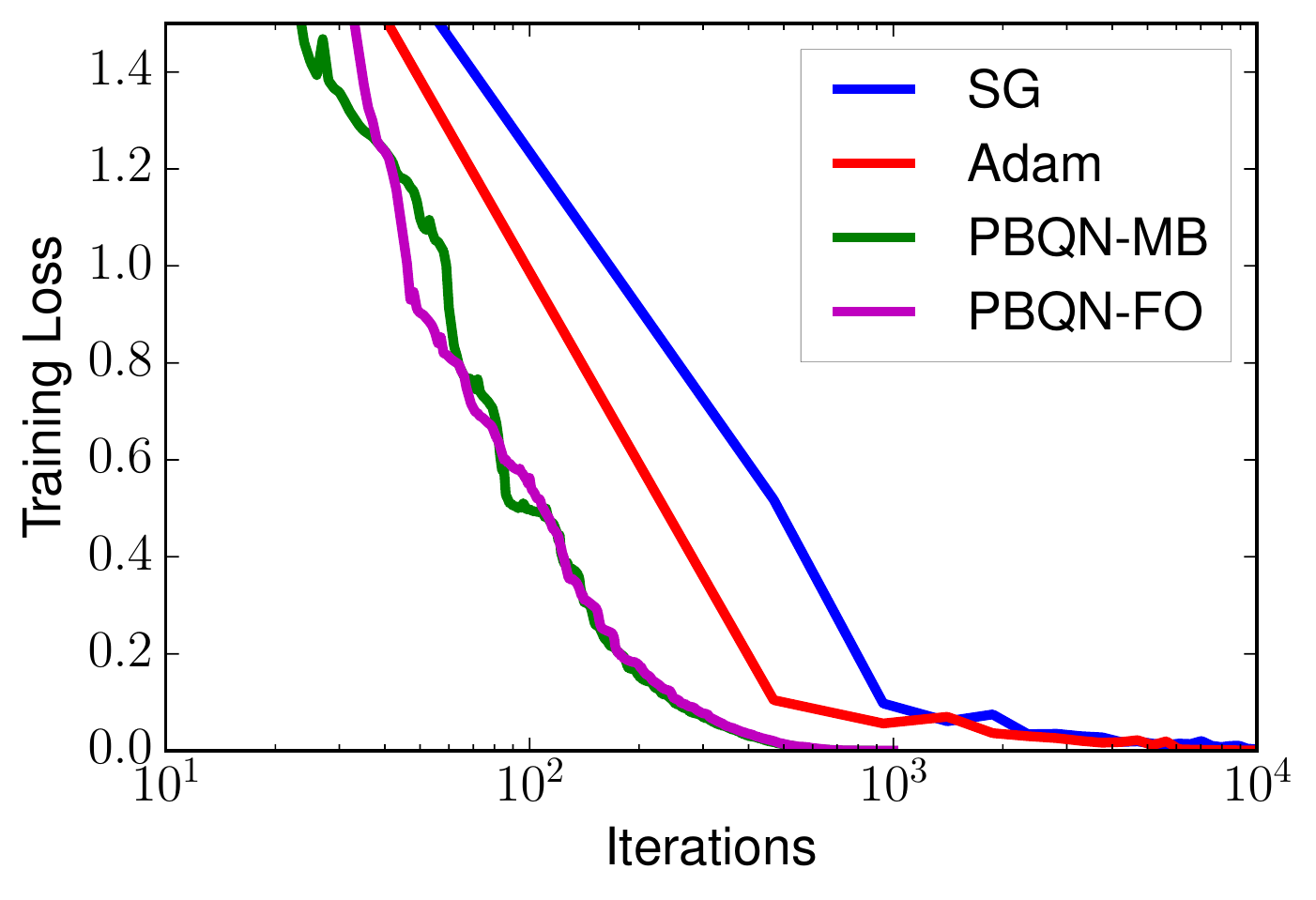}
\includegraphics[width=0.33\linewidth]{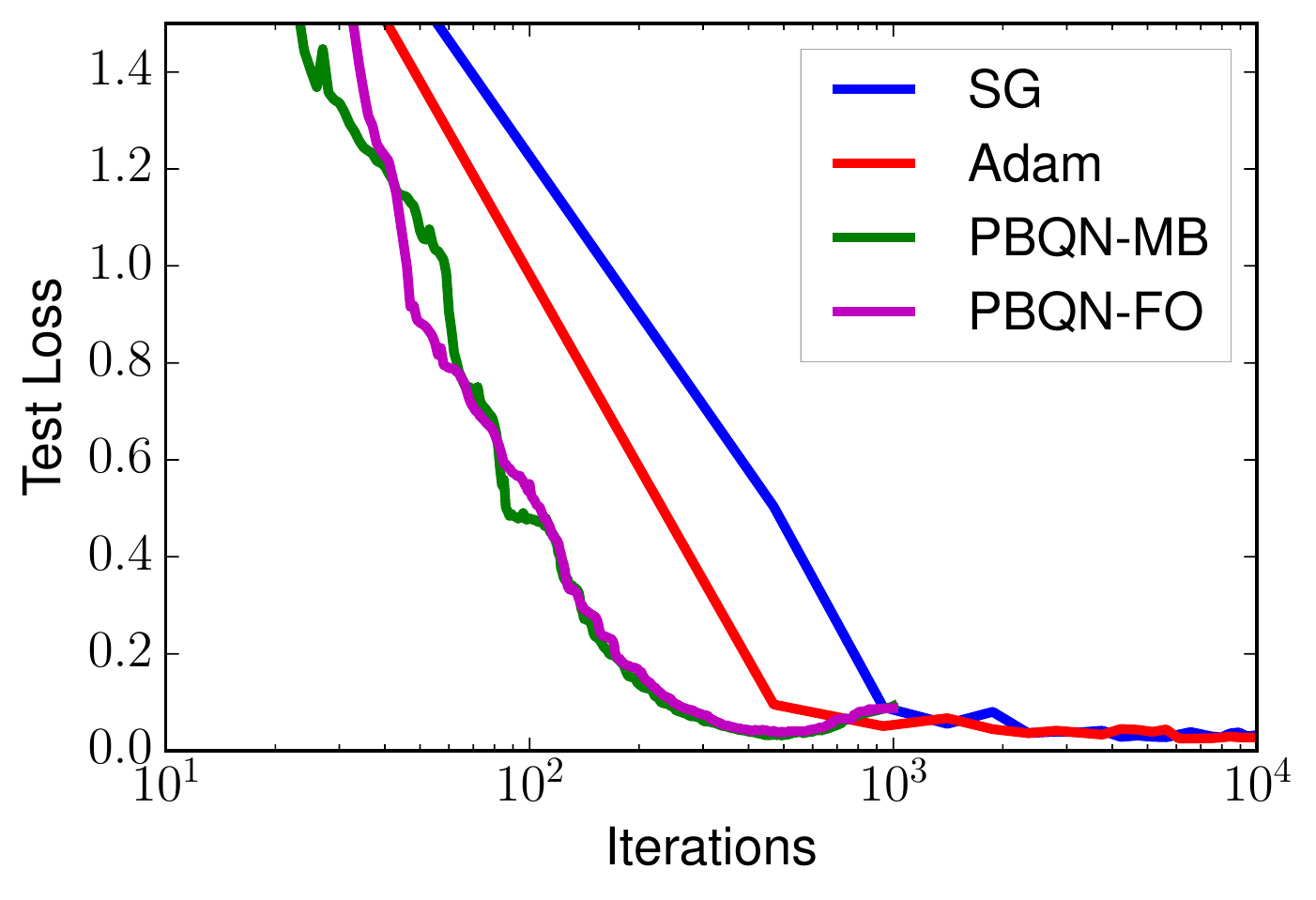}
\includegraphics[width=0.33\linewidth]{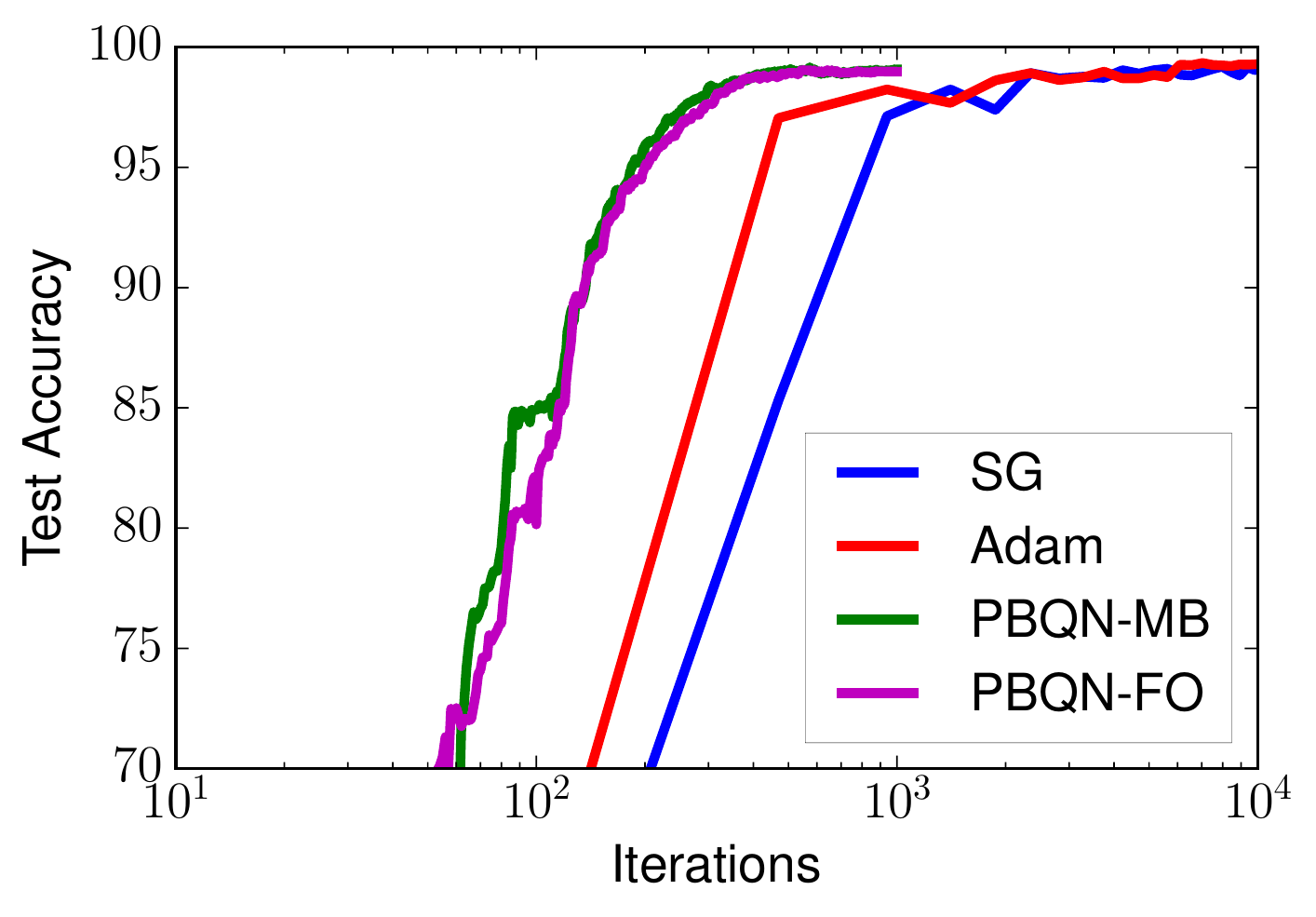}
\includegraphics[width=0.33\linewidth]{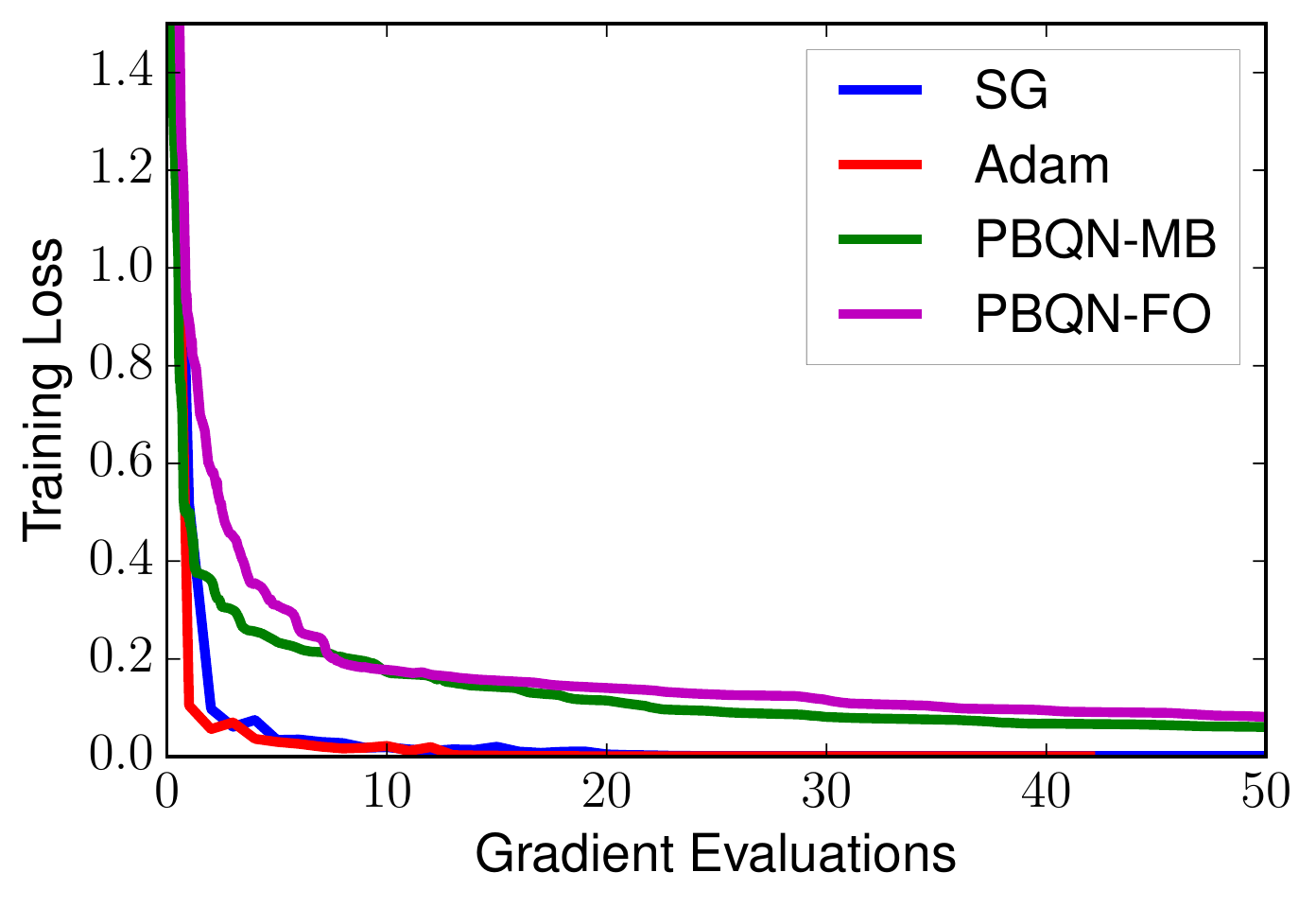}
\includegraphics[width=0.33\linewidth]{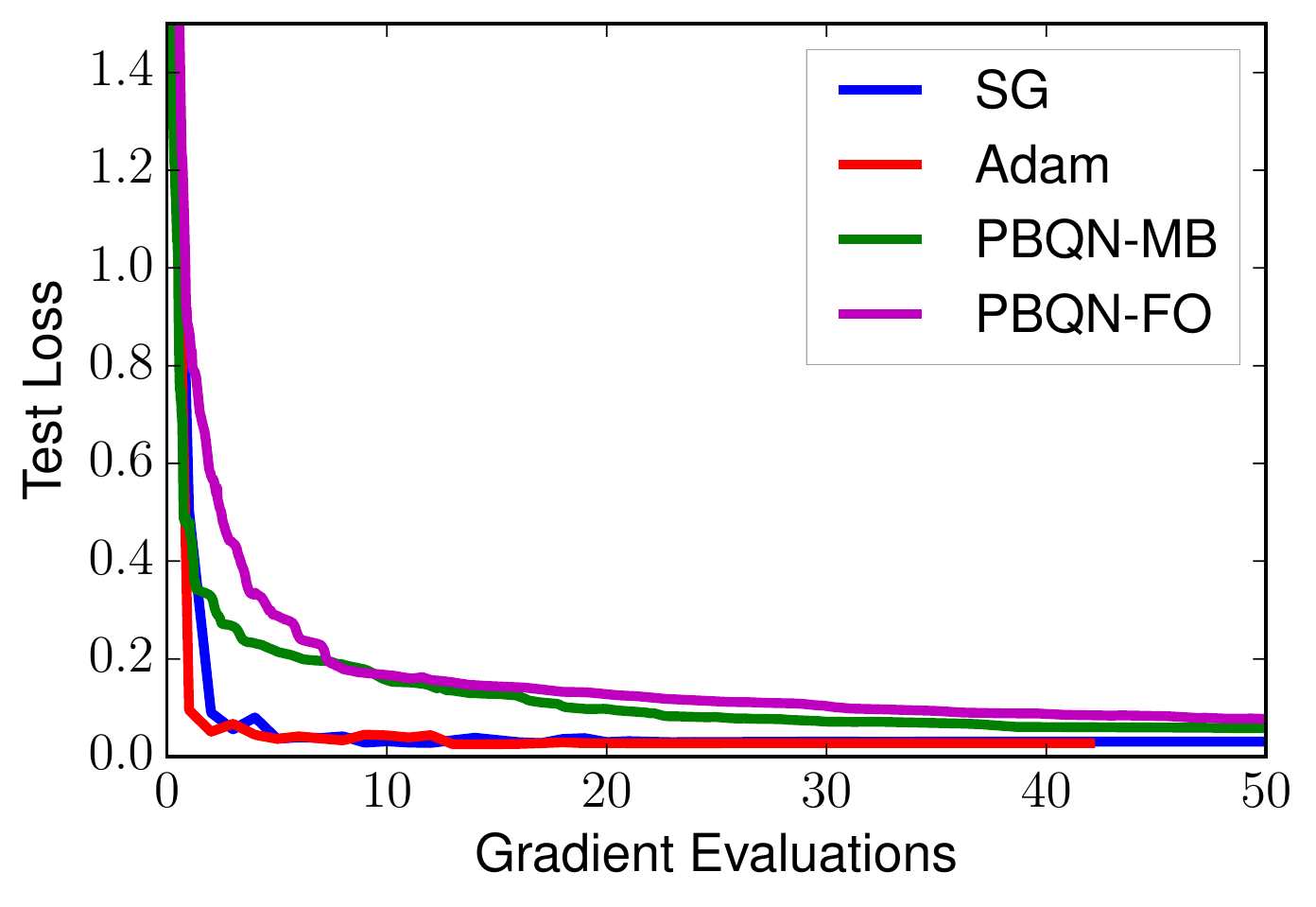}
\includegraphics[width=0.33\linewidth]{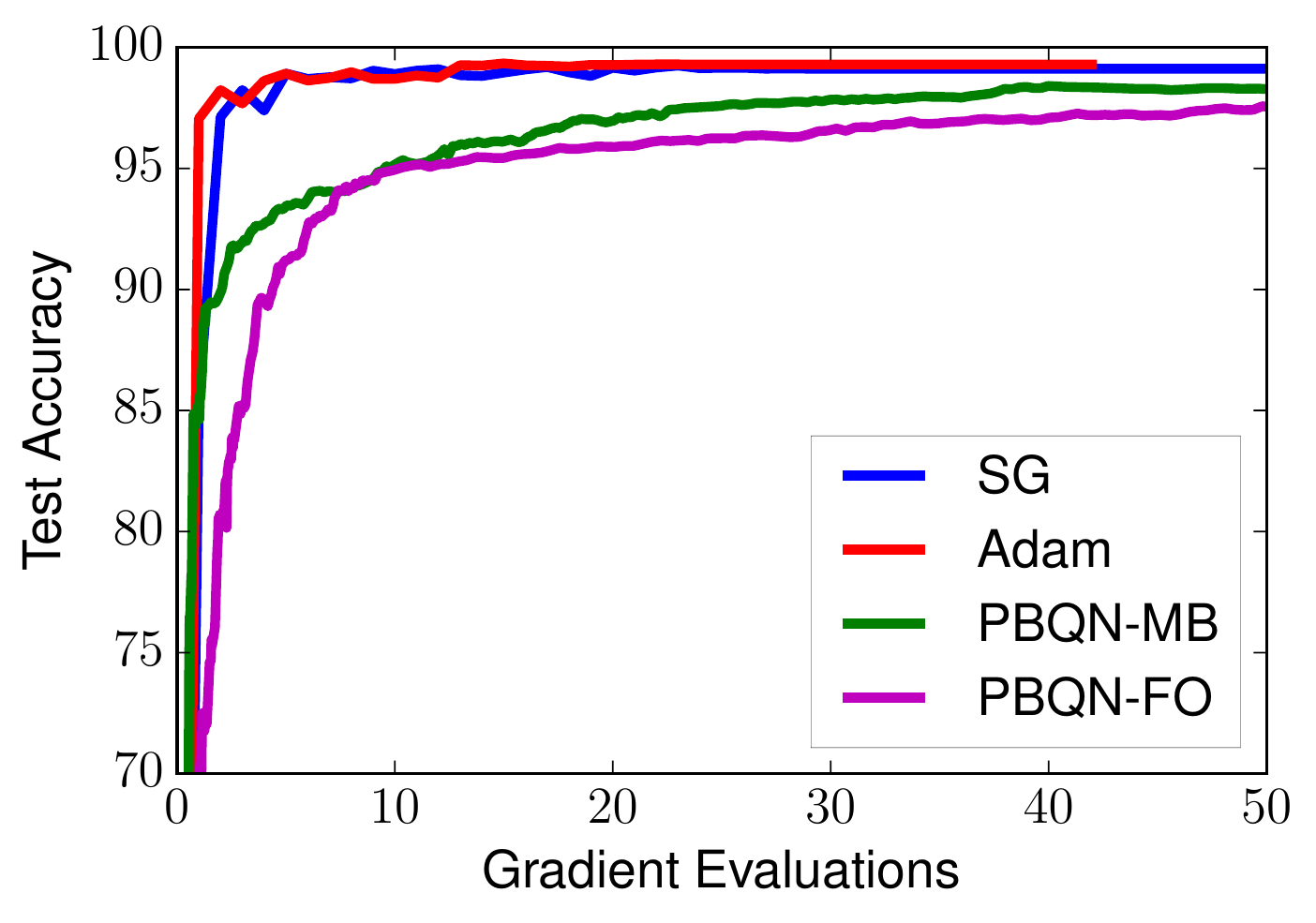}
\includegraphics[width=0.33\linewidth]{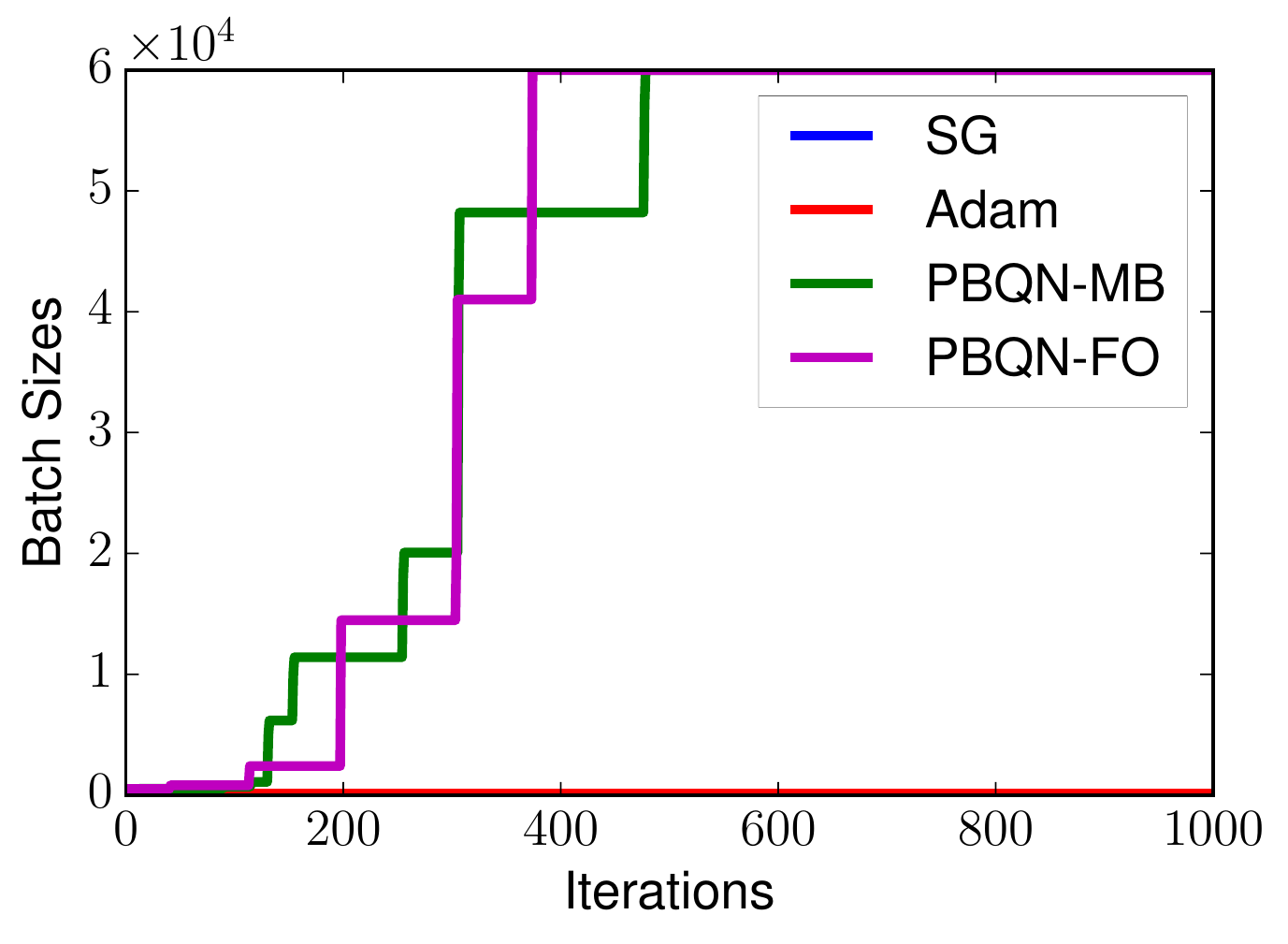}
\includegraphics[width=0.33\linewidth]{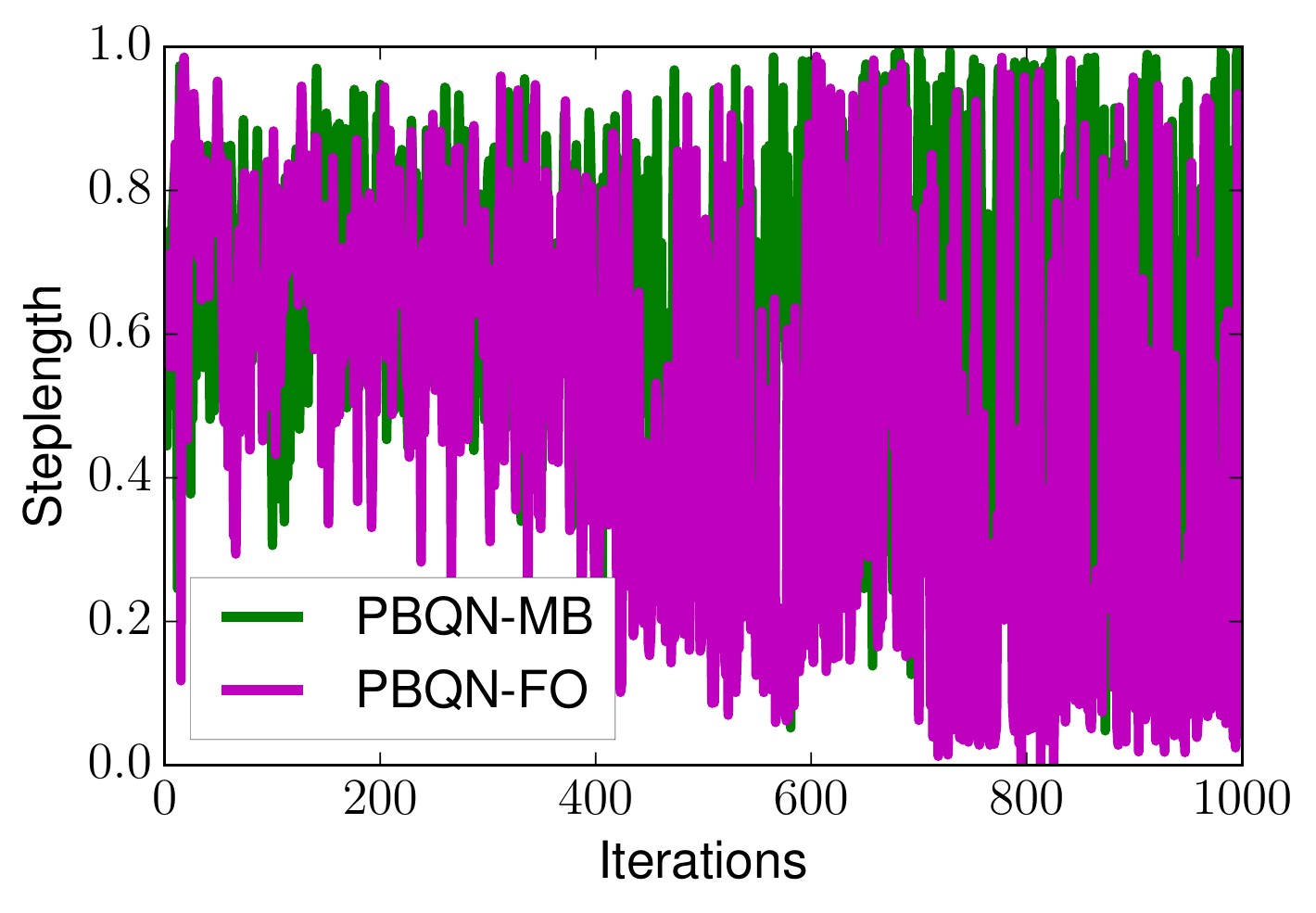}
\caption{\textbf{MNIST AlexNet $(\mathcal{A}_1)$:} Performance of the progressive batching L-BFGS methods, with multi-batch (MB) (25\% overlap) and full-overlap (FO) approaches, and the SG and Adam methods. The best results for L-BFGS are achieved with $\theta = 2$.}
\label{exp:mnist_alexnet}
\end{centering}
\end{figure*}

\begin{figure*}
\begin{centering}
\includegraphics[width=0.33\linewidth]{cifar10_alexnet_iters_train_loss_theta_0.9.pdf}
\includegraphics[width=0.33\linewidth]{cifar10_alexnet_iters_test_loss_theta_0.9.pdf}
\includegraphics[width=0.33\linewidth]{cifar10_alexnet_iters_test_acc_theta_0.9.pdf}
\includegraphics[width=0.33\linewidth]{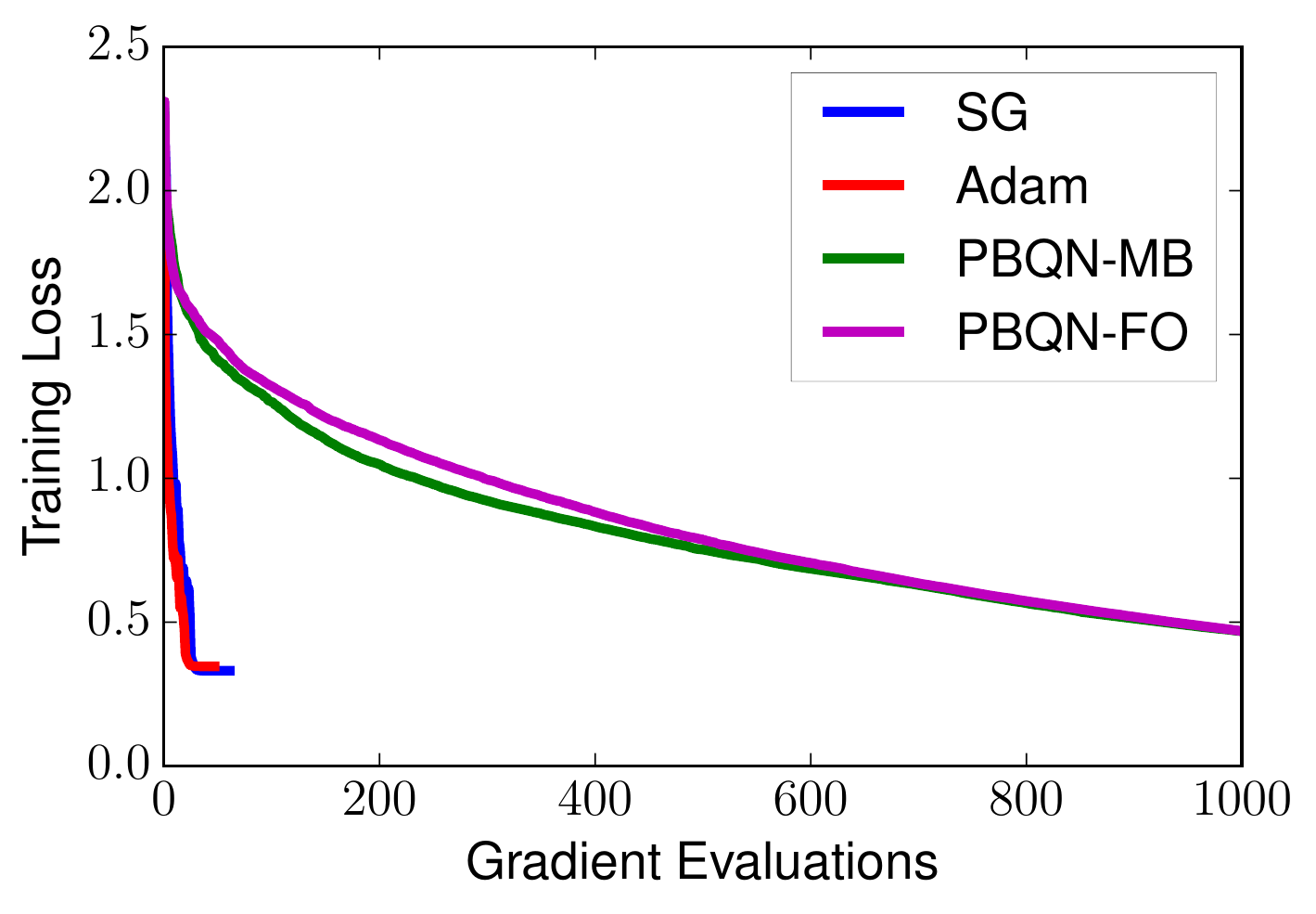}
\includegraphics[width=0.33\linewidth]{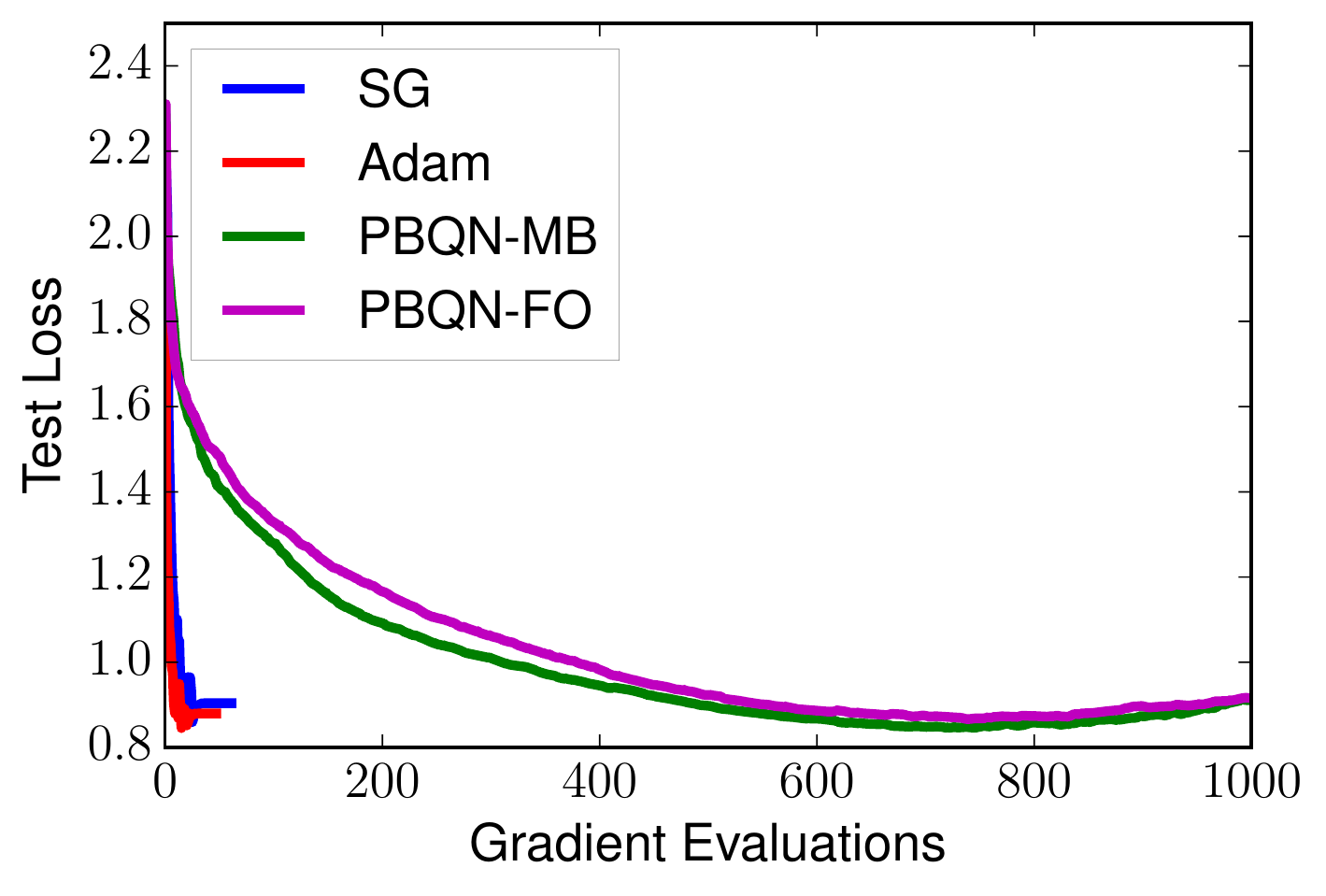}
\includegraphics[width=0.33\linewidth]{cifar10_alexnet_props_test_acc_theta_0.9.pdf}
\includegraphics[width=0.33\linewidth]{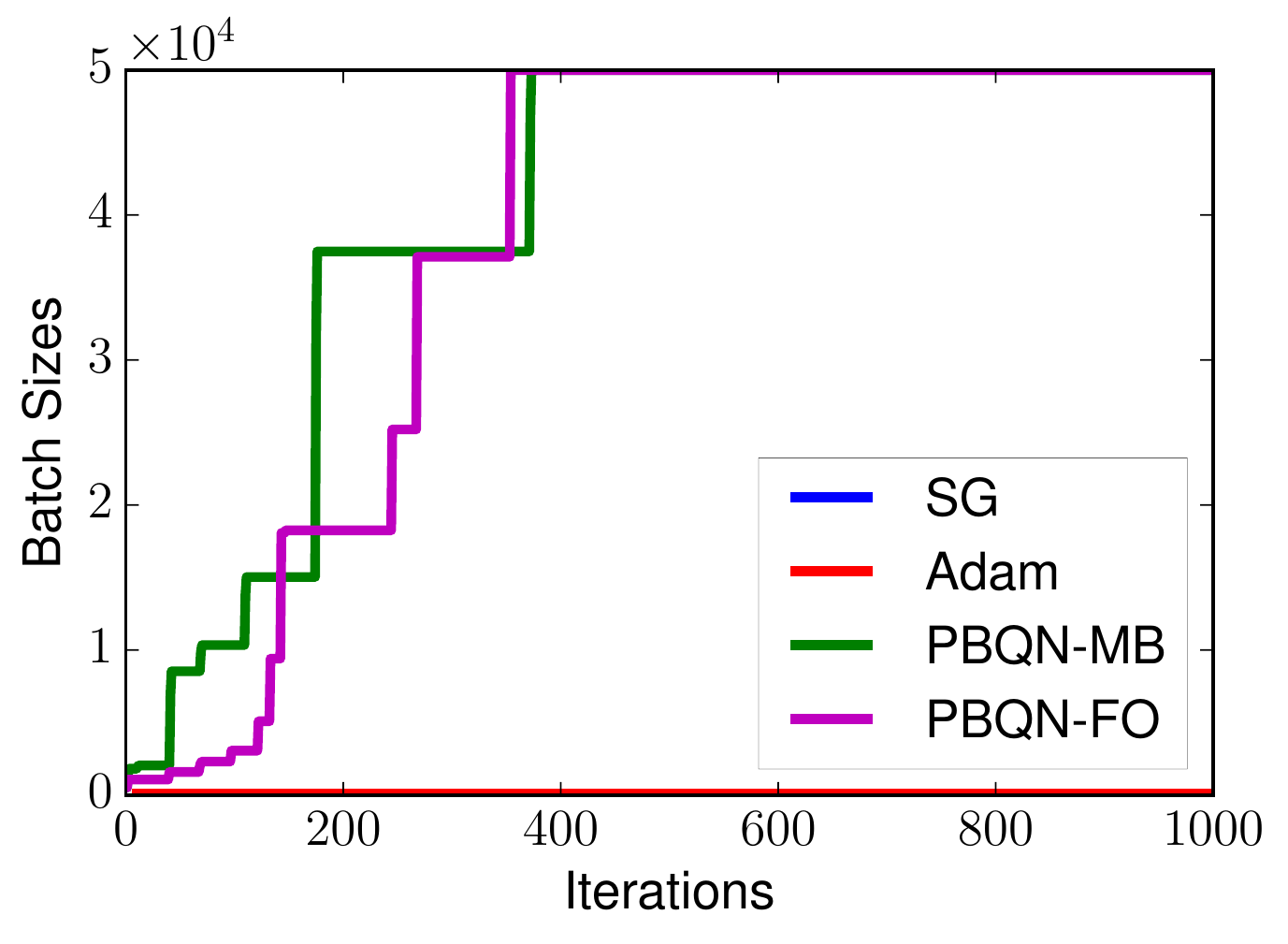}
\includegraphics[width=0.33\linewidth]{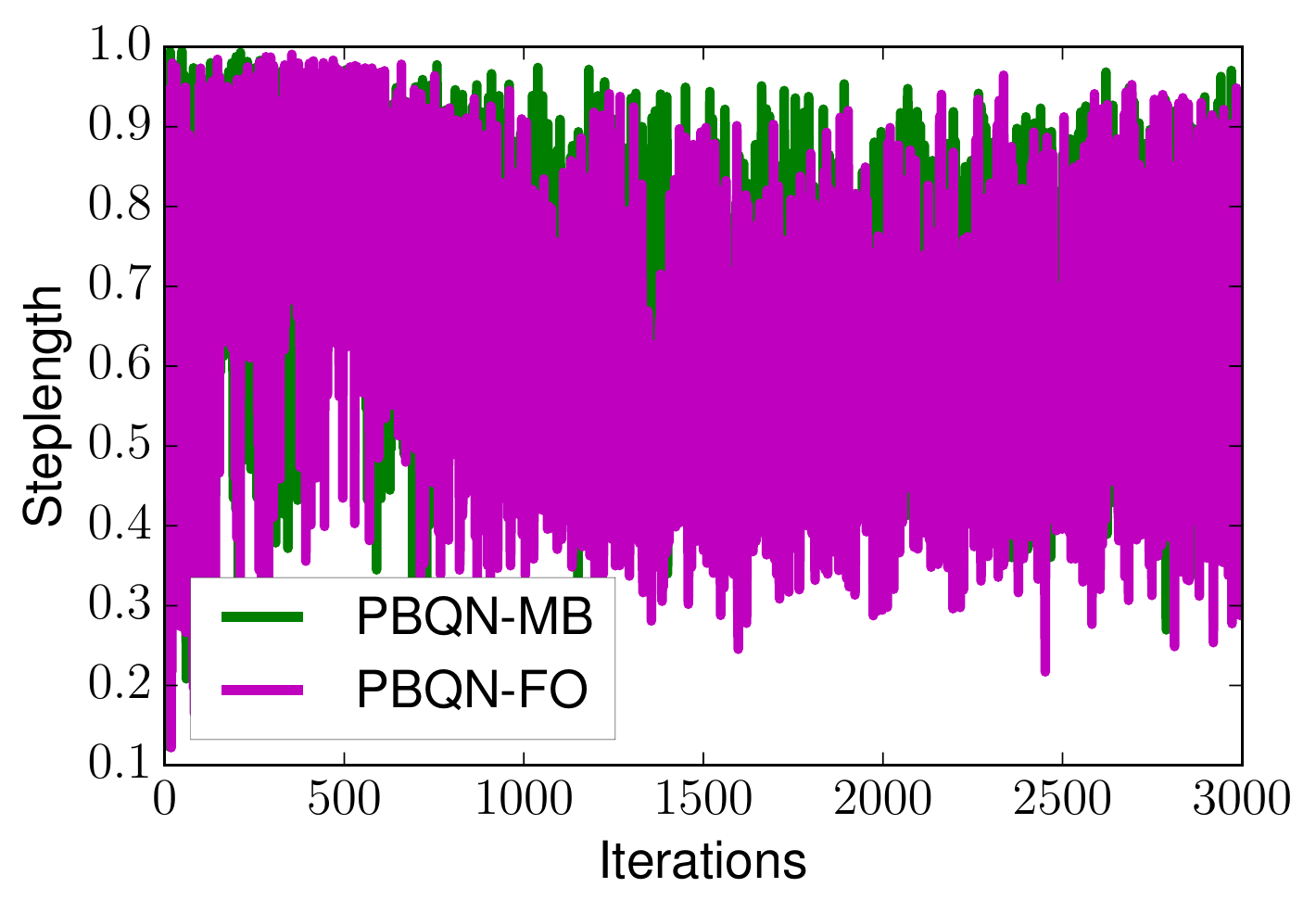}
\caption{\textbf{CIFAR-10 AlexNet $(\mathcal{A}_2)$:} Performance of the progressive batching L-BFGS methods, with multi-batch (MB) (25\% overlap) and full-overlap (FO) approaches, and the SG and Adam methods. The best results for L-BFGS are achieved with $\theta = 0.9$.}
\label{exp:cifar10_alexnet}
\end{centering}
\end{figure*}

\begin{figure*}
\begin{centering}
\includegraphics[width=0.33\linewidth]{cifar10_resnet18_iters_train_loss_theta_2.pdf}
\includegraphics[width=0.33\linewidth]{cifar10_resnet18_iters_test_loss_theta_2.pdf}
\includegraphics[width=0.33\linewidth]{cifar10_resnet18_iters_test_acc_theta_2.pdf}
\includegraphics[width=0.33\linewidth]{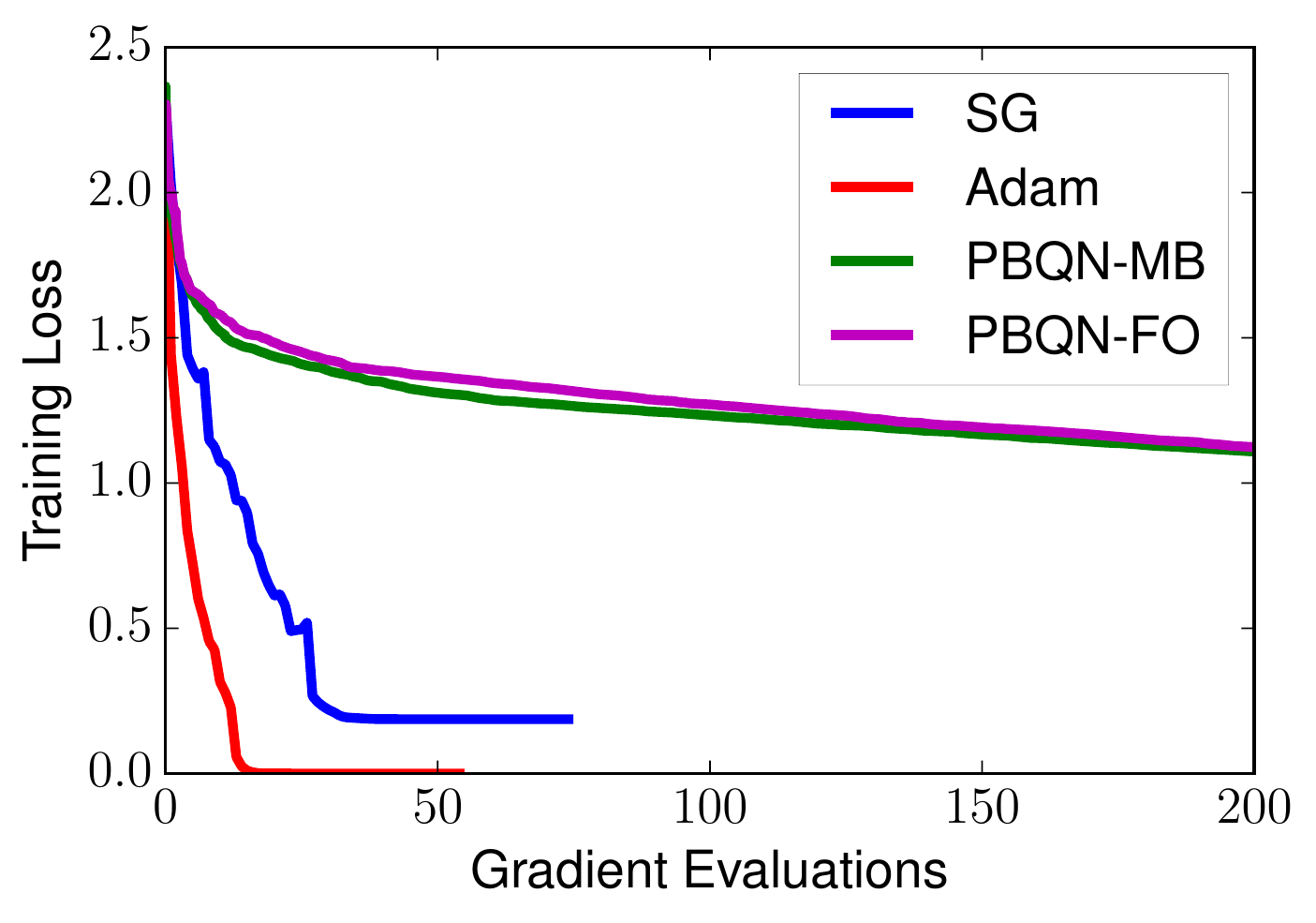}
\includegraphics[width=0.33\linewidth]{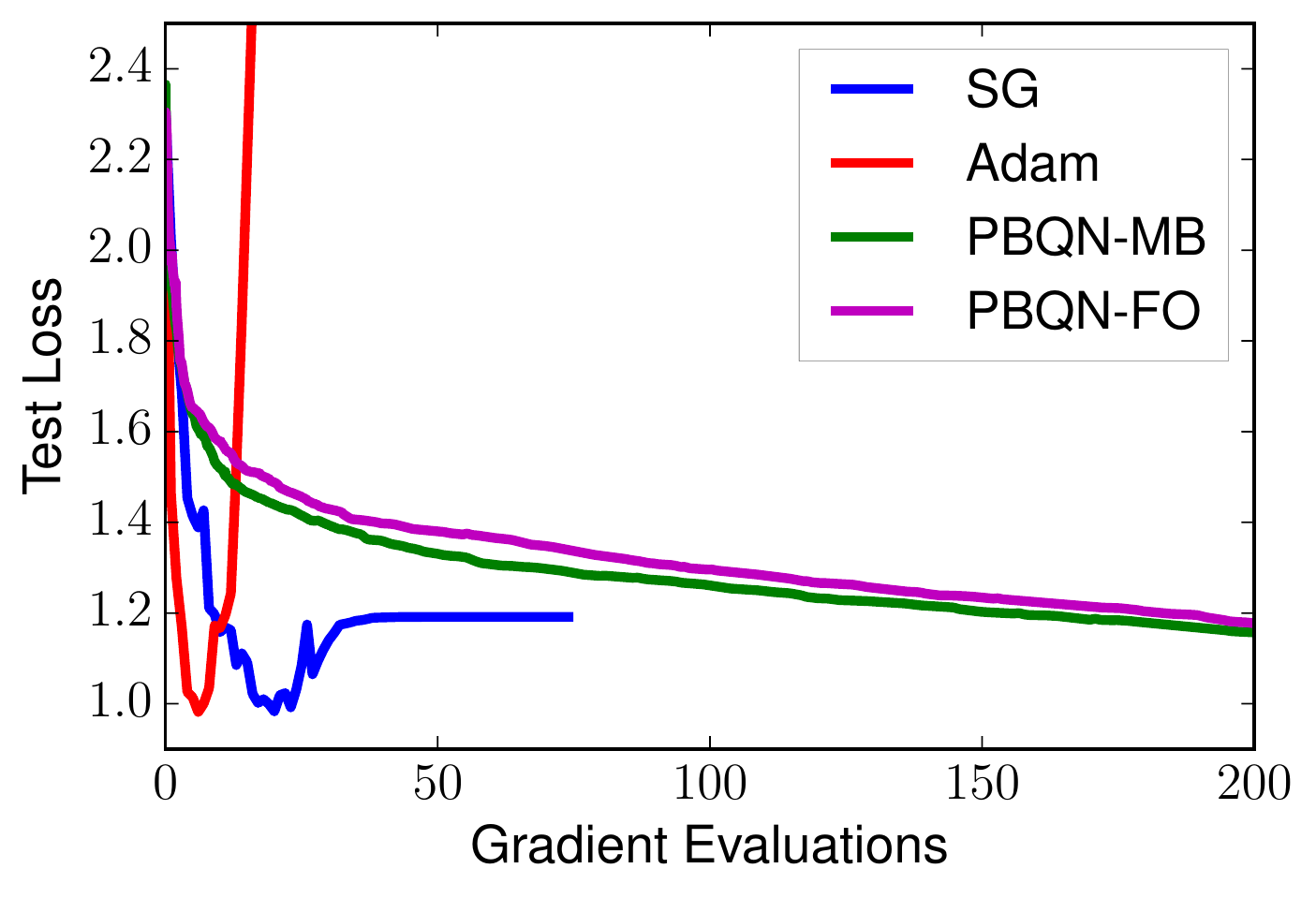}
\includegraphics[width=0.33\linewidth]{cifar10_resnet18_props_test_acc_theta_2.pdf}
\includegraphics[width=0.33\linewidth]{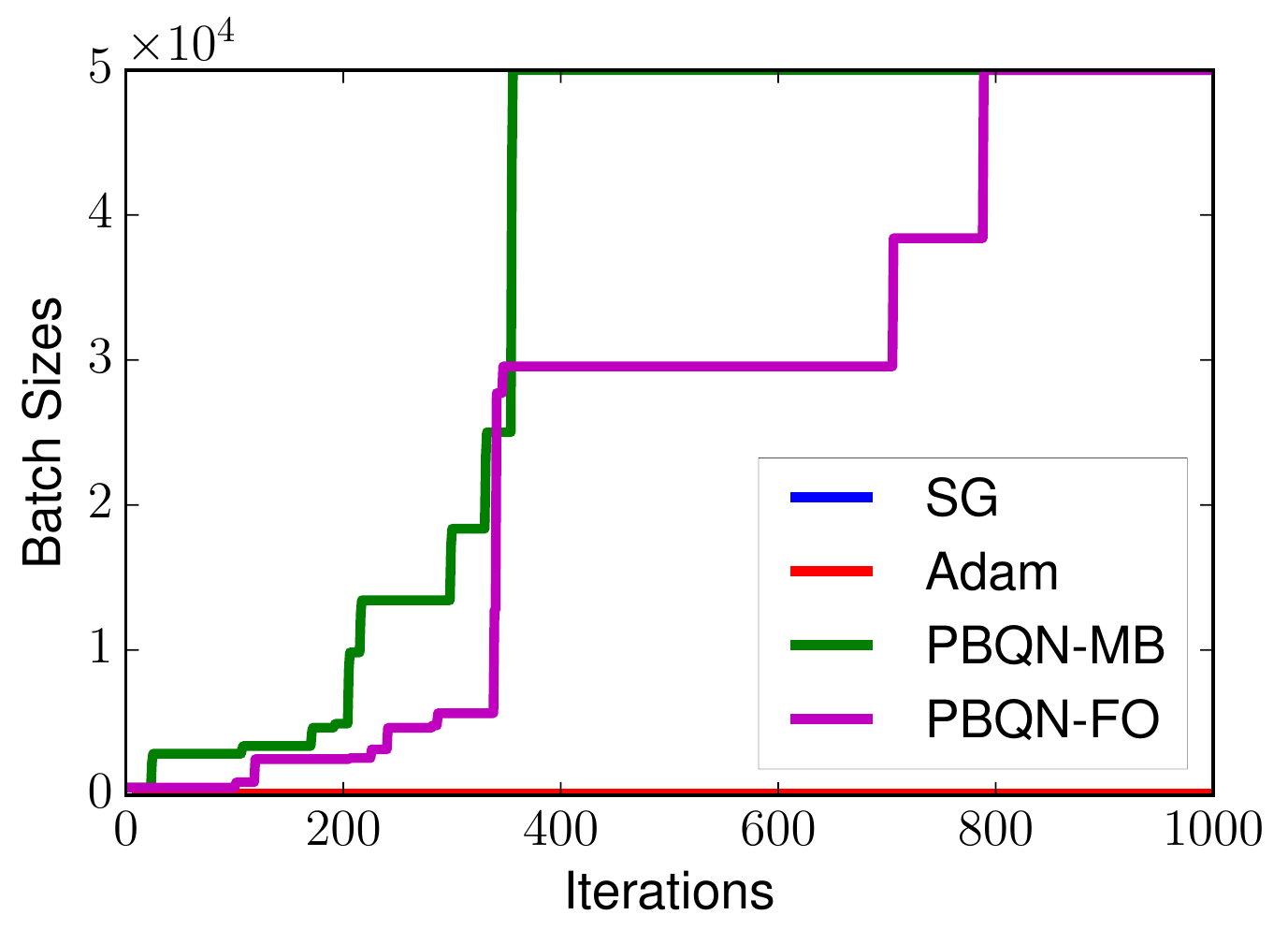}
\includegraphics[width=0.33\linewidth]{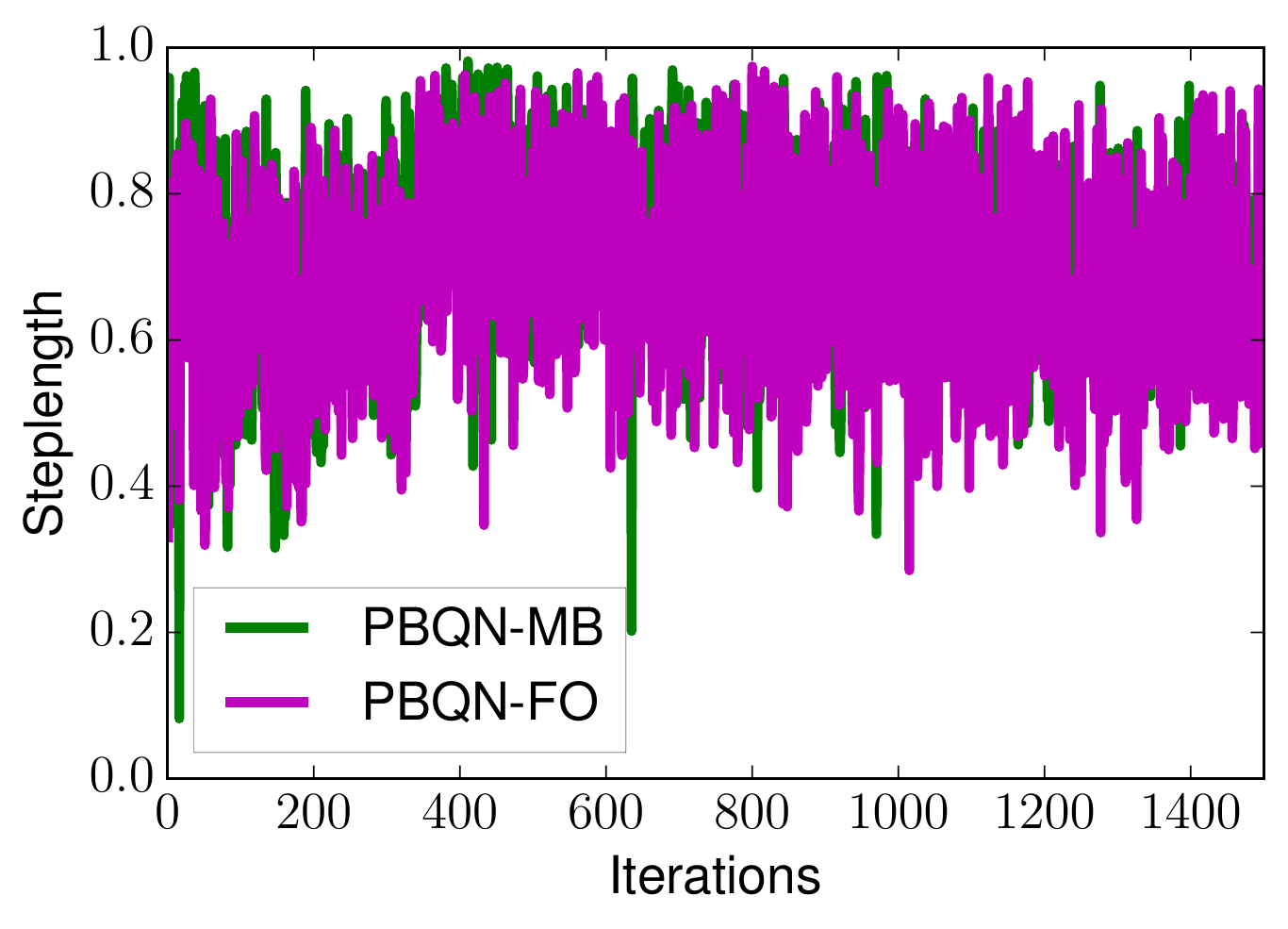}
\caption{\textbf{CIFAR-10 ResNet18 $(\mathcal{R})$:} Performance of the progressive batching L-BFGS methods, with multi-batch (MB) (25\% overlap) and full-overlap (FO) approaches, and the SG and Adam methods. The best results for L-BFGS are achieved with $\theta = 2$.}
\label{exp:cifar10_resnet18}
\end{centering}
\end{figure*}

	\newpage
	\section{Performance Model}
\label{sec:performance}

The use of increasing batch sizes in the PBQN algorithm yields a larger effective batch size than the SG method, allowing PBQN to scale to a larger number of nodes than currently permissible even with large-batch training \cite{goyal2017accurate}. With improved scalability and richer gradient information, we expect reduction in training time. To demonstrate the potential to reduce training time of a parallelized implementation of PBQN, we extend the idealized performance model from \cite{keskar2016large} to the PBQN algorithm. For PBQN to be competitive, it must achieve the following: (i) the quality of its solution should match or improve SG's solution (as shown in Table 1 of the main paper); (ii) it should utilize a larger effective batch size; and (iii) it  should converge to the solution in a lower number of iterations. We provide an initial analysis for this by establishing the analytic requirements for improved training time; we leave discussion on implementation details, memory requirements, and large-scale experiments for future work.

Let the effective batch size for PBQN and conventional SG batch size be denoted as $\widehat{B_L}$ and $B_S$, respectively. From Algorithm \ref{alg: pb_L-BFGS}, we observe that the PBQN iteration involves extra computation in addition to the gradient computation as in SG. The additional steps are as follows: the L-BFGS two-loop recursion, which includes several operations over the stored curvature pairs and network parameters (Algorithm \ref{alg: pb_L-BFGS}:6); the stochastic line search for identifying the steplength (Algorithm \ref{alg: pb_L-BFGS}:7-16); and curvature pair updating (Algorithm \ref{alg: pb_L-BFGS}:18-21). However, most of these supplemental operations are performed on the weights of the network, which is orders of magnitude lower than computing the gradient. The two-loop recursion performs $O(10)$ operations over the network parameters and curvature pairs. The cost for variance estimation is negligible since we may use a fixed number of samples throughout the run for its computation which can be parallelized while avoiding becoming a serial bottleneck. 

The only exception is the stochastic line search, which requires additional forward propagations over the model for different sets of network parameters. However, this happens only when the step-length is not accepted, which happens infrequently in practice. We make the pessimistic assumption of an addition forward propagation every iteration, amounting to an additional $\frac{1}{3}$ the cost of the gradient computation (forward propagation, back propagation with respect to activations and weights). Hence, the ratio of cost-per-iteration for PBQN $C_L$ to SG's cost-per-iteration $C_S$ is $\frac{4}{3}$. Let $I_S$ and $I_L$ be the number of iterations that it takes SG and PBQN, respectively, to reach similar test accuracy. The target number of nodes to be used for training is $N$, such that $N < \widehat{B_L}$. For $N$ nodes, the parallel efficiency of SG is assumed to be $P_e(N)$ and we assume that for the target node count, there is no drop in parallel efficiency for PBQN due to the large effective batch size. 

For a lower training time with the PBQN method, the following relation should hold:
\begin{equation}
\label{eqn:perfModel1}
I_L C_L \frac{\widehat{B_L}}{N} < I_S C_S \frac{B_S}{N P_e(N)}.
\end{equation}
In terms of iterations, we can rewrite this as
\begin{equation}
\label{eqn:perfModel2}
\frac{I_L}{I_S} < \frac{C_S}{C_L} \frac{B_S}{\widehat{B_L}} \frac{1}{P_e(N)}.
\end{equation}
Assuming target node count $N = B_S < \hat{B_L}$, the scaling efficiency of SG drops significantly due to the reduced work per single node, giving a parallel efficiency of $P_e(N) = 0.2$; see \cite{kurth2017deep,you2017scaling}. If we additionally assume that effective batch size for PBQN is $4\times$ larger, with SG large batch $\approx$ 8K and PBQN $\approx$ 32K as observed in our experiments (from Section \ref{sec:parallel}), this gives $\nicefrac{\widehat{B_L}}{B_S} = 4$. PBQN must converge with about the same number of iterations as SG in order to achieve lower training time. From Section \ref{sec:parallel}, the results show that PBQN converges in significantly fewer iterations than SG, hence establishing the  potential for lower training times. We refer the reader to \cite{das2016distributed} for a more detailed model and commentary on the effect of batch size on performance.

	
	
	
\end{document}